\documentclass{article}

\title{The free boundary Schur process and applications I}
\author{Dan Betea\thanks{IRIF, CNRS et Universit\'e Paris Diderot, Case 7014, F--75205 Paris Cedex 13, \texttt{dan.betea@gmail.com}} \and
  J\'er\'emie Bouttier\thanks{Institut de Physique Th\'eorique, Universit\'e Paris-Saclay, CEA, CNRS, F-91191 Gif-sur-Yvette, \texttt{jeremie.bouttier@ipht.fr} \protect\\
    \indent \ \,
    Univ Lyon, Ens de Lyon, Univ Claude Bernard, CNRS, Laboratoire de Physique, F-69342 Lyon, France} \and
  Peter Nejjar\thanks{IST Austria, 3400 Klosterneuburg, Austria, \texttt{peter.nejjar@ist.ac.at}. Supported by ERC Advanced  Grant  No. 338804 \protect\\
    \indent \ \, and ERC Starting Grant No.
 716117} \and
  Mirjana Vuleti\'{c}\thanks{Department of Mathematics, University of Massachusetts Boston, Boston, MA 02125, USA, \texttt{mirjana.vuletic@umb.edu}}}

\usepackage{graphicx, graphics, amsmath, amsthm, amssymb, xcolor,mathtools,bbm}
\usepackage{hyperref}
\usepackage{tikz}
\usetikzlibrary{calc,decorations.pathmorphing,decorations.markings, decorations.pathreplacing,patterns,shapes,arrows}
\usepackage[margin=1.0in]{geometry}
\usepackage{subcaption, sidecap}
\usepackage{stmaryrd}
\newtheorem{thm}{Theorem}
\newtheorem{prop}[thm]{Proposition}
\newtheorem{cor}[thm]{Corollary}
\newtheorem{lem}[thm]{Lemma}

\theoremstyle{definition}

\theoremstyle{remark}
\newtheorem{rem}[thm]{Remark}
\numberwithin{equation}{section}
\numberwithin{thm}{section}
\allowdisplaybreaks

\DeclareMathAlphabet{\mathpzc}{OT1}{pzc}{m}{it}

\newcommand{\R}{\mathbb{R}}
\newcommand{\Z}{\mathbb{Z}}
\newcommand{\N}{\mathbb{N}}
\newcommand{\C}{\mathbb{C}}
\newcommand{\Gami}{\Gamma_{-}}
\newcommand{\Gapl}{\Gamma_{+}}
\newcommand{\Gatpl}{\Gamma'_{+}}
\newcommand{\Gatmi}{\Gamma'_{-}}

\newcommand{\es}{\emptyset}
\DeclareMathOperator*{\pf}{pf}
\newcommand{\Par}{\mathcal{P}}
\newcommand{\F}{\mathcal{F}}
\newcommand{\s}{\mathcal{S}}
\DeclareMathOperator{\Prob}{Prob}
\DeclareMathOperator{\Ad}{Ad}
\newcommand{\GOE}{\mathrm{GOE}}
\newcommand{\GUE}{\mathrm{GUE}}
\newcommand{\GSE}{\mathrm{GSE}}

\newcommand{\bra}[1]{\langle #1 |}
\newcommand{\ket}[1]{| #1 \rangle}
\newcommand{\vv}{| 0 \rangle}
\newcommand{\vcv}{\langle 0 |}
\newcommand{\vfv}{| \underline{v} \rangle}
\newcommand{\ufcv}{\langle \underline{u} |}
\newcommand{\uv}{\langle \underline{u} | \underline{v} \rangle}
\newcommand{\vtfv}{| \underline{v,t} \rangle}
\newcommand{\utfcv}{\langle \underline{u,t} |}
\newcommand{\utfvtf}{\langle \underline{u,t} | \underline{v,t} \rangle}

\newcommand{\oneerv}{| \underline{1}_{ec} \rangle}

\newcommand{\im }{\mathrm{i}}  
\newcommand{\sgn}{\mathrm{sgn}}
\newcommand{\dx }{\mathrm{d}}

\newcommand{\dilog}{\mathrm{Li}_2}
\newcommand{\di}{\mathrm{Li}_2}
\newcommand{\x}{\mathpzc{x}}
\newcommand{\y}{\mathpzc{y}}
\let\a\undefined
\newcommand{\a}{\mathpzc{a}}
\newcommand{\X}{\mathpzc{X}}
\newcommand{\Y}{\mathpzc{Y}}
\let\A\undefined
\newcommand{\A}{\mathpzc{A}}
\newcommand{\h}{\mathpzc{h}}
\let\i\undefined
\let\k\undefined
\newcommand{\i}{\mathsf{i}}   
\newcommand{\k}{\mathsf{k}}

\def\fs{\footnotesize}

\begin{document}

\maketitle

\begin{abstract}
  We investigate the free boundary Schur process, a variant of the
  Schur process introduced by Okounkov and Reshetikhin, where we allow
  the first and the last partitions to be arbitrary (instead of empty
  in the original setting).  The pfaffian Schur process, previously
  studied by several authors, is recovered when just one of the
  boundary partitions is left free.  We compute the correlation
  functions of the process in all generality via the free fermion
  formalism, which we extend with the thorough treatment of ``free
  boundary states''.  For the case of one free boundary, our approach
  yields a new proof that the process is pfaffian. For the case of two
  free boundaries, we find that the process is not pfaffian, but a
  closely related process is.  We also study three different
  applications of the Schur process with one free boundary:
  fluctuations of symmetrized last passage percolation models, limit
  shapes and processes for symmetric plane partitions, and for plane
  overpartitions.
\end{abstract}

\section{Introduction} \label{sec:intro}

In this paper we introduce and study the free boundary Schur process,
a random sequence of partitions which we now define. Recall that an
(integer) \emph{partition} $\lambda$ is a nonincreasing sequence of
integers $\lambda_1 \geq \lambda_2 \geq \cdots$ which vanishes
eventually. Its \emph{size} is $|\lambda|:=\sum_{i \geq 1} \lambda_i$.
For two partitions $\lambda,\mu$ such that $\lambda \supset \mu$
(i.e.\ $\lambda_i \geq \mu_i$ for all $i$), let $s_{\lambda/\mu}$ be
the skew Schur function of shape $\lambda/\mu$ --- see the beginning of
Section~\ref{sec:fbcorr} for a summary of the relevant notions. Let us
fix a nonnegative integer $N$, two nonnegative real numbers $u$ and
$v$, and two families $(\rho_k^+)_{1\leq k \leq N}$ and
$(\rho_k^-)_{1\leq k \leq N}$ of specializations (which we can think
of as collections of variables). To a sequence of partitions of the
form
\begin{equation}
  \label{eq:seq}
  \mu^{(0)} \subset \lambda^{(1)} \supset \mu^{(1)} \subset \cdots
  \supset \mu^{(N-1)} \subset \lambda^{(N)} \supset \mu^{(N)}
\end{equation}
we assign a weight
\begin{equation}
  \label{eq:fbschur}
  \mathcal{W}(\vec{\lambda},\vec{\mu}) := u^{|\mu^{(0)}|} v^{|\mu^{(N)}|}
  \prod_{k=1}^N \left(s_{\lambda^{(k)}/\mu^{(k-1)}}\left(\rho_k^+\right)
    s_{\lambda^{(k)}/\mu^{(k)}}\left(\rho_k^-\right) \right).
\end{equation}
The \emph{partition function}
$Z \equiv Z(u,v,\rho_1^+,\rho_1^-,\ldots,\rho_N^+,\rho_N^-)$ is the
sum of weights of all sequences of the form \eqref{eq:seq}.  Under
certain assumptions on the parameters
$u,v,\rho^\pm_1,\ldots,\rho^\pm_N$ to be detailed in
Section~\ref{sec:fbcorr}, the partition function is finite, and
$\mathcal{W}/Z$ defines a probability distribution which is the
\emph{free boundary Schur process}.

For $u=v=0$, we recover the original Schur process of Okounkov and
Reshetikhin \cite{or}, which is such that the boundary partitions
$\mu^{(0)}$ and $\mu^{(N)}$ are both equal to the empty (zero) partition
$\emptyset$. For $u>0$ and $v=0$, only $\mu^{(N)}$ is constrained to
be zero, and we recover the so-called pfaffian Schur process \cite{br}
up to the inessential change that, in this reference, $\mu^{(0)}$ is
assumed to be the conjugate of an even partition --- see Remark~\ref{rem:oc}
below. Of course, the case $u=0$ and $v>0$ is equivalent by
symmetry. The new situation considered in this paper is when $uv>0$,
i.e.\ when both boundaries are free. Note that the constant sequence
equal to $\lambda$ has weight $(uv)^{|\lambda|}$; therefore it is
necessary to have $uv<1$ for the partition function to be finite. Let
us mention that, if we condition on having $\mu^{(0)}=\mu^{(N)}$, then
we recover Borodin's periodic Schur process \cite{B2007cyl}.
Conversely, the free boundary Schur process of length $N$ can be seen
as a symmetrized version of the periodic Schur process of length $2N$.

The Schur process may be viewed as a simple point process on
$\mathbb{Z}^2$ --- see Section~\ref{sec:fbcorr} below. As
such, a natural question is to characterize the nature of this
process. Okounkov and Reshetikhin showed that the original Schur
process is determinantal \cite{or}, while Borodin and Rains proved that the
pfaffian Schur process is, well, pfaffian \cite{br} --- see Appendix \ref{sec:pfaffian} for the definition of pfaffian point processes. In this paper, we shall
see that the free boundary Schur process is not pfaffian in general,
but a closely related point process is, and we will explicitly compute its
correlation kernel. This situation is reminiscent of the periodic
Schur process, which becomes determinantal only after a certain
``shift-mixing'' \cite{B2007cyl}.

\paragraph{Context and motivations.}
The problems we consider in this paper are part of an active area of
research dubbed ``integrable probability'' --- see for instance the
exposition in \cite{bg}. A first major result in the area was the
resolution of Ulam's problem by Baik, Deift and Johansson \cite{bdj}
who have shown that the longest increasing subsequence of a random
permutation exhibits Tracy--Widom GUE fluctuations around its mean,
and thus behaves like the largest eigenvalue of a large Gaussian
Unitary Ensemble random matrix \cite{tw}. By the Robinson--Schensted
correspondence, Ulam's problem is closely related to the so-called
Plancherel measure on the set of partitions, and Okounkov \cite{oko}
realized that this measure is a particular instance of a \emph{Schur
  measure}, whose determinantal correlations can be computed
explicitly using the infinite wedge space (or free fermion)
formalism. Asymptotics then reduces to simple saddle point
analysis. Around the same time, a discrete version of the Plancherel
measure, that of last passage percolation (LPP) in a rectangle with
independent geometric weights, has been analyzed by Johansson
\cite{joh2} using Schur measures and orthogonal polynomial
techniques. Using asymptotics of Meixner polynomials, Johansson showed
that this model belongs to the same universality class and that, in
particular, the last passage time also fluctuates according to the
Tracy--Widom GUE law.

The story continues with a series of papers by Baik and Rains
\cite{rai, br99, br1, br2} studying longest increasing subsequences in
random permutations subject to certain symmetry constraints (e.g.,\
involutions). Upon poissonization the corresponding processes are
pfaffian instead of determinantal, and new distributions like the
Tracy--Widom GOE and GSE laws appear as fluctuations.

In parallel, Okounkov and Reshetikhin \cite{or} introduced a
time-dependent generalization of the Schur measure called the
\emph{Schur process}, using again the infinite wedge space formalism
to prove its determinantal nature, and applied their result to
analyze large plane partitions. As mentioned above, in the original
setting, the process is constrained to start and end with the empty
partition. Borodin and Rains \cite{br} developed another approach to
the Schur process via the Eynard--Mehta theorem; they treated
similarly the \emph{pfaffian Schur process}, which appeared implicitly
in an earlier work of Sasamoto and Imamura \cite{si}, and corresponds
in our language to having one free boundary and one empty boundary
(alternatively it can be viewed as a symmetrized Schur process upon
interpreting the free boundary as a reflection axis). In a different
direction, Borodin~\cite{B2007cyl} considered the Schur process with
periodic boundary conditions.

In this paper, we explore the ``missing'' type of boundary conditions,
namely that of two free boundaries. Our main technical tool will be,
as in \cite{or}, the infinite wedge space/free fermion formalism. Free
boundaries are represented in this formalism as \emph{free boundary
  states}, which were first introduced in \cite{bcc} in order to
compute the partition function of free boundary steep tilings (an
instance of free boundary Schur process to appear in Section \ref{sec:overpartitions}). Here, we proceed to the next level of computing correlation functions, which requires understanding
the interplay between free boundary states and fermionic
operators. The determinantal nature of the original Schur process with
empty boundary conditions results from Wick's theorem for free
fermions. As we shall see, the adaptation of this theorem for free
boundaries is not completely straightforward, and involves extended
free boundary states which are not eigenvectors of the charge
operator. A consequence of this is that the free boundary Schur
process is neither determinantal nor pfaffian in general, but becomes
pfaffian after we perform a certain random vertical shift of the point
configuration, that translates in the point process language the
``charge mixing'' occurring in extended free boundary states. This
phenomenon has some similarities with Borodin's shift-mixing for the
periodic Schur process~\cite{B2007cyl}, but the fermionic picture is
rather different: as explained in~\cite{BB18}, for
periodic boundary conditions, Borodin's shift-mixing can be interpreted as the
passage to the grand canonical ensemble, needed to apply Wick's
theorem at finite temperature. In the case of a single free boundary,
the shift goes away, and our approach yields a new derivation of the
correlations functions of the pfaffian Schur process, alternative to
that by Borodin and Rains~\cite{br} and the very recent one by
Ghosal~\cite{gho17} using Macdonald difference operators.

Among other recent developments related to the pfaffian Schur process,
let us mention the work by Baik, Barraquand, Corwin and Suidan
studying its applications to LPP in half-space~\cite{BBCS17} and
facilitated TASEP~\cite{bbcs17bis}, and the work of Barraquand, Borodin,
Corwin and Wheeler~\cite{bbcw17} introducing its Macdonald analogue.
Here, we further investigate applications of the pfaffian Schur
process by considering symmetric LPP thus complementing the results of
\cite{br1,br2,BBCS17}, as well as symmetric plane partitions and
plane overpartitions --- two models which can be rephrased in terms of
lozenge and domino tilings, respectively. The fact that dimer models
with free boundaries are related to pfaffians is not surprising.
This was already observed for instance in~\cite{ste,CK11} via
nonintersecting lattice paths. See also~\cite{DFR12,Panova15} for
other limit shape results on tilings with free
boundaries. Applications of the Schur process with two free boundaries
will be investigated in a subsequent publication.

\paragraph{Outline.}

The paper is organized as follows: in Section~\ref{sec:mainres} we
list the main results of the paper, only introducing the basic
concepts needed for the statements. It is divided in two parts:
Section~\ref{sec:fbcorr} leads to the two fundamental
Theorems~\ref{thm:onebound} and~\ref{thm:main_thm} stating that
certain point processes associated with the free boundary Schur
process have pfaffian correlations, while
Section~\ref{sec:applications} deals with listing the applications we
draw from the first. Section~\ref{sec:fermions} is devoted to the
proof of our two fundamental theorems via the machinery of free
fermions. We also obtain in Theorem~\ref{thm:rhonfold} an expression
for the general multipoint correlation
functions. Sections~\ref{sec:slpp}, \ref{sec:symm_pp} and
\ref{sec:overpartitions} deal with asymptotic applications of
Theorem~\ref{thm:onebound} to models of symmetric last passage
percolation, symmetric plane partitions and plane overpartitions
respectively. Section~\ref{sec:conc} gathers some concluding remarks
and perspectives. We list, in the Appendices, odds and ends we deemed
too cumbersome to put in the main text.

\paragraph{Note.} This paper is a slightly abridged version of the
preprint \cite{BBNV17}. An extended abstract was presented at
FPSAC2017 \cite{MR3678626}.

\section{Main results}
\label{sec:mainres}

\subsection{Correlation functions of the free boundary Schur process}
\label{sec:fbcorr}

\paragraph{Preliminaries on symmetric functions.} We start with some
definitions and notations that are needed to state our results in
compact form. We refer to \cite[Chapter 1]{mac} or \cite[Chapter
7]{sta} for general background. Let $\mathsf{Sym}$ be the algebra of
symmetric functions, and let $h_n$ (resp.\ $p_n$) be the complete
homogeneous (resp.\ power sum) symmetric function of degree $n$.  For
two partitions $\lambda \supset \mu$, the skew Schur function
$s_{\lambda / \mu}$ is given by
$s_{\lambda / \mu}:=\det_{1 \leq i,j \leq \ell(\lambda)}
h_{\lambda_i-i+\mu_j-j}$ where $\ell(\lambda):=\max\{i:\lambda_i>0\}$
is the length of $\lambda$, and the ordinary Schur function
$s_\lambda$ is obtained by taking $\mu=\emptyset$.

A \emph{specialization} $\rho$ is an algebra homomorphism from
$\mathsf{Sym}$ to the field $\C$ of complex numbers.  It is uniquely
determined by its values on the $h_n$'s (or equivalently the $p_n$'s),
hence by the generating function
\begin{equation}
  H(\rho;t) := \sum_{n \geq 0} h_n(\rho) t^n = \exp \left(
    \sum_{n \geq 1} \frac{p_n(\rho) t^n}{n} \right).
\end{equation}
As is customary, for a symmetric function $f \in \mathsf{Sym}$, we
will write $f(\rho)$ in lieu of $\rho(f)$. For $\rho,\rho'$ two
specializations and $s$ a complex number, we denote by
$\rho \cup \rho'$ and $s\rho$ the specializations defined by
\begin{equation}
  H(\rho \cup \rho';t) := H(\rho;t) H(\rho';t), \qquad
  H(s\rho;t) := H(\rho;st)
\end{equation}
or equivalently $p_n(\rho \cup \rho') := p_n(\rho)+p_n(\rho')$,
$p_n(s \rho) := s^n p_n(\rho)$ for $n \geq 1$.
Denoting by $\Par$ the set of all partitions, we also define the
(possibly infinite) quantities
\begin{equation}
  \label{eq:hdef}
  \begin{split}
    H(\rho;\rho') &:= \sum_{\lambda \in \Par} s_\lambda(\rho) s_\lambda(\rho')
    = \exp \left( \sum_{n \geq 1} \frac{p_n(\rho) p_n(\rho')}{n} \right), \\
    \tilde{H}(\rho) &:= \sum_{\lambda \in \Par} s_\lambda(\rho)
    = \exp \left( \sum_{n \geq 1} \left (\frac{p_{2n-1}(\rho)}{2n-1} +
        \frac{p_n(\rho)^2}{2n} \right) \right).
  \end{split}
\end{equation}
The definitions in terms of the $s_\lambda$'s or the $p_n$'s are
equivalent by virtue of the so-called Cauchy and Littlewood identities
\cite[Theorem~7.12.1 and Corollary~7.13.8]{sta}.
Note that the notation $H(\cdot;\cdot)$ is consistent: for $\rho'$ the
specialization in the single variable $t$, we have
$H(\rho;\rho')=H(\rho;t)$.  We also have the relations
\begin{align}
  H(\rho; \rho' \cup \rho'') = H(\rho; \rho') H(\rho; \rho''), \qquad 
  \tilde{H} (\rho \cup \rho') = \tilde{H}(\rho) \tilde{H}(\rho') H(\rho; \rho').
\end{align}

A specialization $\rho$ is said \emph{nonnegative} if
$s_{\lambda/\mu}(\rho)$ is a nonnegative real number for any
$\lambda,\mu$. In view of \eqref{eq:fbschur}, all specializations
$\rho^\pm_1,\ldots,\rho^\pm_N$ should be nonnegative in order for the
weight $\mathcal{W}(\vec{\lambda},\vec{\mu})$ to be nonnegative.  A
necessary and sufficient condition \cite{tho, ais} for $\rho$ to be
nonnegative is that its generating function be of the form
\begin{equation}
  \label{eq:rhoposparam}
  H(\rho;t) = e^{\gamma t} \prod_{i \geq 1} \frac{1+\beta_i t}{1-\alpha_i t}
\end{equation}
where
$\gamma,\alpha_1,\beta_1,\alpha_2,\beta_2,\ldots$ form a summable
collection of nonnegative real numbers (in particular, when
$\gamma=\beta_1=\beta_2=\cdots=0$, we recover the specialization in
the variables $\alpha_1,\alpha_2,\ldots$).

\paragraph{Partition function.}

The computation of the partition function of the general free boundary
Schur process was essentially carried out in \cite[Section~5.3]{bcc}. In
our current notation it is given as follows.

\begin{prop}
  \label{prop:fbz}
  The partition function of the free boundary Schur process reads
  \begin{equation}
    \label{eq:fbz}
    Z = \prod_{1 \leq k \leq \ell \leq N} H(\rho_k^+;\rho_\ell^-)
    \prod_{n \geq 1} \frac{\tilde{H}(u^{n-1}v^{n}\rho^+)
       \tilde{H}(u^{n}v^{n-1}\rho^-) H(u^{2n}\rho^+;v^{2n}\rho^-)}{1-u^n v^n}
  \end{equation}
  where
  \begin{equation}
    \label{eq:rhopmdef}
    \rho^\pm := \rho^\pm_1 \cup \rho^\pm_2 \cup \cdots \cup \rho^\pm_N.
  \end{equation}
\end{prop}

  For $u=v=0$, the second product in the right hand side of \eqref{eq:fbz} reduces
  to $1$ and we recover the partition function of the original Schur
  process. For $u=1$ and $v=0$, the only nontrivial factor in the
  second product is $\tilde{H}(\rho^-)$ and we recover the partition
  function of the pfaffian Schur process \cite[Proposition 3.2]{br},
  up to slightly different conventions.

\paragraph{Simplifying assumptions.} To ease the forthcoming
discussion, we shall assume from now on that $u,v \leq 1$ and $uv<1$,
that the specializations $\rho_k^\pm$ are nonnegative and that the
series $H(\rho_k;\cdot)$ are analytic and nonzero in some disk of
radius $R>1$ --- see the preprint \cite{BBNV17} for a discussion of
more general assumptions. Our assumptions imply that $Z$ is finite and
that the free boundary Schur process is a probability distribution.

\paragraph{Point process.}

Following \cite{or}, we define the \emph{point configuration}
associated with a sample $(\vec{\lambda},\vec{\mu})$ of the free
boundary Schur process as
\begin{equation}
  \label{eq:pointconfig}
  \mathfrak{S}(\vec{\lambda}) := \left\{
    \left(i,\lambda^{(i)}_j-j+\frac{1}{2}\right), 1 \leq i \leq N , j \geq 1
  \right\} \subset \Z \times \Z'
\end{equation}
where $\Z':=\Z+1/2$ (having half-integer ordinates makes formulas
slightly more symmetric). This is a simple point process on
$\Z \times \Z'$.
Note that there is no loss of generality in considering only the
partitions $\lambda^{(1)},\ldots,\lambda^{(N)}$ in the definition of
the point configuration, and this makes the forthcoming formulas more
compact. One may study the statistics of the $\mu$'s by considering an
auxiliary Schur process with increased length and zero specializations
inserted where appropriate (as
$s_{\lambda/\mu}(0)=\mathbbm{1}_{\lambda=\mu}$).

\paragraph{Correlations for one free boundary.}

Let us first discuss the previously known case of the pfaffian Schur
process \cite{br}, obtained for $u=0$. By homogeneity of the Schur
functions, we may assume $v=1$ without loss of generality.

\begin{thm} \label{thm:onebound} For $u=0$ and $v=1$,
  $\mathfrak{S}(\vec{\lambda})$ is a pfaffian point process (see
  Appendix~\ref{sec:pfaffian} for the definition) whose
  correlation kernel entries are given by
  \begin{equation} \label{eq:Kint}
    \begin{split}
      K_{1,1}(i, k; i', k') &= \frac{1}{(2\im \pi)^2}
      \oint_{|z|=r} \frac{\dx z}{z^{k+1}} \oint_{|w|=r'} \frac{\dx w}{w^{k'+1}}
      F(i, z) F(i', w) \kappa_{1,1}(z,w), \\
      K_{1,2}(i,k;i',k') &= - K_{2,1}(i',k';i,k)
       = \frac{1}{(2\im \pi)^2}
      \oint_{|z|=r} \frac{\dx z}{z^{k+1}} \oint_{|w|=r'} \frac{\dx w}{w^{-k'+1}}
      \frac{F(i, z)}{F(i', w)} \kappa_{1,2}(z,w), \\
      K_{2,2}(i,k;i',k') &= \frac{1}{(2\im \pi)^2}
      \oint_{|z|=r} \frac{\dx z}{z^{-k+1}} \oint_{|w|=r'} \frac{\dx w}{w^{-k'+1}}
      \frac{1}{F(i, z) F(i', w)} \kappa_{2,2}(z,w)
    \end{split}
  \end{equation}
  where the radii $r,r'$ are such that $1 < r' < r < R$ if $i \leq i'$
  and $1 < r < r' < R$ otherwise, and where
  \begin{equation} \label{eq:kappasimp}
    \begin{aligned}
      &F(i,z) = \frac{\prod_{1 \leq \ell \leq i} H(\rho_\ell^+;z)}{
        H(\rho^+;z^{-1}) \prod_{i \leq \ell \leq N} H(\rho_\ell^-;z^{-1})},\qquad &
      &\kappa_{1,1}(z,w) = \frac{(z-w)\sqrt{zw}}{(z+1)(w+1)(zw-1)}, \\
      &\kappa_{1,2}(z,w) = \frac{(zw-1)\sqrt{zw}}{(z+1)(w-1)(z-w)}, &
      &\kappa_{2,2}(z,w) = \frac{(z-w)\sqrt{zw}}{(z-1)(w-1)(zw-1)}.
    \end{aligned}
  \end{equation}
\end{thm}

\begin{rem} \label{rem:contint} The double contour integrals
  in~\eqref{eq:Kint} correspond to extracting coefficients in certain
  bivariate Laurent series. Note that only integer powers are involved
  since the $\sqrt{zw}$ in $\kappa$'s is compensated by $k,k'$ being
  half-integers.  Intuitively speaking, in each factor
  $F(i,z)^{\pm 1}$, a $H(\cdot;z)$ should be thought as a series in
  $z$ and a $H(\cdot;z^{-1})$ as a series in $z^{-1}$ (and similarly
  for $w$), while the $\kappa(z,w)$ should be thought as bivariate
  series in $z^{-1}$ and $w^{-1}$. In $\kappa_{1,2}(z,w)$, the pole
  $1/(z-w)$ should be expanded as $\sum_{k \geq 0} w^k/z^{k+1}$ for
  $i \leq i'$, and as $-\sum_{k < 0} w^k/z^{k+1}$ otherwise.
\end{rem}

\begin{rem} \label{rem:oc} Our expressions do not quite match those of
  \cite[Theorem~3.3]{br} mainly because Borodin and Rains impose that
  the ``free boundary'' is a partition whose conjugate has even
  parts. This change is inessential, and it is possible to go from one
  convention to another by a simple change of the boundary
  specialization. Actually, one can \emph{interpolate} between the two
  conventions by multiplying the weight \eqref{eq:fbschur} by an extra
  factor $\alpha^{oc}$, where $oc$ denotes the number of odd columns
  of the Young diagram of $\mu^{(N)}$ and where $\alpha$ is a
  nonnegative parameter smaller than $R$. With this extra weighting,
  Theorem~\ref{thm:onebound} still holds provided that we take
  $r,r'>\alpha$ and modify the $\kappa$'s into
  \begin{equation}
    \label{eq:kapoc}
    \begin{aligned}
      \kappa_{1,1}(z,w) &= \frac{(z-\alpha) (w-\alpha)(z-w) \sqrt{zw} }{(z^2-1)
        (w^2-1) (zw-1)}, \qquad & \kappa_{1,2}(z,w) &= \frac{(zw-1)
        (z-\alpha)\sqrt{zw}}{(z^2-1) (w-\alpha)(z-w)}, \\ \kappa_{2,2}(z,w)
      &= \frac{(z-w)\sqrt{zw}}{(z-\alpha)(w-\alpha)(zw-1)}. & &
    \end{aligned}
  \end{equation}
  For $\alpha=1$ we get back \eqref{eq:kappasimp}, while for
  $\alpha=0$ we recover \cite[Theorem~3.3]{br} up to a simple change
  of variables.  See Section~\ref{sec:variations} for the derivation.
\end{rem}

\paragraph{Correlations for two free boundaries.}

We now turn to the general case of two free boundaries.  Similarly to
the periodic Schur process studied in \cite{B2007cyl}, the random
point process $\mathfrak{S}(\vec{\lambda})$ is \emph{neither}
determinantal nor pfaffian in general, but a modification of it
is. More precisely, let us fix an auxiliary real parameter $t$, and
consider a $\Z$-valued random variable $D_t$ independent of the Schur
process, with law
\begin{equation}
  \Prob(D_t=d) = \frac{t^{2d} (uv)^{2d^2}}{\theta_3(t^2;(uv)^4)}.
\end{equation}
Here the normalization factor involves the Jacobi theta function
$\theta_3(z;q):=\sum_{n \in \Z} q^{n^2/2} z^n$ --- see
Appendix~\ref{sec:theta}.  We then consider the shifted point
configuration
\begin{equation}
  \mathfrak{S}_t(\vec{\lambda}):=\mathfrak{S}(\vec{\lambda})+(0,2D_t),
\end{equation}
that is to say we move all points of $\mathfrak{S}(\vec{\lambda})$
vertically by a same shift $2D_t$. Note that, in contrast with the
periodic Schur process \cite{B2007cyl,BB18}, we have to shift the point
configuration by an \emph{even} integer. As we shall see, the origin
of this shift in the free fermion formalism is rather different. 

\begin{thm} \label{thm:main_thm} The point process
  $\mathfrak{S}_t(\vec{\lambda})$ is pfaffian, and the entries of its
  correlation kernel still have the form \eqref{eq:Kint}, with the
  radii $r,r'$ now such that $\max(v,R^{-1}) < r,r' < \min(R,u^{-1})$,
  $r'<r$ if $i \leq i'$ and $r < r'$ otherwise, and with $F$ and
  $\kappa$ now given by
  \begin{equation}
    \label{eq:Fkap2b}
    \begin{split}
            F(i,z) &= \frac{\prod_{1 \leq \ell \leq i} H(\rho_\ell^+;z)}{
       \prod_{i \leq \ell \leq N} H(\rho_\ell^-;z^{-1})} \cdot 
     \prod_{n \geq 1} \frac{H(u^{2n} v^{2n-2} \rho^-;  z) H(u^{2n} v^{2n} \rho^+; z)}{
       H(u^{2n-2} v^{2n} \rho^+; z^{-1}) H(u^{2n} v^{2n} \rho^-; z^{-1})},\\
      \kappa_{1,1}(z,w) &= \frac{v^2 }{t z^{1/2} w^{3/2}} \cdot \frac{((uv)^2; (uv)^2)^2_{\infty}}  {(uz, uw, - \frac{v}{z}, - \frac{v}{w}; uv)_{\infty}} \cdot  \frac{\theta_{(uv)^2}(\frac{w}{z})}{\theta_{(uv)^2} (u^2 zw)} \cdot \frac{\theta_3 \left( ( \frac{t zw}{v^2})^2 ; (uv)^4 \right)}{\theta_3(t^2;(uv)^4)}, \\
      \kappa_{1,2}(z,w) &= \frac{w^{1/2}}{z^{1/2}} \cdot \frac{((uv)^2; (uv)^2)^2_{\infty}}{(uz, -uw,- \frac{v}{z}, \frac{v}{w}; uv)_{\infty}} \cdot \frac{\theta_{(uv)^2}(u^2 zw)}{\theta_{(uv)^2} (\frac{w}{z})} \cdot \frac{\theta_3 \left( ( \frac{t z}{w})^2 ; (uv)^4 \right)}{\theta_3(t^2;(uv)^4)}, \\
      \kappa_{2,2}(z,w) &= \frac{t v^2}{z^{1/2} w^{3/2}} \cdot \frac{((uv)^2; (uv)^2)^2_{\infty}}{(-uz, -uw,\frac{v}{z}, \frac{v}{w}; uv)_{\infty}} \cdot \frac{\theta_{(uv)^2}(\frac{w}{z})}{\theta_{(uv)^2} (u^2 zw)} \cdot \frac{\theta_3 \left( ( \frac{t v^2}{zw})^2 ; (uv)^4 \right)}{\theta_3(t^2;(uv)^4)}\\
    \end{split}
  \end{equation}
  where
  $(a_1,\ldots,a_m;q)_\infty:=\prod_{k=0}^\infty(1-a_1
  q^k)\cdots(1-a_m q^k)$
  is the infinite $q$-Pochhammer symbol with multiple arguments, and
  $\theta_q(z):=(z;q)_\infty(q/z;q)_\infty$ is the ``multiplicative''
  theta function --- see Appendix~\ref{sec:theta}. 
\end{thm}

Several remarks are now in order:
\begin{enumerate}
\item We recover of course Theorem~\ref{thm:onebound} for $u=0$ and
  $v=1$, as $D_t=0$ hence
  $\mathfrak{S}(\vec{\lambda})=\mathfrak{S}_t(\vec{\lambda})$.
\item Remark~\ref{rem:contint} still provides some ``intuition''
  regarding the choice of contours: they should encircle certain poles
  of the integrands and not others, in order to pick the appropriate
  expansions of $H(\cdot;z)$ and $H(\cdot;z^{-1})$.
  The main complication lies in the kernels $\kappa(z,w)$: they
  actually describe the free boundary Schur process of length $N=0$,
  for which $F \equiv 1$, and which is nothing but a single random
  partition drawn according to the $(uv)^{\mathrm{size}}$
  measure.
  See also~\cite[Corollary~2.6]{B2007cyl} for a related
  observation.

\item \label{rem:kappole} As in the case of one free boundary,
  $\kappa_{1,1}(z,w)$ and $\kappa_{2,2}(z,w)$ have a simple zero at
  $z=w$ while $\kappa_{1,2}(z,w)$ has a simple pole, due to the
  $\theta_{(uv)^2}(w/z)$ factor appearing in the numerator or
  denominator. Note that, because of the constraints on $r$ and $r'$,
  we cannot hit any other zero of the factors $\theta_{(uv)^2}(w/z)$
  and $\theta_{(uv)^2}(u^2 z w)$, hence no other pole of
  $\kappa(z,w)$.
    
\item The fact that we have an arbitrary parameter $t$ at our disposal
  allows in principle to return to the correlation functions for the unshifted
  point process $\mathfrak{S}(\vec{\lambda})$.
  Actually, it is possible to
  obtain an explicit expression for the $n$-point correlation functions
  of both $\mathfrak{S}_t(\vec{\lambda})$ and
  $\mathfrak{S}(\vec{\lambda})$ in the form of a $2n$-fold contour
  integral --- see Theorem~\ref{thm:rhonfold}.
\end{enumerate}

\subsection{Applications}
\label{sec:applications}

We now present some applications of Theorem~\ref{thm:onebound} to last
passage percolation, symmetric plane partitions and plane
overpartitions. We plan to present applications of
Theorem~\ref{thm:main_thm} in a subsequent paper.

In our applications, all $\rho_k^-$ are equal to the zero
specialization. The weight \eqref{eq:fbschur} is then nonzero only for
sequences \eqref{eq:seq} such that $\lambda^{(k)}=\mu^{(k)}$ for all
$k$, which can be seen more simply as \emph{ascending} sequences of
partitions
$\emptyset \subset \lambda^{(1)} \subset \cdots \subset
\lambda^{(N)}$.
Furthermore, each specialization $\rho_k^+$ will be either a
specialization in a single variable $x_k$ (i.e.\
$H(\rho_k^+;z)=(1-x_k z)^{-1}$) or its ``dual''
($H(\rho_k^+;z)=1+x_k z$). Recall that, for a single variable $x$, we
have
$s_{\lambda/\mu}(x)=x^{|\lambda|-|\mu|} \mathbbm{1}_{\lambda \succ
  \mu}$
where the notation $\lambda \succ \mu$ means that the skew shape
$\lambda/\mu$ is a \emph{horizontal strip} (i.e.\
$\lambda_1 \geq \mu_1 \geq \lambda_2 \geq \mu_2 \geq \cdots$).
Similarly, for the dual specialization $\bar{x}$, we have
$s_{\lambda/\mu}(\bar{x})=s_{\lambda'/\mu'}(x)=x^{|\lambda|-|\mu|}
\mathbbm{1}_{\lambda \succ' \mu}$
where the notation $\lambda \succ' \mu$ means that the skew shape
$\lambda/\mu$ is a \emph{vertical strip} (i.e.\
$\lambda_i-\mu_i \in \{0,1\}$ for all $i$).

The \emph{$H$-ascending Schur process} consists in taking only
specializations in single variables. In that case, we obtain a measure
over sequences of the form
\begin{equation}
  \es \prec \lambda^{(1)} \prec \lambda^{(2)} \prec \cdots \prec \lambda^{(N-1)} \prec \lambda^{(N)}.
\end{equation}
The \emph{$HV$-ascending Schur process} consists in taking
alternatively a specialization in a single variable or a dual variable, to
get a measure over sequences
\begin{equation}
  \es \prec \lambda^{(1)} \prec' \lambda^{(2)} \prec \cdots \prec \lambda^{(2M-1)} \prec' \lambda^{(2M)}, \qquad N=2M.
\end{equation}
In both cases, the unnormalized weight of a sequence will be
\begin{equation}
  x_1^{|\lambda^{(1)}|} x_2^{|\lambda^{(2)}|-|\lambda^{(1)}|} \cdots 
  x_N^{|\lambda^{(N)}|-|\lambda^{(N-1)}|},
\end{equation}
possibly with the extra weighting $\alpha^{oc}$ of
Remark~\ref{rem:oc}. For convenience, we state the following:
\begin{prop}
  \label{prop:H-HV}
  For the $H$- and $HV$-ascending Schur processes, the function $F$
  appearing in Theorem~\ref{thm:onebound} reads respectively
  \begin{equation}
    F_H(i,z) = \frac{\prod_{k=1}^N (1-x_k/z)}{\prod_{k=1}^i (1-x_k z)}, \qquad
    F_{HV}(i,z) = \frac{\prod_{k=1}^{\lfloor i/2 \rfloor} (1+x_{2k} z)}{
      \prod_{k=1}^{\lceil i/2 \rceil} (1-x_{2k-1} z)}
    \prod_{k=1}^M \frac{(1-x_{2k-1}/z)}{(1+x_{2k}/z)}.
  \end{equation}
\end{prop}

In the next three subsections we describe the main results stated and proved in the applications parts of the paper: Sections~\ref{sec:slpp}, \ref{sec:symm_pp} and \ref{sec:overpartitions}.

\subsubsection{Symmetric last passage percolation}

The \emph{last passage percolation (LPP) time} through a symmetric $n \times n$ nonnegative (integer or real) valued matrix $\omega$ is the maximum sum one can collect over all up-right paths going through the matrix from the bottom left entry to the upper right entry. We note our matrices, if embedded in the plane, are symmetric around the $x=y$ diagonal. See Figure~\ref{fig:lpp_matrix} for an example.

\begin{SCfigure}
\begin{tikzpicture}[scale = 0.6]
\foreach \x in {0,...,6} {
  \draw[thick, smooth] (\x, 0)--(\x, 6);
  \draw[thick, smooth] (0, \x)--(6, \x);
} 
\node at (0.5,0.5) {$8$};
\node at (0.5,1.5) {$5$};
\node at (0.5,2.5) {$0$};
\node at (0.5,3.5) {$1$};
\node at (0.5,4.5) {$1$};
\node at (0.5,5.5) {$1$};
\node at (1.5,1.5) {$2$};
\node at (1.5,2.5) {$9$};
\node at (1.5,3.5) {$9$};
\node at (1.5,4.5) {$3$};
\node at (1.5,5.5) {$2$};
\node at (2.5,2.5) {$2$};
\node at (2.5,3.5) {$2$};
\node at (2.5,4.5) {$11$};
\node at (2.5,5.5) {$1$};
\node at (3.5,3.5) {$1$};
\node at (3.5,4.5) {$10$};
\node at (3.5,5.5) {$0$};
\node at (4.5,4.5) {$11$};
\node at (4.5,5.5) {$0$};
\node at (5.5,5.5) {$1$};
\draw[thick, smooth, red] (0.5, 0.5)--(0.5, 1.5)--(1.5, 1.5)--(1.5, 4.5)--(4.5, 4.5)--(4.5, 5.5)--(5.5, 5.5);
\end{tikzpicture}
\caption{\fs{An example of a $6 \times 6$ symmetric matrix $\omega$ (only diagonal and above-diagonal elements shown) filled with nonnegative integers, and the last passage percolation time of 69 (the sum of the elements on the red path).}}
\label{fig:lpp_matrix}
\end{SCfigure}
  
In the present work we consider the LPP  time with symmetric, and (up to symmetry) independent geometric weights $\{\omega_{r,t}\}_{r,t \in \Z_{\geq 1}}.$  These weights (i.e. random variables) are given by  
\begin{equation}\label{symweightsintro}
  \omega_{r,t}=\omega_{t,r}\sim
\begin{cases}\vspace{0.3cm}
g(a_r a_t), \quad & \mathrm{if}\quad r\neq t, \\
g(\alpha a_r), \quad& \mathrm{if}\quad r=t
\end{cases}
\end{equation}
where $a_{n},\alpha a_{n}\in(0,1),n\geq 1,$ and $\Prob(g(q)=k)=q^{k}(1-q)$ for $k\in \Z_{\geq 0}$.
For $(n_1,n_2), (r,t) \in \Z^{2}_{\geq 1}$ with $n_1\leq r, n_2\leq t$, consider up-right paths $\pi$ from $(n_1,n_2)$ to $(r,t)$, i.e. $\pi =(\pi(0), \pi(1),\ldots, \pi(r-n_1+t-n_2))$ with $\pi(0)=(n_1,n_2), \pi(r-n_1+t-n_2)=(r,t)$ and $\pi(i)-\pi(i-1)\in \{(0,1),(1,0)\}$. The symmetric LPP time with geometric weights \eqref{symweightsintro} is then defined to be
\begin{equation}\label{symLPPintro}
L_{(n_1,n_2)\to (r,t)}:= \max_{\pi :(n_1,n_2)\to (r,t)}\sum_{(m,n)\in \pi} \omega_{m,n}.
\end{equation}

Under the RSK bijection, LPP times become the largest part of integer partitions --- see Figure~\ref{fig:lpp_partition} for a simulation. When considering $L_{(n_1,n_2)\to(r_l,t_l)},l=1,\ldots,k$ with $(r_l,t_l)$ lying on a down-right path, these partitions form a Schur process with one free boundary, which is H-ascending for  $(r_l,t_l)$ lying on a horizontal line --- see Section \ref{Proofs} for more details. Consequently, the event $\cap_{l=1}^{k}\{L_{(1,1)\to(r_l,t_l)}\leq s_l\}$ becomes the event that a point configuration \eqref{eq:pointconfig} has no points in a set $B$. Such gap probabilities are given by Fredholm pfaffians since we have a pfaffian point process by Theorem  \ref{thm:onebound} --- see Appendix \ref{sec:pfaffian} for the definition of Fredholm pfaffians. This leads to the following theorem, which will be proven as Theorem \ref{LPPmulti} in Section \ref{Proofs}.
\begin{thm} \label{LPPmultisec2} 
  Consider the LPP time \eqref{symLPPintro} with weights \eqref{symweightsintro}.
Let  $r_l,t_l\in \Z_{\geq 1},r_l \leq t_l,l=1,\ldots,k$ with  $r_1 \leq \ldots \leq r_k , t_1 \geq \ldots \geq t_k .$  Then
\begin{align}\label{labelneeded2}
\Prob\left(\bigcap_{l=1}^{k}\{L_{(1,1)\to (r_l,t_l) }\leq s_l \} \right)=\pf(J-K)_{B},
\end{align}
where $K,B$ are given in Theorem \ref{LPPmulti}.
\end{thm}

The identity \eqref{labelneeded2} now allows us to extract asymptotics of the LPP time $L_{(1,1)\to(r,t)}$ as $r,t\to \infty.$ The limiting fluctuations will, of course, depend on the end point $(r,t)$ and the choice of parameters $a_n, \alpha$. Here, we do not aim to exploit Theorem \ref{LPPmultisec2} in all possible directions. In the following theorem we fix  $a_n =\sqrt{q}$ and choose $\alpha$ and the endpoint such that we are in a crossover regime, from which different limit laws can be recovered. The following Theorem will be proven as Theorem \ref{LPPThm2} in Section \ref{Proofs}.

\begin{thm}\label{LPPThm2intro}
Consider the weights \eqref{symweightsintro} with $a_j=\sqrt{q},q\in(0,1),j\geq 1$ and $\alpha=1-2vc_{q}N^{-1/3}, $ where  $c_q=\frac{1-\sqrt{q}}{q^{1/6}(1+\sqrt{q})^{1/3}},v\in\R$. Let $u_1 >\cdots >u_k \geq 0$.
Then 
\begin{equation}
  \begin{split}
\lim_{N \to \infty}\Prob\left(\bigcap_{i=1}^{k}\big\{L_{(1,1)\to (N-\lfloor u_i N^{2/3}\rfloor,N)}\leq \frac{2\sqrt{q}N}{1-\sqrt{q}}-u_{i}\frac{\sqrt{q}N^{2/3}}{1-\sqrt{q}}+c_{q}^{-1}s_i N^{1/3}\big\}\right)  \\
 = \pf( J - \chi_{s} K^{v} \chi_{s})_{\{u_1,\ldots,u_k\}\times \R}
  \end{split}
\end{equation}
where $\chi_{s}(u_i,x)=\mathbf{1}_{x>s_i}$ and $K^{v}$ is defined in \eqref{Kv}.
\end{thm}
Specializing to $k=1$ we obtain in particular
\begin{equation}
\lim_{N \to \infty}\Prob\left(L_{(1,1)\to (N-\lfloor u N^{2/3}\rfloor,N)}\leq \frac{2\sqrt{q}N}{1-\sqrt{q}}-u\frac{\sqrt{q}N^{2/3}}{1-\sqrt{q}}+c_{q}^{-1}sN^{1/3}\right)=\pf( J - K^{v} )_{(s,\infty)}=:F_{u,v}(s).
\end{equation}
$F_{u,v}(s)$ performs a crossover between the classical distributions from random matrix theory. Namely, one has $F_{0,0}=F_{\GOE},$ $\lim_{v\to +\infty}F_{0,v}(s)=F_{\GSE}(s)$ and $\lim_{u\to +\infty}F_{u,v}(s-u^{2}d_{q}^{2})=F_{\GUE}(s),$ where $d_q=\frac{q^{1/6}}{2(1+\sqrt{q})^{2/3}}$ --- see Section \ref{sec:defdist}.

For $k=1$ and $u=0$, Theorem   \ref{LPPThm2intro} recovers results already obtained by Baik and Rains \cite{br99, br1, br2}. Furthermore, \cite{si} considered off-diagonal fluctuations in half-space PNG, equivalent to symmetric LPP. For symmetric LPP with exponential weights, the same kind of crossover  between $F_{\GSE},F_{\GUE},F_{\GOE}$ was obtained, by different methods,  recently in \cite{BBCS17}.

\begin{SCfigure}
\includegraphics[scale=0.18]{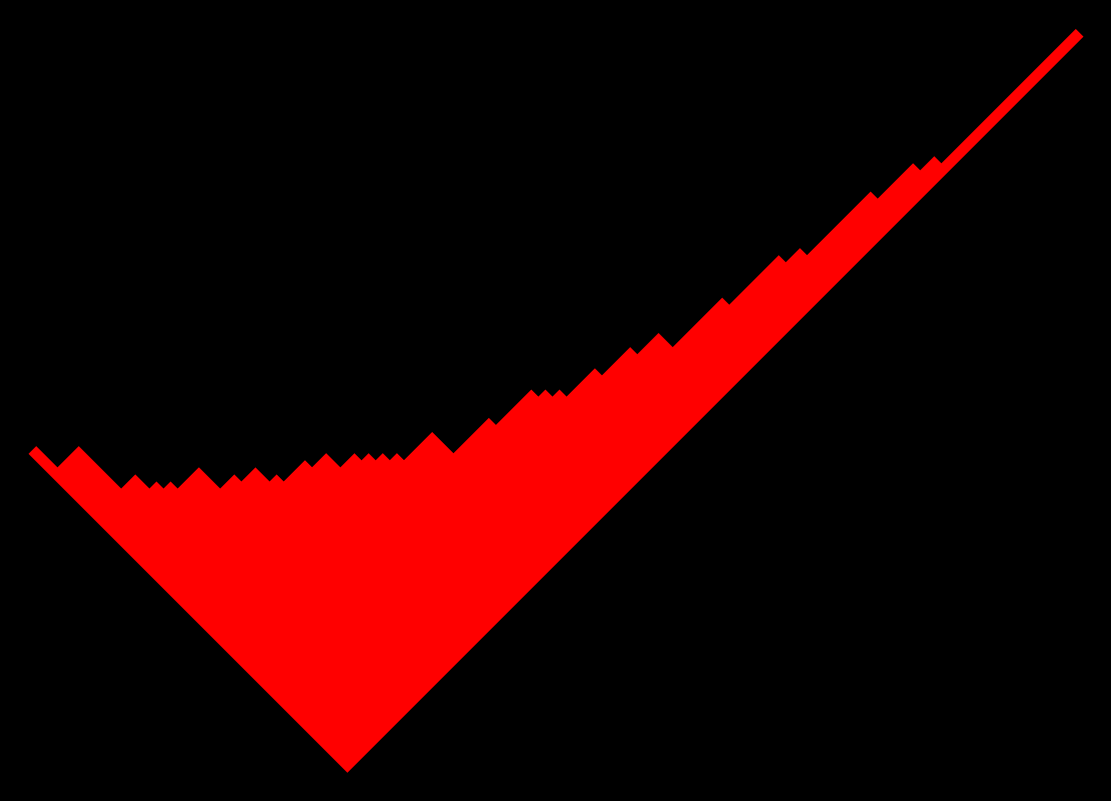}
\caption{\fs{A random partition $\lambda$, in Russian notation, whose weight is proportional to $s_{\lambda} (q, \dots, q)$ (100 $q$'s) for $q=1/3$. The right arm (its first part) is, via the RSK correspondence, the last passage percolation time in a $100 \times 100$ symmetric matrix filled with otherwise independent  geometric random numbers of parameter $q^2$ (off diagonal) and $q$ (on the diagonal). }}
\label{fig:lpp_partition}
\end{SCfigure}

\subsubsection{Symmetric plane partitions}

A \emph{symmetric plane partition of length $N$} is a plane partition --- i.e. an array of numbers $(\pi_{i,j})_{1 \leq i, j \leq N}$ such that $\pi_{i, j} \geq \pi_{i, j+1}, \pi_{i, j} \geq \pi_{i+1, j}$, satisfying the symmetry condition $\pi_{i, j} = \pi_{j, i}$. It can be viewed as a symmetric pile of cubes stacked into the corner of a room or a lozenge tiling of the plane. See Figure~\ref{fig:pp_example} below for pictorial descriptions. A symmetric plane partition (more precisely, the half of it that determines the whole --- Figure~\ref{fig:pp_example} on the left) can be sliced into ordinary partitions $\vec{\lambda} = (\emptyset \prec \lambda^{(1)} \prec \dots \prec \lambda^{(N)})$ with $\ell(\lambda^{(i)}) \leq i$ using the simple formula $\lambda^{(i)}_k = \pi_{N-i+k,k}$ for $1 \leq k \leq i$. We study the $q^{\text{Volume}}$ measure, for $q \in (0,1)$, which can be treated as an $H$-ascending Schur process for appropriately chosen (single variable) specializations.

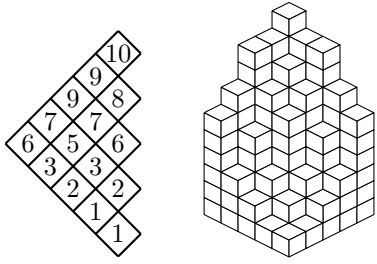
\begin{SCfigure}
  \begin{tikzpicture}[scale=0.3]
    \draw[thick, smooth] (-1,0)--(4,5);
      \draw[thick, smooth] (-1,0)--(4,-5);
    \foreach \x in {0,...,4} {
        \draw[thick, smooth] (\x, {-\x-1})--(5, {-2*\x+4});
        \draw[thick, smooth] (\x, {\x+1})--(5, {2*\x-4});
    }
    \node at (0,0) {$6$};
    \node at (1,1) {$7$};
    \node at (1,-1) {$3$};
    \node at (2,2) {$9$};
    \node at (2,0) {$5$};
    \node at (2,-2) {$2$};
    \node at (3,3) {$9$};
    \node at (3,1) {$7$};
    \node at (3,-1) {$3$};
    \node at (3,-3) {$1$};
    \node at (4,4) {$10$};
    \node at (4,2) {$8$};
    \node at (4,0) {$6$};
    \node at (4,-2) {$2$};
    \node at (4,-4) {$1$};
  \end{tikzpicture}
  \qquad
\includegraphics[scale=0.2]{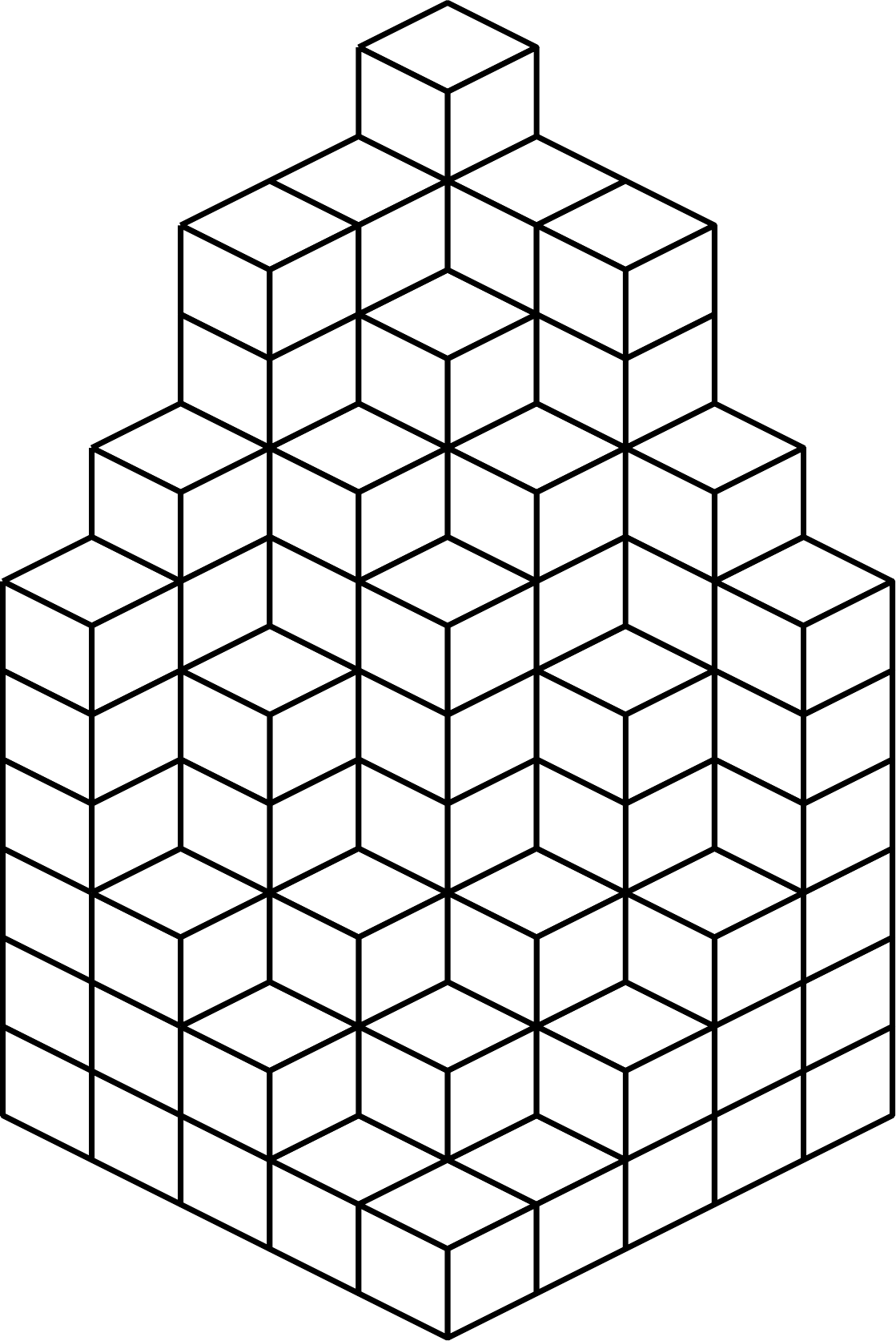}
\caption{\fs{A symmetric plane partition: the numbers on the left determine the heights of cubes on the right.}}
\label{fig:pp_example}
\end{SCfigure}

In Section~\ref{sec:symm_pp} we consider the large volume limit of symmetric plane partitions: $q =e^{-r}\to 1$ for scaling $ri \to \x$ and $r k \to \y$ where $(\x, \y) \in \mathbb{R}_+ \times \mathbb{R}$. A sample of a symmetric random plane for $q$ close to 1 is given in Figure~\ref{fig:pp_large_intro} (left for $r N \to \infty$, right for $rN \to \a < \infty$). In both cases there are two distinct regions: the liquid region $\mathcal{L}$ where behavior is random, and the frozen region where behavior is deterministic. One can visualize this as in the figure using different colors for the three types of lozenges in the plane tiling. 

The arctic curve, dividing the liquid and frozen regions is the zero locus of 
\begin{equation}
  \begin{split}
  D(\X, \Y) &= -4 (1 - \A \Y) (\X - \A \Y) + (-1 - \X + \Y + \A^2 \Y)^2
  \end{split}
\end{equation}
where $\X = \exp(-\x)$, $\Y = \exp(-\y)$, and $\A = \exp(-\a)$.

Case $\A=0$ corresponds to symmetric plane partitions with no bound on the length, and the liquid region can be written in appropriate $(u, v)$ coordinates as the amoeba of the polynomial $1+u+v$ --- see Section~\ref{sec:symm_pp} for the definition of amoeba. As expected, one obtains the same liquid region as for non-symmetric plane partitions. The arctic shape for plane partitions was first obtained by Bl\"ote, Hilhorst and Nienhuis \cite{bnh} in the physics literature and Cerf and Kenyon \cite{ck} in the mathematics literature. It was later rederived by Okounkov and Reshetikhin \cite{or} using the Schur process. 

The limit shape result can be derived from the following explicit formula for the density of particles.

\begin{prop} \label{Densityspp} For  $(\x,\y) \in \mathcal{L}$ the density of particles is
\begin{equation}
\rho(\x,\y)=\frac{\theta(\x,\y)}{\pi}
\end{equation}
where $\theta(\x,\y)=\arg(z_{+}(\x,\y))$ and
\begin{align} 
     z_{\pm} (\x, \y) = \frac{1+\X-(1+\A^2)\Y \pm \sqrt{D(\X, \Y)}} {2(\X-\A\Y)}.
\end{align}
\end{prop}

To derive this result we start from the finite correlations given by Theorem~\ref{thm:onebound}, changing $i \to N-i$ for convenience, and perform steepest descent analysis to obtain the limiting behavior of the correlation kernel. In the limit we obtain a (incomplete beta) determinantal process for $\x > 0$ and a pfaffian process for $\x=0$.

 \begin{thm} \label{MainTheoremSPP}
   Let $(\x,\y) \in \mathcal{L}$ and rescale the coordinates as
\begin{align}\label{scaling}
   i = \left\lfloor \frac{\x}{r} \right\rfloor + \i, \qquad k = \left\lfloor \frac{\y}{r} \right\rfloor + \k
   \end{align}
   where $r \to 0+$ and $\i,\k$ are fixed. Then, the rescaled process
    \begin{itemize}
    \item converges for $\x = 0$ to a pfaffian process with kernel 
    \begin{equation} \label{Kspp}
    \begin{split}
    \mathsf{K}_{1,1}(\i, \k; \i', \k') &= 
    \int_{\gamma_{+}} (1-z)^{\i} \left(1-\frac{1}{z}\right)^{\i'} z^{\k - \k' - 1} \frac{1-z}{1+z} \frac{\dx z}{2 \pi \im}, \\
        \mathsf{K}_{1,2}(\i, \k; \i', \k') &= \int_{\gamma_{\pm}} (1-z)^{\i - \i'} z^{\k' - \k - 1} \frac{\dx z}{2 \pi \im},\\
    \mathsf{K}_{2,2}(\i, \k; \i', \k') &= \int_{\gamma_{-}} (1-z)^{-\i} \left(1-\frac{1}{z}\right)^{-\i'} z^{\k' - \k - 1} \frac{1+z}{1-z} \frac{\dx z}{2 \pi \im}
    \end{split}
  \end{equation}
  \noindent where $\gamma_+$ is taken if and only if $\i \geq \i'$. $\gamma_{\pm}$ is a contour from $z_{-}(\x,\y)$ to $z_+(\x,\y)$, $\gamma_+$ passing to the right of 0 and $\gamma_-$ to left of 0;
  \item converges for $\x \neq 0$ to a determinantal process with kernel $\mathsf{K}_{1,2}$ as in \eqref{Kspp}.
 \end{itemize}
\end{thm}

\begin{rem}
  \label{rem:corrdecay}
  It is natural to consider the asymptotics of the kernel entries
  \eqref{Kspp} as $\i$, $\i'$, $\i-\i'$ or $\k-\k'$ tend to
  $\pm \infty$. This may be done using Laplace's method: the integrand
  has the generic form $f(z)^{\mathsf{n}} g(z)$ where $\mathsf{n}$
  denotes a parameter tending to $+\infty$, and it is always possible
  to choose an integration contour such that $|f(z)|$ is maximal at
  the endpoints $z_\pm(\x,\y)$, with a nonvanishing derivative. As $g$
  does not vanish at the endpoints, we deduce that the integral
  behaves at leading order as
  $\frac{ \alpha f(z_+)^{\mathsf{n}} + \bar{\alpha}
    f(z_-)^{\mathsf{n}}}{\mathsf{n}}$ for some $\alpha \in \C$. Using
  Proposition~\ref{ConjPfaff}, we may perform a rescaling to suppress
  the exponential blowup/decay, so that the rescaled entry behaves as
  $\mathsf{n}^{-1}$ times some oscillating factor. The bottom line is
  that, up to oscillations, the properly rescaled correlation kernel
  decays as the inverse of the distance between points, and its
  diagonal entries (which make the process nondeterminantal) decay as
  the inverse of the distance to the free boundary.
\end{rem}

\begin{rem} \label{rem:pfaffin} In the case $\x = 0$ it is crucial
  that we keep $i=\i$ finite as $r \to 0+$ to obtain a pfaffian
  process. If instead we rescale $i=f(r)+\i$ with
  $1 \ll f(r) \ll r^{-1}$, then we obtain a determinantal process with
  the kernel $\mathsf{K}_{1,2}$ of \eqref{Kspp} at $\x=0$. Intuitively
  speaking, the convergence to $\mathsf{K}_{1,2}$ is more robust as it
  only depends on the difference $\i - \i'$.  Together with the
  previous remark, this shows that the free boundary affects the
  nature of correlations only at a finite range in the bulk. In
  contrast, at the edge, as illustrated by Theorem~\ref{LPPThm2intro},
  correlations remain pfaffian on a larger range (namely $N^{2/3}$ in
  the LPP setting).
\end{rem}

\begin{figure}[ht]
  \begin{center}
  \includegraphics[scale=0.11]{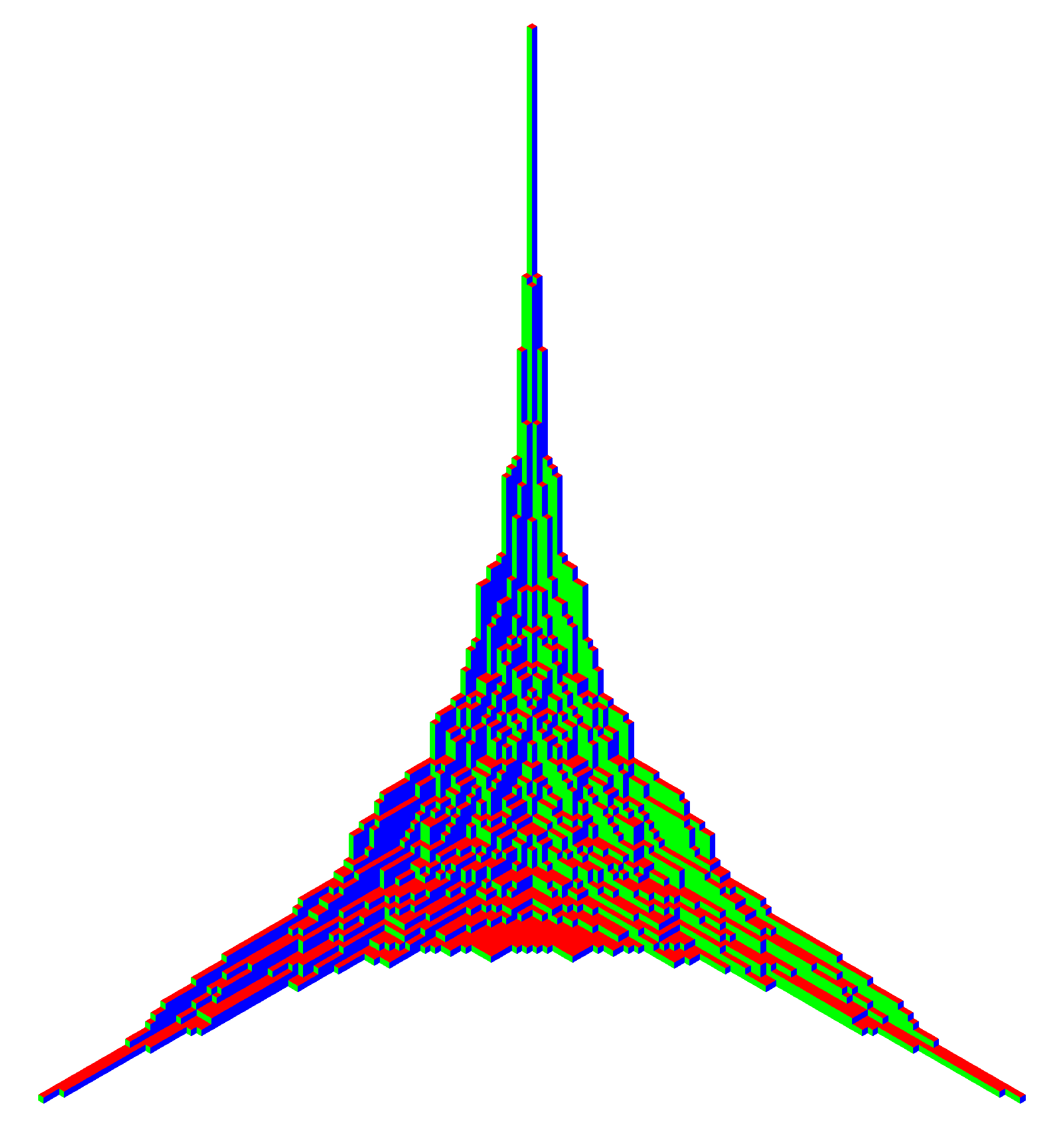}
  \quad \qquad
  \includegraphics[scale=0.04]{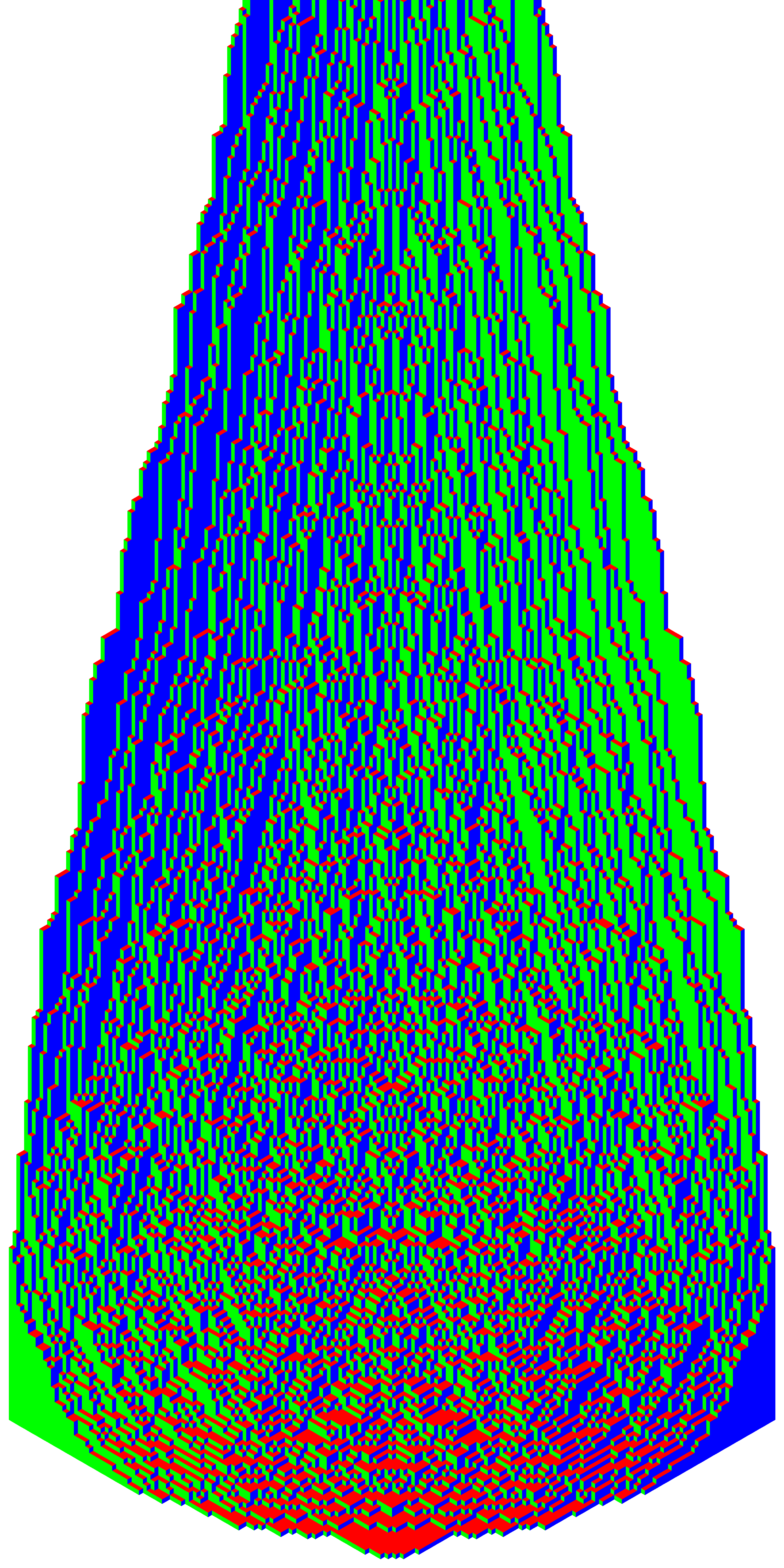}
  \end{center}
  \caption{\fs{Large unbounded (left) and bounded (right) random symmetric $q^{\text{Volume}}$-weighted plane partitions.}}
  \label{fig:pp_large_intro}
  \end{figure}

\subsubsection{Plane overpartitions}

A \emph{plane overpartition} is a plane partition where in each row the last occurrence of an integer can be overlined or not and all the other occurrences of this integer are not overlined, while in each column the first occurrence of an integer can be overlined or not and all the other occurrences of this integer are overlined. An example is given in Figure~\ref{fig:po_example}. There is a natural measure one can study on plane overpartitions: the $q^{\text{Volume}}$ measure, where the volume is given by the sum of all its entries.

A plane overpartition with the largest entry at most $N$ and shape $\lambda$ can be recorded as a sequence of partitions $\vec{\lambda} = ( \emptyset \prec \lambda^{(1)} \prec'\lambda^{(2)}\prec\cdots \prec \lambda^{(2n-1)} \prec' \lambda^{(2N)}=\lambda )$ where $\lambda^{(i)}$ is the partition whose shape is formed by all fillings greater than ${N-i/2}$ with the convention that $\overline{k}=k-1/2$. In this context, with the $q^{\text{Volume}}$ measure considered, the sequence $\vec{\lambda}$ becomes an $HV$-ascending Schur process for appropriately chosen (single or dual variable) specializations.

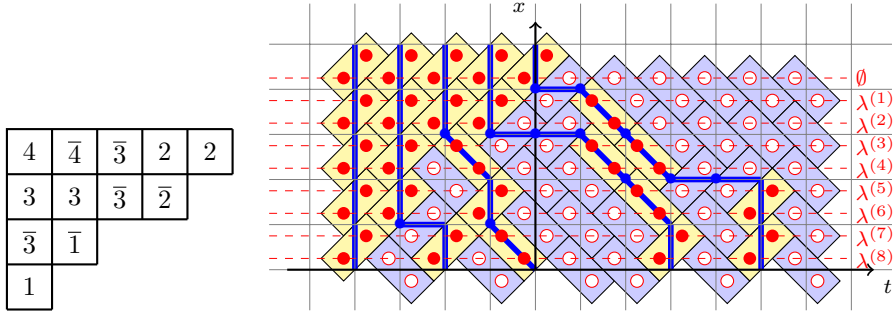
\begin{figure}
\begin{center}
\begin{tikzpicture}[scale=0.6]
\draw[thick, smooth] (0,0)--(0,4)--(5,4);
\draw[thick, smooth] (0,0)--(1,0);
\draw[thick, smooth] (0,1)--(2,1);
\draw[thick, smooth] (0,2)--(4,2);
\draw[thick, smooth] (0,3)--(5,3);
\draw[thick, smooth] (1,0)--(1,4);
\draw[thick, smooth] (2,1)--(2,4);
\draw[thick, smooth] (3,2)--(3,4);
\draw[thick, smooth] (4,2)--(4,4);
\draw[thick, smooth] (5,3)--(5,4);
\node at (0.5,0.5) { 1};
\node at (0.5,1.5) {$\overline{3}$};
\node at (0.5,2.5) { 3};
\node at (0.5,3.5) { 4};
\node at (1.5,1.5) { $\overline{1}$};
\node at (1.5,2.5) { 3};
\node at (1.5,3.5) { $\overline{4}$};
\node at (2.5,2.5) { $\overline{3}$};
\node at (2.5,3.5) { $\overline{3}$};
\node at (3.5,2.5) {$\overline{2}$};
\node at (3.5,3.5) { 2};
\node at (4.5,3.5) {2};
\end{tikzpicture}
\quad
\begin{tikzpicture}[scale=0.6]
  \foreach \x in {0}{ \foreach \y in {3,4}
   {\filldraw [fill=blue!20!white, draw=black]  (\x-0.25,\y-0.25) -- (\x+0.25,\y-0.75) --  (\x+1.25,\y+0.25) -- (\x+0.75,\y+0.75)-- cycle;
   \draw[line width=0.75mm,blue] (\x,\y) -- (\x+1,\y);
   \filldraw[fill=white, draw=red] (\x+0.25,\y-0.25) circle[radius=4pt] ;
    \filldraw[fill=white, draw=red] (\x+0.75,\y+0.25) circle[radius=4pt] ;}};
  
   \foreach \x in {-1}{ \foreach \y in {3}
   {\filldraw [fill=blue!20!white, draw=black]  (\x-0.25,\y-0.25) -- (\x+0.25,\y-0.75) --  (\x+1.25,\y+0.25) -- (\x+0.75,\y+0.75)-- cycle;
   \draw[line width=0.75mm,blue] (\x,\y) -- (\x+1,\y);
   \filldraw[fill=white, draw=red] (\x+0.25,\y-0.25) circle[radius=4pt] ;
   \filldraw[fill=white, draw=red] (\x+0.75,\y+0.25) circle[radius=4pt] ;}};
  
   \foreach \x in {-3}{ \foreach \y in {1}
   {\filldraw [fill=blue!20!white, draw=black]  (\x-0.25,\y-0.25) -- (\x+0.25,\y-0.75) --  (\x+1.25,\y+0.25) -- (\x+0.75,\y+0.75)-- cycle;
   \draw[line width=0.75mm,blue] (\x,\y) -- (\x+1,\y);
   \filldraw[fill=white, draw=red] (\x+0.25,\y-0.25) circle[radius=4pt] ;
   \filldraw[fill=white, draw=red] (\x+0.75,\y+0.25) circle[radius=4pt] ;}};
  
   \foreach \x in {3}{ \foreach \y in {2}
   {\filldraw [fill=blue!20!white, draw=black]  (\x-0.25,\y-0.25) -- (\x+0.25,\y-0.75) --  (\x+1.25,\y+0.25) -- (\x+0.75,\y+0.75)-- cycle;
   \draw[line width=0.75mm,blue] (\x,\y) -- (\x+1,\y);
   \filldraw[fill=white, draw=red] (\x+0.25,\y-0.25) circle[radius=4pt] ;
   \filldraw[fill=white, draw=red] (\x+0.75,\y+0.25) circle[radius=4pt] ;}};
  
   \foreach \x in {4}{ \foreach \y in {2}
   {\filldraw [fill=blue!20!white, draw=black]  (\x-0.25,\y-0.25) -- (\x+0.25,\y-0.75) --  (\x+1.25,\y+0.25) -- (\x+0.75,\y+0.75)-- cycle;
   \draw[line width=0.75mm,blue] (\x,\y) -- (\x+1,\y);
   \filldraw[fill=white, draw=red] (\x+0.25,\y-0.25) circle[radius=4pt] ;
   \filldraw[fill=white, draw=red] (\x+0.75,\y+0.25) circle[radius=4pt] ;}};
  
    \foreach \x in {2,3,4,5,6}{ \foreach \y in {4}
   {\filldraw [fill=blue!20!white, draw=black]  (\x-0.75,\y+0.25) -- (\x+0.25,\y-0.75) --  (\x+0.75,\y-0.25) -- (\x-0.25,\y+0.75)-- cycle;
   \filldraw[fill=white, draw=red] (\x+0.25,\y-0.25) circle[radius=4pt] ;
   \filldraw[fill=white, draw=red] (\x-0.25,\y+0.25) circle[radius=4pt] ;}};
  
    \foreach \x in {3,4,5,6}{ \foreach \y in {3}
   {\filldraw [fill=blue!20!white, draw=black]  (\x-0.75,\y+0.25) -- (\x+0.25,\y-0.75) --  (\x+0.75,\y-0.25) -- (\x-0.25,\y+0.75)-- cycle;
   \filldraw[fill=white, draw=red] (\x+0.25,\y-0.25) circle[radius=4pt] ;
   \filldraw[fill=white, draw=red] (\x-0.25,\y+0.25) circle[radius=4pt] ;}};
  
     \foreach \x in {-2,0,1,6}{ \foreach \y in {2}
   {\filldraw [fill=blue!20!white, draw=black]  (\x-0.75,\y+0.25) -- (\x+0.25,\y-0.75) --  (\x+0.75,\y-0.25) -- (\x-0.25,\y+0.75)-- cycle;
   \filldraw[fill=white, draw=red] (\x+0.25,\y-0.25) circle[radius=4pt] ;
   \filldraw[fill=white, draw=red] (\x-0.25,\y+0.25) circle[radius=4pt] ;}};
  
    \foreach \x in {0,1,2,4,6}{ \foreach \y in {1}
   {\filldraw [fill=blue!20!white, draw=black]  (\x-0.75,\y+0.25) -- (\x+0.25,\y-0.75) --  (\x+0.75,\y-0.25) -- (\x-0.25,\y+0.75)-- cycle;
   \filldraw[fill=white, draw=red] (\x+0.25,\y-0.25) circle[radius=4pt] ;
   \filldraw[fill=white, draw=red] (\x-0.25,\y+0.25) circle[radius=4pt] ;}};
  
    \foreach \x in {-3,-1,1,2,4,6}{ \foreach \y in {0}
   {\filldraw [fill=blue!20!white, draw=black]  (\x-0.75,\y+0.25) -- (\x+0.25,\y-0.75) --  (\x+0.75,\y-0.25) -- (\x-0.25,\y+0.75)-- cycle;
   \filldraw[fill=white, draw=red] (\x+0.25,\y-0.25) circle[radius=4pt] ;
   \filldraw[fill=white, draw=red] (\x-0.25,\y+0.25) circle[radius=4pt] ;}};
  
   \foreach \x in {1}{ \foreach \y in {3,4}
   {\filldraw [fill=yellow!40!white, draw=black]  (\x-0.25,\y-0.25) -- (\x+0.75,\y-1.25) --  (\x+1.25,\y-0.75) -- (\x+0.25,\y+0.25)-- cycle;
   \draw[line width=0.75mm,blue] (\x,\y) -- (\x+1,\y-1);
   \fill[red] (\x+0.25,\y-0.25) circle[radius=4pt];
  \fill[red] (\x+0.75,\y-0.75) circle[radius=4pt];}};
  
    \foreach \x in {2}{ \foreach \y in {2,3}
   {\filldraw [fill=yellow!40!white, draw=black]  (\x-0.25,\y-0.25) -- (\x+0.75,\y-1.25) --  (\x+1.25,\y-0.75) -- (\x+0.25,\y+0.25)-- cycle;
   \draw[line width=0.75mm,blue] (\x,\y) -- (\x+1,\y-1);
   \fill[red] (\x+0.25,\y-0.25) circle[radius=4pt];
  \fill[red] (\x+0.75,\y-0.75) circle[radius=4pt];}};
  
   \foreach \x in {-1}{ \foreach \y in {1}
   {\filldraw [fill=yellow!40!white, draw=black]  (\x-0.25,\y-0.25) -- (\x+0.75,\y-1.25) --  (\x+1.25,\y-0.75) -- (\x+0.25,\y+0.25)-- cycle;
   \draw[line width=0.75mm,blue] (\x,\y) -- (\x+1,\y-1);
   \fill[red] (\x+0.25,\y-0.25) circle[radius=4pt];
  \fill[red] (\x+0.75,\y-0.75) circle[radius=4pt];}};
  
   \foreach \x in {-2}{ \foreach \y in {3}
   {\filldraw [fill=yellow!40!white, draw=black]  (\x-0.25,\y-0.25) -- (\x+0.75,\y-1.25) --  (\x+1.25,\y-0.75) -- (\x+0.25,\y+0.25)-- cycle;
   \draw[line width=0.75mm,blue] (\x,\y) -- (\x+1,\y-1);
  \fill[red] (\x+0.25,\y-0.25) circle[radius=4pt];
  \fill[red] (\x+0.75,\y-0.75) circle[radius=4pt];}};
  
  \foreach \x in {5}{ \foreach \y in {2,1}
   {\filldraw [fill=yellow!40!white, draw=black]  (\x-0.75,\y-0.75) -- (\x-0.25,\y-1.25) --  (\x+0.75,\y-0.25) -- (\x+0.25,\y+0.25)-- cycle;
   \draw[line width=0.75mm,blue] (\x,\y) -- (\x,\y-1);
  \fill[red] (\x+0.25,\y-0.25) circle[radius=4pt];
  \fill[red] (\x-0.25,\y-0.75) circle[radius=4pt];}};
  
  \foreach \x in {3}{ \foreach \y in {1}
   {\filldraw [fill=yellow!40!white, draw=black]  (\x-0.75,\y-0.75) -- (\x-0.25,\y-1.25) --  (\x+0.75,\y-0.25) -- (\x+0.25,\y+0.25)-- cycle;
   \draw[line width=0.75mm,blue] (\x,\y) -- (\x,\y-1);
  \fill[red] (\x+0.25,\y-0.25) circle[radius=4pt];
  \fill[red] (\x-0.25,\y-0.75) circle[radius=4pt];}};
  
  \foreach \x in {-1,-2,-3,-4}{ \foreach \y in {4,5}
   {\filldraw [fill=yellow!40!white, draw=black]  (\x-0.75,\y-0.75) -- (\x-0.25,\y-1.25) --  (\x+0.75,\y-0.25) -- (\x+0.25,\y+0.25)-- cycle;
   \draw[line width=0.75mm,blue] (\x,\y) -- (\x,\y-1);
  \fill[red] (\x+0.25,\y-0.25) circle[radius=4pt];
  \fill[red] (\x-0.25,\y-0.75) circle[radius=4pt];}};
  
  \foreach \x in {-3,-4}{ \foreach \y in {2,3}
   {\filldraw [fill=yellow!40!white, draw=black]  (\x-0.75,\y-0.75) -- (\x-0.25,\y-1.25) --  (\x+0.75,\y-0.25) -- (\x+0.25,\y+0.25)-- cycle;
   \draw[line width=0.75mm,blue] (\x,\y) -- (\x,\y-1);
  \fill[red] (\x+0.25,\y-0.25) circle[radius=4pt];
  \fill[red] (\x-0.25,\y-0.75) circle[radius=4pt];}};
  
  \foreach \x in {-1}{ \foreach \y in {2}
   {\filldraw [fill=yellow!40!white, draw=black]  (\x-0.75,\y-0.75) -- (\x-0.25,\y-1.25) --  (\x+0.75,\y-0.25) -- (\x+0.25,\y+0.25)-- cycle;
   \draw[line width=0.75mm,blue] (\x,\y) -- (\x,\y-1);
  \fill[red] (\x+0.25,\y-0.25) circle[radius=4pt];
  \fill[red] (\x-0.25,\y-0.75) circle[radius=4pt];}};
  
  \foreach \x in {-2,-4}{ \foreach \y in {1}
   {\filldraw [fill=yellow!40!white, draw=black]  (\x-0.75,\y-0.75) -- (\x-0.25,\y-1.25) --  (\x+0.75,\y-0.25) -- (\x+0.25,\y+0.25)-- cycle;
   \draw[line width=0.75mm,blue] (\x,\y) -- (\x,\y-1);
  \fill[red] (\x+0.25,\y-0.25) circle[radius=4pt];
  \fill[red] (\x-0.25,\y-0.75) circle[radius=4pt];}};
  
  \foreach \x in {0}{ \foreach \y in {5}
   {\filldraw [fill=yellow!40!white, draw=black]  (\x-0.75,\y-0.75) -- (\x-0.25,\y-1.25) --  (\x+0.75,\y-0.25) -- (\x+0.25,\y+0.25)-- cycle;
   \draw[line width=0.75mm,blue] (\x,\y) -- (\x,\y-1);
  \fill[red] (\x+0.25,\y-0.25) circle[radius=4pt];
  \fill[red] (\x-0.25,\y-0.75) circle[radius=4pt];}};
  
  \foreach \x in {0} \draw[dashed,red,very thin] (-5.9,4.25-\x/2)--(6.9,4.25-\x/2)node[anchor=west] {\fs{$\emptyset$}};
  \foreach \x in {2,4,6,8} \draw[dashed,red,very thin] (-5.9,4.25-\x/2)--(6.9,4.25-\x/2)node[anchor=west] {\fs{$\lambda^{(\x)}$}};
  \foreach \x in {1,3,5,7} \draw[dashed,red,very thin] (-5.9,4.25-\x/2)--(6.9,4.25-\x/2)node[anchor=west] {\fs{$\lambda^{(\x)}$}};
  \draw[step=1cm,gray,very thin] (-5.9,-0.9) grid (7.9,5.9);
   \draw[thick,->] (-5.5,0) -- (7.5,0) node[anchor=north west] {\fs{$t$}};
  \draw[thick,->] (0,0) -- (0,5.5) node[anchor=south east] {\fs{$x$}};
  \foreach \xy in {(-3,1),(-2,3),(-1,3),(-1,1),(0,4), (0,3),(1,4),(1,3),(2,3),(2,2),(3,2),(4,2)} {\node at \xy {$\color{blue} \bullet$};};
  \end{tikzpicture}
\end{center}
\caption{\fs{A plane overpartition (left) and its associated point configuration (right).}}
\label{fig:po_example}
\end{figure}

Plane overpartitions are in bijection with domino tilings --- see Figure~\ref{fig:po_example} (right). We hope that the ``picture is worth a thousand words'', but for more details see Section~\ref{sec:overpartitions}. In Section~\ref{sec:overpartitions} we consider the large volume limit of plane overpartitions: $q=e^{-r} \to 1$ for scaling $ri\to \x$ and $rk \to \y$. We consider the case $r N \to \infty$ only and leave $rN \to \a < \infty$ for subsequent work. A sample for $q$ close to 1 is given Figure~\ref{fig:po_large}.

The liquid region $\mathcal{L}$ is half of the amoeba of the polynomial $-1+u+v+uv$ for the right choice of coordinates $(u,v)$. The density in the liquid region is as in Proposition~\ref{Densityspp} for different $z_{\pm}$, with the explicit expression given in Section~\ref{sec:overpartitions}. This was originally obtained in \cite{vul2}, see also~\cite{vul3} for results on the convergence of height fluctuations to the Gaussian free field in the equivalent language of strict plane partitions.

We analyze the pfaffian local correlations given by Theorem~\ref{thm:onebound} in the limit and obtain an analogue of Theorem~\ref{MainTheoremSPP}.
(Remarks~\ref{rem:corrdecay} and \ref{rem:pfaffin} still hold mutatis mutandis.)

\begin{thm}
  \label{MainTheoremPO} 
    Let $(\x,\y)\in \mathcal{L}$. As $r \to 0+$ the rescaled process, where the rescaling is given precisely in Section~\ref{sec:overpartitions}, 
    \begin{itemize} 
   \item converges for $\x = 0$ to a pfaffian process with kernel 
   \begin{equation}\label{Kpop}
    \begin{split}
    \mathsf{K}_{1,1}(\i, \k; \i', \k') &= 
    \int_{\gamma_{+}} \left(\frac{1-z} {1+{z}}\right)^{\i+\i'+1} (-1)^{\i} z^{\k - \k' - 1}  \frac{\dx z}{2 \pi \im}, \\
        \mathsf{K}_{1,2}(\i, \k; \i', \k') &= \int_{\gamma_{\pm}}\left(\frac{1-z} {1+{z}}\right)^{\i-\i'}  z^{\k' - \k - 1} \frac{\dx z}{2 \pi \im},\\
    \mathsf{K}_{2,2}(\i, \k; \i', \k') &= \int_{\gamma_{-}} \left(\frac{1-z} {1+{z}}\right)^{-\i-\i'-1}  (-1)^{\i}z^{\k' - \k - 1} \frac{\dx z}{2 \pi \im}
    \end{split}
  \end{equation}
   \noindent where $\gamma_+$ is taken if and only if $\i \geq \i'$ and $\gamma_{\pm}$ are defined as in Theorem~\ref{MainTheoremSPP};
  \item converges for $\x \neq 0$ to  a determinantal process with kernel $\mathsf{K}_{1,2}$ as in \eqref{Kpop}.
     \end{itemize}
   \end{thm}

\begin{figure} 
\begin{center}
\includegraphics[scale = 0.10]{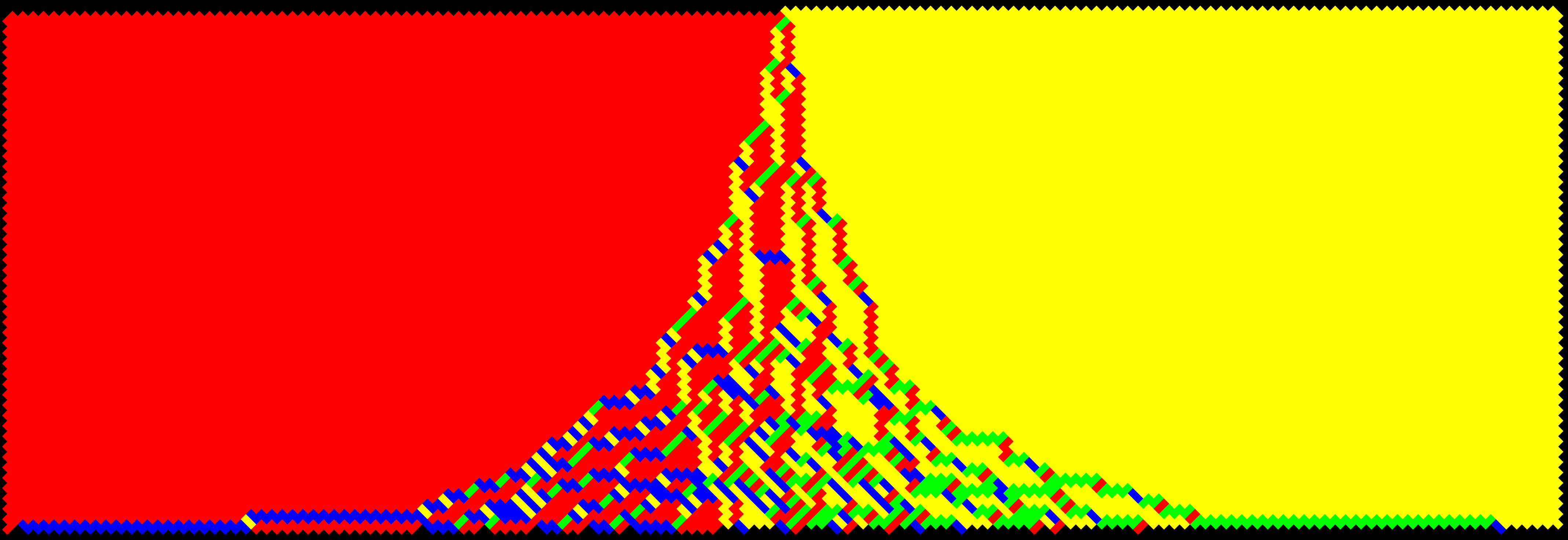} 
\end{center}
\caption{\fs{A large random plane overpartition shown as a domino tiling.}}
\label{fig:po_large}
\end{figure}

\section{Free fermions}
\label{sec:fermions}

This section is devoted to the proof of the results presented in
Section~\ref{sec:fbcorr} via the free fermion formalism.

\subsection{Preliminaries}

\subsubsection{Notations and reminders}
\label{sec:fermrem}

Here we recall the standard material which is useful for the study of
the usual Schur process \cite{or}, following the notation conventions
of \cite[Appendix A]{oko}. See also \cite{jm,djm,az} and \cite[Chapter
14]{kac}.

\paragraph{Admissible sets and partitions.}

Let us denote by $\Z':=\Z+1/2$ the set of half-integers. We say that a
subset $S$ of $\Z'$ is \emph{admissible} if it has a greatest element
and its complement has a least element. Equivalently, we require
$S_+ := S \setminus \Z'_{< 0}$ and $S_- := \Z'_{< 0} \setminus S$ to
be both finite. We denote by $\s$ the set of admissible subsets. To
each $S \in \s$ we may associate its \emph{charge} $C(S)$ and its
\emph{energy} $H(S)$ defined by
\begin{equation}
  \label{eq:chdef}
  C(S) := |S_+| - |S_-|, \qquad
  H(S) := \sum_{\substack{k \in S_+}} k - \sum_{\substack{k \in S_-}} k.
\end{equation}
Clearly the energy is nonnegative and vanishes if and only if $S=\Z'_{< 0}$.
The set of partitions being denoted $\Par$, there is a well-known
bijection between $\s$ and $\Par \times \Z$ (the ``combinatorial
boson--fermion correspondence''): to each partition $\lambda \in \Par$
and integer $c \in \Z$, we associate the admissible set
\begin{equation}
  \mathfrak{S}(\lambda,c):=\{ \lambda_i - i + 1/2 + c ,\ i \geq 1\}.
\end{equation}
It is not difficult to see that $\mathfrak{S}(\lambda,c)$ has charge
$c$ and energy $|\lambda|+c^2/2$.

\paragraph{Fock space and fermionic operators.}

The \emph{fermionic Fock space}, denoted $\F$, is the infinite-dimensional
Hilbert space spanned by the orthonormal basis $\ket{S}, S
\in \s$. Here we use the bra--ket notation and will denote by
$\bra{\cdot}$ dual vectors. We may think of a basis vector $\ket{S}$
as the semi-infinite wedge product
\begin{equation}
  \ket{S} = \underline{s_1} \wedge \underline{s_2} \wedge \underline{s_3} \wedge \cdots
\end{equation}
where $s_1 > s_2 > s_3 > \cdots$ are the elements of $S$, and
$\{\underline{k},\ k \in \Z'\}$ is an orthonormal basis of some
smaller ``one-particle'' vector space. For $\lambda$ a partition and
$c$ an integer we introduce the shorthand notations
\begin{equation}
  \ket{\lambda,c} := \ket{\mathfrak{S}(\lambda,c)}, \qquad
  \ket{\lambda} := \ket{\lambda,0}, \qquad
  \ket{c} := \ket{\emptyset,c}.
\end{equation}
The vector $\vv$ is called the \emph{vacuum}. The charge and energy
naturally become diagonal operators acting on $\F$, which we still
denote by $C$ and $H$ respectively.
We also denote by $R$ the
\emph{shift operator} such that $R \ket{S} = \ket{S+1}$ (i.e.\ all
elements of $S$ are incremented by $1$).

We now define the \emph{fermionic operators}: for $k \in \Z'$, let us
define the operators $\psi_k$ and $\psi_k^*$ by
\begin{equation}
  \label{eq:psidef}
  \psi_k \ket{S} :=
  \begin{cases}
    0, & \text{if $k \in S$} \\
    (-1)^j \ket{S \cup \{k\}}, & \text{if $k \notin S$} \\
  \end{cases}, \qquad
  \psi_k^* \ket{S} :=
  \begin{cases}
    (-1)^j \ket{S \setminus \{k\}},  & \text{if $k \in S$} \\
    0, & \text{if $k \notin S$} \\
  \end{cases}
\end{equation}
where $j = |S \cap \Z'_{> k}|$. In the semi-infinite wedge picture,
$\psi_k$ corresponds to the exterior multiplication by $\underline{k}$
on the left, and $\psi_k^*$ to its adjoint operator. They satisfy the
canonical anticommutation relations
\begin{equation}
  \label{eq:car}
  \{ \psi_k, \psi_\ell^* \} = \delta_{k,\ell}, \qquad
  \{ \psi_k, \psi_\ell \} = \{ \psi_k^*, \psi_\ell^* \} = 0, \qquad
  k,\ell \in \Z'
\end{equation}
where $\{a,b\}:=ab+ba$. We also define the generating series
\begin{equation}
 \label{eq:psigendef}
 \psi(z) := \sum_{k \in \Z'} \psi_k z^k, \qquad
 \psi^*(w) := \sum_{k \in \Z'} \psi^*_k w^{-k}.
\end{equation}
Observe that $\psi_k \vv = \psi_{-k}^* \vv = 0$ for $k<0$. We now
recall Wick's lemma in a form suitable for future generalizations --- see
for instance \cite[Appendix B]{bbccr} for a proof.
\begin{lem}[Wick's lemma]
  \label{lem:wickorig}
  Let $\varPsi$ be the vector space spanned by (possibly infinite linear combinations of) the $\psi_k$ and $\psi_k^*$,
  $k \in \Z'$.  For $\phi_1,\ldots,\phi_{2n} \in \varPsi$, we have
  \begin{equation}
    \label{eq:wickorig}
    \vcv \phi_1 \cdots \phi_{2n} \vv = \pf A
  \end{equation}
  where $A$ is the antisymmetric matrix defined by $A_{ij}:=\vcv \phi_i
  \phi_j \vv$ for $i<j$.
\end{lem}

\paragraph{Bosonic and vertex operators.}

The \emph{bosonic operators} $\alpha_n$ are defined by
\begin{equation}
  \alpha_n := \sum_{k \in \Z'} \psi_{k-n} \psi_k^*, \qquad
  n = \pm 1, \pm 2, \ldots
\end{equation}
and $\alpha_0$ is the charge operator. We have
$\alpha_n^*=\alpha_{-n}$, $\alpha_n \vv=0$ for $n>0$, and the
commutation relations
\begin{equation} \label{eq:alphacommut}
  [\alpha_n,\alpha_m] = n \delta_{n,-m},  \qquad
  [\alpha_n,\psi(z)] = z^n \psi(z), \qquad
  [\alpha_n,\psi^*(w)] = - w^n \psi^*(w).
\end{equation}
For $\rho$ a specialization of the algebra of symmetric functions, we
define the \emph{vertex operators} $\Gamma_\pm(\rho)$ by
\begin{equation}
  \Gamma_\pm(\rho) := \exp \left( \sum_{n \geq 1} \frac{p_n(\rho) \alpha_{\pm n}}{n} \right).
\end{equation}
When
$x$ is a variable, we denote by $\Gamma_\pm(x)$ (resp.\
$\Gamma'_\pm(x)$) the vertex operators for the specialization in the
single variable $x$ (resp.\ its dual $\bar{x}$), for which
$p_n(x)=x^n$ (resp.\ $p_n(\bar{x})=(-1)^{n-1} x^n$). Clearly,
$\Gamma_-(\rho)$ is the adjoint of $\Gamma_+(\rho)$ for any real
$\rho$, and
\begin{equation}
  \label{eq:gamcancel}
  \Gamma_+(\rho) \vv = \vv, \qquad \vcv \Gamma_-(\rho) = \vcv.
\end{equation}
Given two specializations $\rho,\rho'$, as
$p_n(\rho \cup \rho')=p_n(\rho) + p_n(\rho')$, we have
\begin{equation}
  \label{eq:gambranch}
  \Gamma_+(\rho) \Gamma_+(\rho') = \Gamma_+(\rho \cup \rho') =
  \Gamma_+(\rho') \Gamma_+(\rho).
\end{equation}
The commutation relations \eqref{eq:alphacommut} and the Cauchy
identity \eqref{eq:hdef} imply that
\begin{equation}
  \label{eq:gamcomm}
  \Gamma_+(\rho) \Gamma_-(\rho') = H(\rho;\rho')
  \Gamma_-(\rho') \Gamma_+(\rho)
\end{equation}
while
\begin{equation}
  \label{eq:gampsi}
  \Gamma_\pm(\rho) \psi(z) = H(\rho;z^{\pm 1}) \psi(z) \Gamma_\pm(\rho), \qquad
  \Gamma_\pm(\rho) \psi^*(w) = H(\rho;w^{\pm 1})^{-1} \psi^*(w) \Gamma_\pm(\rho).
\end{equation}
These latter relations always make sense at a formal level; at an
analytic level they require that the parameter of $H(\rho;\cdot)$ be
within its disk of convergence. The crucial property of vertex
operators is that skew Schur functions arise as their matrix elements,
namely
\begin{equation}
  \label{eq:schurelem}
  \bra{\lambda,c} \Gamma_+(\rho) \ket{\mu,c'} =
  \bra{\mu,c'} \Gamma_-(\rho) \ket{\lambda,c} =
  \begin{cases}
    s_{\mu/\lambda}(\rho), & \text{if $c=c'$,} \\
    0, & \text{otherwise.}
  \end{cases}
\end{equation}
This results from \eqref{eq:gampsi}, Wick's lemma and the Jacobi--Trudi
identity. 

Finally, we will use the fact that the fermionic operators
can be reconstructed from the vertex operators and the charge and
shift operators $C$ and $R$ --- see e.g.\ \cite[Theorem~14.10]{kac} --- a fact we will refer to as the
\emph{boson--fermion correspondence}.
\begin{prop} \label{prop:boson_fermion}
  We have:
  \begin{equation} \label{eq:boson_fermion}
      \psi(z) = z^{C-\frac{1}{2}} R\, \Gami(z) \Gatpl\left(-z^{-1}\right), \qquad \psi^*(w) = R^{-1} w^{-C+\frac{1}{2}} \Gatmi(-w) \Gapl\left(w^{-1}\right).
  \end{equation}
\end{prop}

\subsubsection{Free boundary states and connection with the Schur process}
\label{sec:fermfree}

Following \cite[Section 5.3]{bcc}, given two parameters $u,v$, we
introduce the \emph{free boundary states}
\begin{equation} \label{def:free_boundary}
  \vfv := \sum_{\lambda \in \Par} v^{|\lambda|} \ket{\lambda}, \qquad
  \ufcv := \sum_{\lambda \in \Par} u^{|\lambda|} \bra{\lambda}.
\end{equation}
Both are (respectively left and right) eigenvectors of the charge
operator $C$ with eigenvalue $0$. For $u=v=0$, we recover respectively
the vacuum $\vv$ and its dual $\vcv$. The following proposition
generalizes \eqref{eq:gamcancel} to arbitrary $u,v$, and is
essentially a reformulation of~\cite[Proposition 10]{bcc}.

\begin{prop}[Reflection relations]
  We have
  \begin{equation}
    \label{eq:refl}
    \Gamma_+(\rho) \vfv = \tilde{H}(v \rho) \Gamma_-(v^2 \rho) \vfv, \qquad
    \ufcv \Gamma_-(\rho) = \tilde{H}(u \rho) \Gamma_+(u^2 \rho) \ufcv
  \end{equation}
  with $\tilde{H}$ defined as in \eqref{eq:hdef}.
\end{prop}

\begin{proof}
  When projected on the standard basis, these relations amount to the
  identity \cite[I.5, Ex.\ 27(a), (3), p.93]{mac} (specialized at
  $\rho$), which itself amounts to the Littlewood identity. See
  also~\cite{bcc} for a combinatorial proof when $\rho$ is the
  specialization in a single variable; by iteration it then holds for
  an arbitrary number of variables, hence holds for any
  specialization.
\end{proof}

Armed with all these definitions, we are now in position to make the
connection with the free boundary Schur process. The remainder of this
section is basically an adaptation of the arguments in \cite{or} (see
also \cite{bbccr}) to the case of free boundaries.

\begin{prop}
  \label{prop:fermschur}
  The partition of the free boundary Schur process is given by
  \begin{equation}
    \label{eq:Zgam}
    Z = \ufcv \Gamma_+(\rho_1^+) \Gamma_-(\rho_1^-) \cdots
    \Gamma_+(\rho_N^+) \Gamma_-(\rho_N^-) \vfv.
  \end{equation}
  For $U$ a finite subset of $\{1,\ldots,N\} \times \Z'$, the
  probability $\varrho(U)$ that the point process
  $\mathfrak{S}(\vec{\lambda})$ contains $U$ reads
  \begin{equation}
    \label{eq:ZU}
    \varrho(U) =
    \frac{Z_U}{Z}
  \end{equation}
  where $Z_U$ is obtained from the product in the right hand side of
  \eqref{eq:Zgam} by inserting, for each $(i,k) \in U$, the operator
  $\psi_k \psi_k^*$ between $\Gamma_+(\rho_i^+)$ and
  $\Gamma_-(\rho_i^-)$. (If several points of $U$ have the same
  abscissa, the $\psi_k \psi_k^*$ can be inserted in any order since
  they commute.)
\end{prop}

\begin{proof}
  This is a basic application of the transfer-matrix method: by
  \eqref{eq:schurelem}, the vertex operators $\Gamma_\pm$ can be seen
  as transfer matrices for the Schur process, and the operator
  $\psi_k \psi_k^*$ ``measures'' whether there is a point at ordinate
  $k$ (we have $\psi_k \psi_k^* \ket{S}=\ket{S}$ if $k \in S$ and $0$
  otherwise).
\end{proof}

As a useful warm-up, we may compute the partition function of the free
boundary Schur process.

\begin{proof}[Proof of Proposition~\ref{prop:fbz}]
  We apply the method introduced in \cite[Section~5.3]{bcc}, which we
  colloquially call \emph{ping-pong}.  The reader might find useful to
  consult this reference for more details.
  
  In order to evaluate the product~\eqref{eq:Zgam}, the first step
  consists in ``commuting'' the $\Gamma_+$ to the right and the
  $\Gamma_-$ to the left, using \eqref{eq:gamcomm} and
  \eqref{eq:gambranch}, to yield
  \begin{equation}
    Z = \prod_{1 \leq k \leq \ell \leq N} H(\rho_k^+;\rho_\ell^-) \times
    \ufcv \Gamma_-(\rho^-) \Gamma_+(\rho^+) \vfv
  \end{equation}
  with $\rho^\pm$ as in \eqref{eq:rhopmdef}. For $u=v=0$, i.e.\ for
  the original Schur process with ``vacuum'' boundary conditions, the
  rightmost factor is equal to $1$ by \eqref{eq:gamcancel}. For
  general $u$, the reflection relations allow to write first
  \begin{equation}
    \ufcv \Gamma_-(\rho^-) \Gamma_+(\rho^+) \vfv = \tilde{H}(u \rho^-)
    \ufcv \Gamma_+(\hat{\rho}) \vfv
  \end{equation}
  with $\hat{\rho}=\rho^+ \cup u^2 \rho^-$. For $v=0$, i.e.\ for the
  pfaffian Schur process, the rightmost factor equals $1$ and we are
  done. For $uv>0$, we need to use again the reflection relations
  infinitely many times to ``bounce'' the $\Gamma$ back and forth:
  \begin{equation}
    \ufcv \Gamma_+(\hat{\rho}) \vfv = \prod_{n \geq 1} \tilde{H}(u^{n-1}v^n
    \hat{\rho}) \times \uv.
  \end{equation}
  Here we assume that $uv<1$ so that the argument of the ``bouncing''
  $\Gamma$ tends to the zero specialization as the number of
  reflections tends to infinity. From the definition of the free
  boundary states we have
  \begin{equation}
    \label{eq:uvprod}
    \uv = \sum_{\lambda \in \Par} (uv)^{|\lambda|}
    = \prod_{n \geq 1} \frac{1}{1-(uv)^n}.
  \end{equation}
  Collecting all factors and rearranging them in a more symmetric
  manner (using the relation
  $\tilde{H}(\rho \cup \rho') =
  \tilde{H}(\rho)\tilde{H}(\rho')H(\rho;\rho')$
  and other easy properties), we end up with the desired expression
  \eqref{eq:fbz} for the partition function.
\end{proof}

We may perform a similar manipulation to rewrite $Z_U$ in
\eqref{eq:ZU}, by playing ping-pong with the $\Gamma$'s. The factors
arising from commutations between $\Gamma$'s or from reflection
relations are the same as in $Z$, and thus cancel when we normalize to
get the correlation function $\varrho(U)$. The fermionic operators
$\psi_k$ and $\psi^*_k$ get ``conjugated'' by the $\Gamma$'s crossing
them. If we list the elements of $U$ by increasing abscissa as
$(i_1,k_1),\ldots,(i_n,k_n)$, then we end up with
\begin{equation}
  \label{eq:rhopsi}
  \varrho(U) = \frac{\ufcv \Psi_{k_1}(i_1) \Psi^*_{k_1}(i_1) \cdots 
  \Psi_{k_n}(i_n) \Psi^*_{k_n}(i_n) \vfv}{\uv}
\end{equation}
where
\begin{equation}
  \Psi_k(i) := \Ad \left( \Gamma_+(\rho_i^\rightarrow) \Gamma_-(\rho_i^\leftarrow)^{-1} \right) \cdot \psi_k, \qquad \Psi^*_k(i) := \Ad \left( \Gamma_+(\rho_i^\rightarrow) \Gamma_-(\rho_i^\leftarrow)^{-1} \right) \cdot \psi^*_k.
\end{equation}
Here $\Ad$ denotes the adjoint action
\begin{equation}
  \Ad(A) \cdot B := A B A^{-1}
\end{equation}
and the specializations $\rho_i^\rightarrow$ and $\rho_i^\leftarrow$ are given by
\begin{equation}
  \label{eq:rhoarrowdef}
  \rho_i^\rightarrow := \bigcup_{\ell=1}^i \rho_i^+ \cup \bigcup_{n \geq 1} \left(u^{2n} v^{2n-2} \rho^- \cup u^{2n} v^{2n} \rho^+ \right), \qquad
  \rho_i^\leftarrow := \bigcup_{\ell=i}^N \rho_i^- \cup \bigcup_{n \geq 1} \left(u^{2n-2} v^{2n} \rho^+ \cup u^{2n} v^{2n} \rho^- \right).
\end{equation}
The intuitive meaning of all this is the following: given a fermionic
operator $\psi_k$ inserted at position $i$ in $Z_U$, the operator
$\Gamma_+(\rho_i^\rightarrow)$ corresponds to the product of all
$\Gamma_+$'s that will cross it from left to right, similarly
$\Gamma_-(\rho_i^\leftarrow)$ corresponds to all $\Gamma_-$'s that
will cross it from right to left. The operator $\Psi_k(i)$ is the
operator resulting after all these commutations have been made. Note
that the ordering of $\Gamma$'s in $\Ad$ is irrelevant since they
commute up to a scalar factor and that, by \eqref{eq:gampsi},
$\Psi_k(i)$ (resp.\ $\Psi^*_k(i)$) is a linear combination of $\psi$'s
(resp.\ $\psi^*$'s).

In the case $u=v=0$ \cite{or}, Okounkov and Reshetikhin were able to
rewrite \eqref{eq:rhopsi} as a determinant using Wick's lemma (we
obtain a pfaffian from Lemma~\ref{lem:wickorig}, but the matrix $A$
has a specific block structure so its pfaffian reduces to a
determinant of size $n$).  From there, they could conclude that the
point process $\mathfrak{S}(\vec{\lambda})$ is determinantal. This
does not extend straightforwardly in the case of free boundaries
(``naive'' generalizations of Wick's lemma are false), and we will
explain how to circumvent this problem in the next section.

\subsection{Pfaffian correlations in the presence of free boundaries}
\label{sec:fermgenwick}

\subsubsection{Extended free boundary states}

The starting point is to observe that a general basis vector
$\ket{\lambda}$ of charge $0$ can be written in the form
\begin{equation}
  \ket{\lambda} = (-1)^{j_1+\cdots+j_r+r/2} \psi_{i_1} \cdots \psi_{i_r} \psi_{j_1}^*
  \cdots \psi_{j_r}^* \vv
\end{equation}
where $i_1>\cdots>i_r$ are arbitrary positive half-integers, namely
the elements of $\mathfrak{S}(\lambda,0)_+$, and $j_1>\cdots>j_r$ are
arbitrary negative half-integers, namely the elements of
$\mathfrak{S}(\lambda,0)_-$. These numbers are closely related to the
Frobenius coordinates of $\lambda$, and $r$ is the size of the Durfee
square of $\lambda$. Multiplying by
$v^{|\lambda|}=v^{i_1+\cdots+i_s-j_1-\cdots-j_s}$ and summing over
all possible pairs of sequences, we get
\begin{equation}
  \label{eq:vfsum}
  \vfv = \left( \sum_{r \geq 0} \sum_{\substack{i_1>\cdots>i_r>0 \\ 0>j_1>\cdots>j_r}}
  \tilde{\psi}_{i_1}(v,t) \cdots \tilde{\psi}_{i_r}(v,t) \tilde{\psi}_{j_1}(v,t) \cdots \tilde{\psi}_{j_r}(v,t) \right) \vv
\end{equation}
with
\begin{equation}
  \label{eq:tpsidef}
  \tilde{\psi}_i(v,t) =
  \begin{cases}
    t^{1/2} v^i \psi_i & \text{for $i \in \Z'_{>0}$,} \\
    (-1)^{i+1/2} t^{-1/2} v^{-i}  \psi_i^* & \text{for $i \in \Z'_{<0}$.} \\
  \end{cases}
\end{equation}
Here $t$ is a parameter which has no effect in \eqref{eq:vfsum} but
will be useful in the following.  The big sum looks nasty
but it turns out that it can be generated in a rather elegant manner.

\begin{prop}
  \label{prop:eleg}
  We have
  \begin{equation}
    \vfv = \Pi_0 e^{X(v,t)} \vv
  \end{equation}
  where $\Pi_0$ denotes the projector onto the fermionic subspace of
  charge $0$, and where
  \begin{equation}
    \label{eq:Xdef}
    X(v,t) :=
    \sum_{\substack{(k,\ell) \in \Z'^2\\k>\ell}} \tilde{\psi}_k(v,t) \tilde{\psi}_\ell(v,t).
  \end{equation}
\end{prop}

\begin{rem}
  For $|v|<1$, the sum in \eqref{eq:Xdef} is convergent in the space
  of bounded operators on $\F$.
\end{rem}

Define the \emph{extended free boundary state} $\vtfv$ as
\begin{equation}
  \label{eq:vtfvdef}
  \vtfv := e^{X(v,t)} \vv.
\end{equation}
Proposition~\ref{prop:eleg} is then an immediate consequence of the
following lemma.

\begin{lem}
  \label{lem:vfexp}
  The extended free boundary state decomposes in the canonical basis
  of $\F$ as
  \begin{equation}
    \label{eq:vfexp}
    \vtfv = \sum_{\substack{S \in \s \\C(S)\text{ is even}}} t^{C(S)/2} v^{H(S)} \ket{S}.
  \end{equation}
\end{lem}

\begin{proof}
  We first observe that the $\tilde{\psi}_i$ all anticommute with one
  another, hence all the terms in $X(v,t)$ commute. Furthermore the
  square of each of these terms vanishes, hence
  \begin{equation}
    e^{X(v,t)} = \prod_{k>\ell} \left( 1 + \tilde{\psi}_k(v,t) \tilde{\psi}_\ell(v,t) \right).
  \end{equation}
  By expanding this product, we obtain a sum over all possible finite
  sets $\{(k_1,\ell_1),\ldots,(k_r,\ell_r)\}$ of pairs of elements of
  $\Z'$ such that $k_s>\ell_s$ for all $s$. Clearly, the contribution
  from sets containing twice the same element of $\Z'$ in different
  pairs vanishes, hence the sets with nonzero contribution can be
  identified with finite partial matchings
  $\{\{k_1,\ell_1\},\ldots,\{k_r,\ell_r\}\}$ of $\Z'$. Let
  $M=\{m_1,\ldots,m_{2r}\}$ be a subset of $\Z'$ with an even number
  $2r$ of elements: each one of the $(2r-1)!!$ perfect matchings of
  $M$ will then arise exactly once in the sum. Using again the
  anticommutativity of the $\tilde{\psi}_i$, we see that all these
  perfect matchings have the same contribution up to a sign. It can be
  seen that all these contributions cancel with one another, except one,
  and we arrive at
  \begin{equation}
    \begin{split}
      e^{X(v,t)} &= \sum_{r \geq 0} \sum_{
        \substack{(m_1,\ldots,m_{2r}) \in \Z'^{2r} \\
          m_1>\cdots>m_{2r}}}
      \tilde{\psi}_{m_1}(v,t) \cdots \tilde{\psi}_{m_{2r}}(v,t) \\
      &= \sum_{\substack{r',r'' \geq 0\\r'+r'' \text{even}}}
      \sum_{\substack{i_1>\cdots>i_{r'}>0 \\
          0>j_1>\cdots>j_{r''}}} \tilde{\psi}_{i_1}(v,t) \cdots
      \tilde{\psi}_{i_{r'}}(v,t) \tilde{\psi}_{j_1}(v,t) \cdots
      \tilde{\psi}_{j_{r''}}(v,t)
    \end{split}
  \end{equation}
  (the second line is a simple relabelling of the sum, where we split
  the positive and negative indices).  This expression can be compared
  with \eqref{eq:vfsum}: we now have a sum over all admissible subsets
  $S$ of $\Z'$ with an even charge (with $S_+=\{i_1,\ldots,i_{r'}\}$
  and $S_-=\{j_1,\ldots,j_{r''}\}$). Multiplying by $\vv$ and plugging
  back the definition of $\tilde{\psi}$, we obtain the wanted
  expression \eqref{eq:vfexp}.
\end{proof}

  By reindexing the sum in \eqref{eq:vfexp} as a sum over all
  partitions and all even charges, we can express $\vtfv$ in terms of
  $\vfv$ and of the shift operator $R$, namely
  \begin{equation}
    \label{eq:vtRv}
    \vtfv = \sum_{c \in 2\Z} t^{c/2} v^{c^2/2} R^{c} \vfv.
  \end{equation}
Of course we may define a dual extended free boundary state
\begin{equation}
  \label{eq:utfcvdef}
  \utfcv = \vcv e^{X(u,t)^*} = \sum_{\substack{S\text{ admissible}\\C(S)\text{ even}}} t^{C(S)/2} u^{H(S)} \bra{S} = \sum_{c \in 2\Z} t^{c/2} u^{c^2/2} \ufcv R^{-c}
\end{equation}
and its scalar product with $\vtfv$ reads
\begin{equation}
  \label{eq:utfvtf}
  \utfvtf = \left( \sum_{c \in 2\Z} t^c (uv)^{c^2/2} \right) \uv =
  \theta_3(t^2;(uv)^4) (uv;uv)^{-1}_\infty
\end{equation}
which is finite for $uv<1$ (see Appendix~\ref{sec:theta} for
reminders on theta functions and $q$-Pochhammer symbols).

\begin{rem}
  The reflection relations \eqref{eq:refl} still hold when we replace
  $\vfv$ by $\vtfv$ and $\ufcv$ by $\utfcv$ (this is immediate from
  \eqref{eq:vtRv}, \eqref{eq:utfvtf} and the commutation between $R$
  and the $\Gamma$'s). The extended free boundary states also satisfy
  the remarkable \emph{fermionic reflection relations}
  \begin{equation}
    \label{eq:fermrefl}
    \psi(z) \vtfv = t^{-1} \frac{v-z}{v+z} \psi^*\left(\frac{v^2}{z}\right) \vtfv, \qquad
    \utfcv \psi^*(w) = t^{-1} \frac{u-w^{-1}}{u-w^{-1}} \utfcv \psi\left(\frac{1}{u^2 w}\right).
  \end{equation}
  which can be checked using the boson--fermion correspondence
  \eqref{eq:boson_fermion}.
\end{rem}

\subsubsection{Wick's lemma for one free boundary}

The notion of extended free boundary state yields a new proof that,
for $u=0$, the point process $\mathfrak{S}(\vec{\lambda})$ is pfaffian
--- the original proof by Borodin and Rains \cite{br} relied instead on
a pfaffian analogue of the Eynard--Mehta theorem. Let us observe that,
for $u=0$, we have $\uv=1$ and the expression \eqref{eq:rhopsi} for
the correlation function $\varrho(U)$ can be rewritten as
\begin{equation}
  \varrho(U) = \vcv \Psi_{k_1}(i_1) \Psi^*_{k_1}(i_1) \cdots 
  \Psi_{k_n}(i_n) \Psi^*_{k_n}(i_n) \vtfv
  \label{eq:rhopsione}
\end{equation}
(use $\vfv=\Pi_0 \vtfv$, and drop the projector $\Pi_0$ since we
multiply on the left by a quantity of charge $0$).

\begin{lem}[Wick's lemma for one free boundary]
  \label{lem:wickone}
  Let $\varPsi$ be as before the vector space spanned by (possibly
  infinite linear combinations of) the $\psi_k$ and $\psi_k^*$,
  $k \in \Z'$.  For $\phi_1,\ldots,\phi_{2n} \in \varPsi$, we have
  \begin{equation}
    \label{eq:wickone}
    \vcv \phi_1 \cdots \phi_{2n} \vtfv = \pf A
  \end{equation}
  where $A$ is the antisymmetric matrix defined by $A_{ij}:=\vcv \phi_i
  \phi_j \vtfv$ for $i<j$.
\end{lem}

\begin{proof}
  Recall the definition \eqref{eq:vtfvdef} of $\vtfv$. We ``commute''
  the operator $e^{X(v,t)}$ to the left, where it is absorbed by the
  left vacuum $\vcv$ (since $\vcv X(v,t)=0$), and get
  \begin{equation}
    \varrho(U) = \vcv \tilde\phi_1 \cdots \tilde\phi_{2n} \vv
  \end{equation}
  where $\tilde\phi_i:=e^{-X(v,t)} \phi_i e^{X(v,t)}$.
 
  Now, the key observation is that, for any $\phi \in \varPsi$, the
  commutator $[X(v,t),\phi]$ is also in $\varPsi$ by the bilinearity
  of $X$ and the canonical anticommutation relations. Therefore,
  $\varPsi$ is stable under conjugation by $e^{-X(v,t)}$, and thus we
  may apply the usual Wick's lemma~\ref{lem:wickorig} to conclude that
  $\varrho(U)=\pf A$ with
  \begin{equation}
    A_{ij} = \vcv \tilde\phi_i \tilde\phi_j \vv = \vcv \phi_i
  \phi_j \vtfv, \qquad i<j.
  \end{equation}
\end{proof}

By applying Lemma~\ref{lem:wickone} to \eqref{eq:rhopsione}, and being
careful about the ordering between operators (note that
$\{\Psi_{k}(i),\Psi_{k'}(i')\}=0$ for all $i,i',k,k'$, while
$\{\Psi_{k}(i),\Psi^*_{k'}(i')\}=0$ for $i=i'$ and $k \neq k'$), we
find that $\mathfrak{S}(\vec{\lambda})$ is a pfaffian point process
whose correlation kernel entries are given by
\begin{equation}
  \label{eq:KSpsione}
  \begin{split}
    K_{1,1}(i, k; i', k') = \vcv \Psi_{k}(i) \Psi_{k'}(i') \vtfv,
    \quad & \qquad
    K_{2,2}(i, k; i', k') = \vcv \Psi^*_{k}(i) \Psi^*_{k'}(i') \vtfv,\\
    K_{1,2}(i, k; i', k') = - K_{2,1}(i', k'; i, k) &=
    \begin{cases}
      \vcv \Psi_{k}(i) \Psi^*_{k'}(i') \vtfv, & \text{if $i \leq i'$}, \\
      - \vcv \Psi^*_{k'}(i') \Psi_{k}(i) \vtfv, & \text{otherwise}.
    \end{cases}
  \end{split}
\end{equation}
Note that the dependency on $t$ is trivial ($K_{1,1}$, $K_{1,2}$ and
$K_{2,2}$ are respectively proportional to $t^{-1}$, $t^0$ and $t^1$),
and can be eliminated by row/column multiplications in the pfaffian,
so that the point process $\mathfrak{S}(\vec{\lambda})$ is independent
of $t$ as it should.

To complete the proof of Theorem~\ref{thm:onebound}, we need to
rewrite the entries of the correlation kernel in the form of double
contour integrals. This will be done in Section~\ref{sec:contourep},
but before we discuss the case of two free boundaries.

\subsubsection{Wick's lemma for two free boundaries}

In the case $uv>0$, when we rewrite \eqref{eq:rhopsi} in terms of the
extended free boundary states, it is no longer possible to ``drop''
the projectors $\Pi_0$ as was done in \eqref{eq:rhopsione}. But, from
\eqref{eq:vtRv} and \eqref{eq:utfcvdef}, we see that $\varrho(U)$ is
proportional to the $t^0$ term in
$\utfcv \Psi_{k_1}(i_1) \Psi^*_{k_1}(i_1) \cdots \Psi_{k_n}(i_n)
\Psi^*_{k_n}(i_n) \vtfv$. This quantity turns out to be pfaffian.

\begin{lem}[Wick's lemma for two free boundaries]
  \label{lem:wickgen}
  Let $\varPsi$ be again the vector space spanned by (possibly infinite linear combinations of) the $\psi_k$
  and $\psi_k^*$, $k \in \Z'$.  For $\phi_1,\ldots,\phi_{2n} \in
  \varPsi$ and $uv<1$, we have
  \begin{equation}
    \label{eq:wickgen}
    \frac{\utfcv \phi_1 \cdots \phi_{2n} \vtfv}{\utfvtf} = \pf A
  \end{equation}
  where $A$ is the antisymmetric matrix defined by
  $A_{ij}=\utfcv \phi_i \phi_j \vtfv/\utfvtf$ for $i<j$.
\end{lem}

\begin{proof}
  It is tempting to proceed as in the proof of
  Lemma~\ref{lem:wickone}, by commuting $e^{X(v,t)}$ to the left and
  similarly commuting $e^{X(u,t)^*}$ to the right, but those two
  quantities do not commute and it is unclear whether they have a nice
  quasi-commutation relation.

  Instead, we again play ping-pong, but this time with fermionic
  operators. Let $\varPsi^+$ (resp.\ $\varPsi^-$) be the vector space
  spanned by the $\psi_k$ and $\psi_{-k}^*$ with $k<0$ (resp.\
  $k>0$). We have $\varPsi=\varPsi^+ \oplus \varPsi^-$ and, writing
  $\phi=\phi^++\phi^-$ for the associated decomposition of
  $\phi \in \varPhi$, we have $\phi^+\vv=0$ and $\vcv \phi^-=0$.  Note
  that $X(v,t)$ is a bilinear combination of operators in $\varPsi^-$
  only (which all anticommute with one another), and it follows that,
  for any $\phi \in \varPsi$, we have
  \begin{equation}
    [\phi^-,X(v,t)]=0, \qquad [\phi^+,X(v,t)] \in \varPsi^-, \qquad
    [[\phi^+,X(v,t)],X(v,t)]=0.
  \end{equation}
  As a consequence we have
  \begin{equation}
    [\phi^+, e^{X(v,t)}] = [\phi^+,X(v,t)] e^{X(v,t)}
  \end{equation}
  and hence
  \begin{equation}
    \label{eq:phivtfv}
    \phi^+ \vtfv = [\phi^+,X(v,t)] \vtfv.
  \end{equation}
  Similarly, we have the dual relations
  \begin{equation}
    \label{eq:utfcvphi}
    \utfcv \phi^- = \utfcv [X(u,t)^*,\phi^-], \qquad [X(u,t)^*,\phi^-] \in \varPsi^+.
  \end{equation}

  We now establish \eqref{eq:wickgen} by induction on $n$. It is a
  tautology for $n=1$. Let us assume it holds up to rank $n-1$. Let
  $\phi_1,\phi_2 \cdots,\phi_{2n}$ be elements of $\varPsi$. We start by
  writing
  \begin{equation}
    \utfcv \phi_1 \phi_2 \cdots \phi_{2n} \vtfv = \utfcv \chi^{(1)} \phi_2 \cdots \phi_{2n} \vtfv
  \end{equation}
  with $\chi^{(1)}=\phi_1^+ + [X(u,t)^*,\phi_1^-] \in \varPsi^+$,
  using \eqref{eq:utfcvphi}. We then move $\chi^{(1)}$ to the right,
  using the fact that the anticommutators $\{\chi^{(1)},\phi_i\}$ are
  all scalars, to get
  \begin{equation}
    \utfcv \chi^{(1)} \phi_2 \cdots \phi_{2n} \vtfv = \sum_{i=2}^{2n} (-1)^i
    \{\chi^{(1)},\phi_i\} \utfcv \phi_2 \cdots \phi_{i-1} \phi_{i+1} \cdots \phi_{2n} \vtfv - \utfcv \phi_2 \cdots \phi_{2n} \chi^{(2)} \vtfv
  \end{equation}
  with $\chi^{(2)}=[\chi^{(1)},X(v,t)] \in \varPsi^-$, using
  \eqref{eq:phivtfv}. Now we move $\chi^{(2)}$ in the rightmost term
  to the left, picking anticommutators on the way, until it hits
  $\utfcv$ and can be transformed into
  $\chi^{(3)}=[X(u,t)^*,\chi^{(2)}] \in \varPsi^+$, which we then move
  to the right, and so on. The $\chi$'s tend to zero as we iterate,
  since we pick at least a factor $u$ or $v$ on each iteration from
  the definition of $X$. Hence we arrive at
  \begin{equation}
    \label{eq:plutfin}
    \utfcv \phi_1 \phi_2 \cdots \phi_{2n} \vtfv = \sum_{i=2}^{2n} (-1)^i
    \{\chi,\phi_i\} \utfcv \phi_2 \cdots \phi_{i-1} \phi_{i+1} \cdots \phi_{2n} \vtfv
  \end{equation}
  where $\chi=\chi^{(1)}-\chi^{(2)}+\chi^{(3)}-\cdots$. Applying this
  equality for $n=2$ and $\phi_2 \to \phi_i$ we get that
    $\utfcv \phi_1 \phi_i \vtfv = \{\chi,\phi_i\} \utfvtf$,
    and hence, by applying the induction hypothesis,
    \eqref{eq:plutfin} can be rewritten as
  \begin{equation}
    \label{eq:lutfin}
    \frac{\utfcv \phi_1 \phi_2 \cdots \phi_{2n} \vtfv}{\utfvtf} =
    \sum_{i=2}^{2n} (-1)^i A_{1i} \pf A^{(1i)}
  \end{equation}
  where $A$ is defined as in the proposition and $A^{(1i)}$ is its
  submatrix with the first and $i$-th rows and columns removed. We
  conclude by recognizing the right hand side of \eqref{eq:lutfin} is the expansion
  of the pfaffian $\pf A$ with respect to the first row/column.
\end{proof}

\begin{rem}
  Our proof relies on the fact that $X(v,t)$ and $X(u,t)^*$
  are bilinear combinations of fermionic operators. The space of such
  (not necessarily charge-preserving) bilinear combinations,
  supplemented with the identity operator $1$, forms a Lie algebra
  denoted $D'_\infty$ which is an infinite-dimensional analogue of the
  even-dimensional orthogonal Lie algebra \cite[\S 7]{jm}. It acts on
  the space $\varPsi$ of fermionic operators as the Lie algebra of the
  group of linear transformations preserving the canonical
  anticommutation relations, also known as fermionic Bogoliubov transformations.
  It would be interesting to exploit this fact to obtain a shorter
  proof of Proposition~\ref{lem:wickgen} for general $u,v$. Let us
  also mention that the generalized Wick theorem mentioned in
  \cite[Section 2.7]{az} does not apply to our situation since it
  requires the preservation of charge, and as such implies
  determinantal (as opposed to pfaffian) correlations.
\end{rem}

We now make the connection with the shifted process
$\mathfrak{S}_t(\vec{\lambda})$ of Theorem~\ref{thm:main_thm} explicit. Set
\begin{equation}
  \label{eq:rhotdef}
  \varrho_t(U) := \frac{\utfcv \Psi_{k_1}(i_1) \Psi^*_{k_1}(i_1) \cdots \Psi_{k_n}(i_n)
\Psi^*_{k_n}(i_n) \vtfv}{\utfvtf}.
\end{equation}
By \eqref{eq:vtRv} and \eqref{eq:utfcvdef}, we have
\begin{equation}
  \varrho_t(U) = \frac{1}{\utfvtf} \sum_{c \in 2\Z} t^c (uv)^{c^2/2} \ufcv \Psi_{k_1-c}(i_1) \Psi^*_{k_1-c}(i_1) \cdots \Psi_{k_n-c}(i_n) \Psi^*_{k_n-c}(i_n) \vfv
\end{equation}
where we use the fact that $R^{-1} \psi_k R = \psi_{k-1}$ and hence
$R^{-1} \Psi_k(i) R = \Psi_{k-1}(i)$ ($R$ commutes with vertex
operators). By \eqref{eq:rhopsi} and \eqref{eq:utfvtf}, we get
\begin{equation}
  \varrho_t(U) = \frac{1}{\theta_3(t^2; (uv)^4)} \sum_{c \in 2\Z} t^c (uv)^{c^2/2} \varrho\left(U - (0, c) \right) = \Prob\left( U \subset \mathfrak{S}_t(\vec{\lambda})\right). 
\end{equation}
In other words, $\varrho_t(U)$ is nothing but the correlation function
for the point process $\mathfrak{S}_t(\vec{\lambda})$. By applying
Lemma~\ref{lem:wickgen} to \eqref{eq:rhotdef}, and being again careful
about the ordering between operators, we conclude that
$\mathfrak{S}_t(\vec{\lambda})$ is indeed a pfaffian point process,
and the entries of its correlation kernel read
\begin{equation}
  \label{eq:KStpsi}
  \begin{split}
    K_{1,1}(i, k; i', k') = \frac{\utfcv \Psi_{k}(i) \Psi_{k'}(i') \vtfv}{\utfvtf},
    \quad & \qquad
    K_{2,2}(i, k; i', k') = \frac{\utfcv \Psi^*_{k}(i) \Psi^*_{k'}(i') \vtfv}{\utfvtf},\\
    K_{1,2}(i, k; i', k') = - K_{2,1}(i', k'; i, k) &=
    \begin{cases}
      \frac{\utfcv \Psi_{k}(i) \Psi^*_{k'}(i') \vtfv}{\utfvtf}, & \text{if $i \leq i'$}, \\
      - \frac{\utfcv \Psi^*_{k'}(i') \Psi_{k}(i) \vtfv}{\utfvtf}, & \text{otherwise}.
    \end{cases}
  \end{split}
\end{equation}
Note that we recover \eqref{eq:KSpsione} in the case $u=0$.

\subsection{Contour integral representations of the correlation functions}
\label{sec:contourep}

\subsubsection{Correlation kernels}
\label{sec:corrker}

Having proved the pfaffian nature of the point process
$\mathfrak{S}_t(\vec{\lambda})$ (which coincides with
$\mathfrak{S}(\vec{\lambda})$ in the case of one free boundary), the last
step to establish Theorem~\ref{thm:main_thm} (and
Theorem~\ref{thm:onebound}) is to show that the entries of the
correlation kernel~\eqref{eq:KStpsi} match their announced
expressions.

\paragraph{Integral representation of $\Psi_k(i)$ and $\Psi^*_k(i)$.}

Following~\cite{or}, we pass to the fermion generating functions
$\psi(z)$ and $\psi^*(w)$ introduced in \eqref{eq:psigendef}. Using
\eqref{eq:gampsi}, we get that
\begin{equation} \label{eq:psikcont}
  \begin{split}
    \Psi_k(i) &= [z^k] \Ad \left( \Gamma_+(\rho_i^\rightarrow)
      \Gamma_-(\rho_i^\leftarrow)^{-1} \right) \cdot \psi(z) =
    \frac{1}{2\im\pi} \oint_{R^{-1}<|z|<R} \frac{\dx z}{z^{k+1/2}}
    \frac{H(\rho_i^\rightarrow;z)}{H(\rho_i^\leftarrow;z^{-1})}
    \psi(z),\\
    \Psi^*_k(i) &= [w^{-k}] \Ad \left( \Gamma_+(\rho_i^\rightarrow) \Gamma_-(\rho_i^\leftarrow)^{-1} \right) \cdot \psi^*(w) =
    \frac{1}{2\im\pi} \oint_{R^{-1}<|w|<R} \frac{\dx w}{w^{-k+1/2}}
    \frac{H(\rho_i^\leftarrow;w^{-1})}{H(\rho_i^\rightarrow;w)} \psi^*(w)
  \end{split}
\end{equation}
where $[z^k]$ and $[w^{-k}]$ denote coefficient extractions in the
Laurent series to the right, which we may represent as contour integrals by our analyticity assumptions.
Noting that
$H(\rho_i^\rightarrow;z)/H(\rho_i^\leftarrow;z^{-1})$ is nothing but
$F(i,z)$ as defined in Theorem~\ref{thm:main_thm},
we may plug these expressions into~\eqref{eq:KStpsi}, and get
the desired contour integral representation \eqref{eq:Kint} with
\begin{equation}
  \label{eq:kappapsi}
  \begin{array}{rl}
    \kappa_{1,1}(z,w) &=  \frac{\utfcv \psi(z) \psi(w) \vtfv}{\utfvtf}, \\
    \kappa_{2,2}(z,w)& =  \frac{\utfcv \psi^*(z) \psi^*(w) \vtfv}{\utfvtf},\\
  \end{array}
  \qquad
  \kappa_{1,2}(z,w) =
  \begin{cases}
    \frac{\utfcv \psi(z) \psi^*(w) \vtfv}{\utfvtf}, & \text{if $|z|>|w|$}, \\
    - \frac{\utfcv \psi^*(w) \psi(z) \vtfv}{\utfvtf}, & \text{if
      $|z|<|w|$}.
  \end{cases}
\end{equation}
However, there are two possible convergence issues to consider:
\begin{itemize}
\item For $\kappa_{1,2}$ we should be careful that the product
  $\psi(z) \psi^*(w)$ (resp.\ $\psi^*(w) \psi(z)$) makes sense as an
  operator on $\F$ only for $|z|>|w|$ (resp.\ $|z|<|w|$), as otherwise
  its diagonal entries are infinite. Thus, to obtain a correct double
  contour integral representation for $K_{1,2}(i,k;i',k')$, we should
  integrate $z$ over a circle of radius $r$, and $w$ over a circle of
  radius $r'$, with $r>r'$ if $i \leq i'$ and $r<r'$ otherwise. Of
  course this choice of contours also works for $\kappa_{1,1}$ and
  $\kappa_{2,2}$ where the nesting condition is not necessary.
\item The second issue is specific to the case of free boundaries: we
  should make sure that the action of $\psi(z)$ or $\psi^*(z)$ on the
  extended free boundary states is well-defined. It can be seen that
  this requires $v < |z| < u^{-1}$. Intuitively speaking, the
  probability that $\vfv$ and $\vtfv$ (resp.\ $\ufcv$ and $\utfcv$)
  have an ``excitation'' at level $k$ decays as $v^{|k|}$ (resp.\
  $u^{|k|}$), and the action of $\psi(z)$ and $\psi^*(z)$ does not
  blow up if and only if $|z|^{-1}<v$ (resp.\ $|z|<u$). Thus, for the
  double contour integrals to make sense, we should take the radii
  $r,r'$ between $v$ and $u^{-1}$.
\end{itemize}
This explains the constraints on the integration radii $r,r'$ in
Theorem~\ref{thm:main_thm}.

\begin{rem}
  Using the canonical anticommutation relations, it is possible to see
  that there is a single meromorphic function $\kappa_{1,2}(z,w)$ with
  a pole at $z=w$, and that its two expressions in \eqref{eq:kappapsi}
  correspond to Laurent expansions in different annuli. Furthermore,
  the fermionic reflection relations~\eqref{eq:fermrefl} imply that
  \begin{equation}
    \label{eq:kapparel}
    \kappa_{1,2}(z,w) = t \frac{v+w}{v-w} \kappa_{1,1}\left(z,\frac{v^2}{w}\right)
    = t^{-1} \frac{z-v}{z+v} \kappa_{2,2}\left(\frac{v^2}{z},w\right).
  \end{equation}
\end{rem}

\paragraph{Evaluation of fermionic propagators.} We now turn to the
evaluation of the $\kappa$'s. In the case of one free boundary
($u=0$), the computations are rather easy, for instance to compute
$\kappa_{1,2}(z,w) = \vcv \psi(z) \psi^*(w) \vtfv$ for $|z|>|w|$ we
use \eqref{eq:vfexp} and notice that the only $\ket{S}$ that
contribute have charge $0$ and correspond to ``hook''
partitions. Taking into account signs and the special contribution from
$\vv$, we get
\begin{equation}
  \label{eq:eval_prop_1fb}
  \kappa_{1,2}(z,w) = \sum_{k \in \in \Z'_{<0}} (-1)^{k+1/2} \frac{z^k}{v^k} \sum_{\ell \in \Z'_{>0}} \frac{w^\ell}{v^\ell} + \sum_{k \in \Z'_{<0}} \frac{z^k}{w^k} = \frac{(zw-v^2)\sqrt{zw}}{(z+v)(w-v)(z-w)}.
\end{equation}
For $v=1$, we obtain the expression announced in \eqref{eq:kappasimp},
and we leave the reader derive similarly the expressions for the other
propagators, given in Appendix~\ref{sec:fermionic_expectations}, to
conclude the proof of Theorem~\ref{thm:onebound}.

We may adapt this approach to the case of two free boundaries, but the
computations become involved. A simpler approach is to use
Proposition~\ref{prop:boson_fermion} to rewrite the fermions in terms
of vertex operators. For instance, to evaluate $\kappa_{1,2}(z,w)$ for
$|z|>|w|$, we write
\begin{equation}
  \begin{split}
    \utfcv \psi(z) \psi^*(w) \vtfv &= \utfcv (z/w)^{C-1/2} \Gami(z)
    \Gatpl\left(-z^{-1}\right) \Gatmi(-w) \Gapl\left(w^{-1}\right)
    \vtfv \\
    &= \sqrt{\frac{w}{z}} \left( \sum_{c \in 2\Z} (tz/w)^c (uv)^{c^2/2} \right)
    \ufcv \Gami(z)
    \Gatpl\left(-z^{-1}\right) \Gatmi(-w) \Gapl\left(w^{-1}\right) \vfv
  \end{split}
\end{equation}
where we use \eqref{eq:vtRv}, \eqref{eq:utfcvdef} and the commutation
of vertex operators with $R$ and $C$. The last factor on the second
line can be identified with the partition function of a certain free
boundary Schur process of length $2$, compare with
Proposition~\ref{prop:fermschur}. Using Proposition~\ref{prop:fbz} to
evaluate this partition function, and recognizing several Pochhammer
symbols and theta functions, we arrive at
\begin{equation}
  \label{eq:eval_prop_2fb}
  \utfcv \psi(z) \psi^*(w) \vtfv = \sqrt{\frac{w}{z}}
  \theta_3 \left( (t z/w)^2 ; (uv)^4 \right) 
  \frac{((uv)^2;(uv)^2)_\infty^2 \theta_{(uv)^2}(u^2 z w) \uv}{(uv,uz,-uw,-v z^{-1},v w^{-1};uv)_\infty \theta_{(uv)^2}(w/z)}.
\end{equation}
Upon dividing by the normalization \eqref{eq:utfvtf}, we obtain the
expression of $\kappa_{1,2}(z,w)$ announced in \eqref{eq:Fkap2b}. The
expressions for the other propagators, given in
Appendix~\ref{sec:fermionic_expectations}, can be checked using
\eqref{eq:kapparel} and simple manipulations of Pochhammer symbols and
theta functions. This concludes the proof of
Theorem~\ref{thm:main_thm}.

\subsubsection{Variations on free boundary states} \label{sec:variations}

In this section, we explain how to handle the extra weighting
$\alpha^{oc}$ mentioned in Remark~\ref{rem:oc}.  Recall that
$\ket{\underline{v}} = \sum_{\lambda} v^{|\lambda|} \ket{\lambda}$. We
define the following companion boundary vectors:
\begin{equation} \label{def:variation_free_boundary}
\begin{aligned}
  \ket{\underline{v^{ec}}} &= \sum_{\lambda:\ \lambda \text{ has even columns}} v^{|\lambda|} \ket{\lambda}, &
  \ket{\underline{v^{\alpha, \ oc}}} &= \sum_{\lambda} \alpha^{\# \text{ of odd columns of }\lambda} v^{|\lambda|} \ket{\lambda}, \\
  \ket{\underline{v^{er}}} &= \sum_{\lambda:\ \lambda \text{ has even rows}} v^{|\lambda|} \ket{\lambda}, &
  \ket{\underline{v^{\beta, \ or}}} &= \sum_{\lambda} \beta^{\# \text{ of odd rows of }\lambda} v^{|\lambda|} \ket{\lambda}
\end{aligned}
\end{equation}
where $ec, er, oc, or$ stand for respectively \emph{even columns, even
  rows, odd columns, odd rows}, and where $\alpha, \beta$ are
parameters.  Note that
$\ket{\underline{v^{1, \ oc}}} = \ket{\underline{v^{1, \ or}}} =
\ket{\underline{v}}$ while
$\ket{\underline{v^{0, \ oc}}} = \ket{\underline{v^{ec}}}$ and
$\ket{\underline{v^{0, \ or}}} =
\ket{\underline{v^{er}}}$. Analogously we may define covectors
$\bra{\underline{u^{ec}}}$, etc.

\begin{prop}
  We have
  \begin{equation}
    \ket{\underline{v^{\alpha, \ oc}}} = \Gami(\alpha v) \ket{\underline{v^{ec}}},\qquad
    \ket{\underline{v^{\beta, \ or}}} = \Gatmi(\beta v) \ket{\underline{v^{er}}}.    
  \end{equation}
\end{prop}

\begin{proof}
  The first (resp.\ second) identity results from the fact that any
  partition can be decomposed uniquely into a partition with even
  columns (resp.\ rows), and a horizontal (resp.\ vertical)
  strip.
\end{proof}

Noting that $\Gami(x) \Gatmi(-x)=1$ for any single variable $x$, we
deduce that the modified free boundary states can be expressed in
terms of the original one as
\begin{equation}
  \label{eq:variation_normal_free_boundary}
  \ket{\underline{v^{\alpha, \ oc}}} = \Gami(\alpha v) \Gatmi(-v) \vfv,\qquad  
  \ket{\underline{v^{\beta, \ or}}} = \Gatmi(\beta v) \Gami(-v) \vfv.
\end{equation}

Replacing $\vfv$ by $\ket{\underline{v^{\alpha, \ oc}}}$ in
Proposition~\ref{prop:fermschur}, we obtain the partition function and
correlation functions for the free boundary Schur process with the
extra weight $\alpha^{oc}$ counting the number of odd columns of
$\mu^{(N)}$. By \eqref{eq:variation_normal_free_boundary}, we readily
see that such modification amounts to replacing the boundary
specialization $\rho_N^-$ by the specialization $\rho_N^{-}[\alpha,v]$
such that
\begin{equation}
  H(\rho_N^{-}[\alpha,v];t) = \frac{1 - t v}{1 - \alpha t v} H(\rho_N^{-};t).
\end{equation}
In other words, we are ``adding'' a specialization in the single
variable $\alpha v$, and ``subtracting'' (in the sense of plethystic
negation) a variable $v$.  Note that $\rho_N^{-}[\alpha,v]$ is a
priori not nonnegative. Still, this allows to deduce easily the
partition function and correlations functions of the modified process
from Proposition~\ref{prop:fbz} and
Theorem~\ref{thm:main_thm}.
In \eqref{eq:Fkap2b}, changing
$\rho_N^-$ into $\rho_N^{-}[\alpha,v]$ produces some extra factors in
$F(i,z)$. These factors will appear in the contour integral
representation \eqref{eq:Kint}, once for $F(i,z)$ and once for
$F(i',w)$. We may conventionally choose to absorb them in a
redefinition of the $\kappa$'s, and keep $F$ unchanged: this yields
the announced expression \eqref{eq:kapoc} for $u=0$ and $v=1$. It is
not difficult to see that such a redefinition of the $\kappa$'s amounts
to replacing in \eqref{eq:kappapsi} the extended free boundary state $\vtfv$ by
\begin{equation}
  \label{eq:variation_normal_free_boundaryt}
  \ket{\underline{v^{\alpha, \ oc}, t}} := \Gami(\alpha v) \Gatmi(-v) \vtfv,
\end{equation}
i.e. we modify the fermionic propagators rather than the boundary
specialization $\rho_N^-$.

Of course, we could perform a similar trick to introduce instead an
extra weight $\beta^{or}$ counting the number of odd rows of
$\mu^{(N)}$. Let us record the corresponding redefinition of the
$\kappa$'s for $u=0$ and $v=1$:
\begin{equation}
  \label{eq:kapor}
  \begin{aligned}
    \kappa_{1,1}(z,w) &= \frac{\sqrt{zw} (z-w) }{ (zw-1) (z+\beta)
      (w+\beta)}, \qquad &
    \kappa_{1,2}(z,w) &= \frac{\sqrt{zw} (zw-1) (w + \beta)}{(z-w)
      (w^2-1) (z+\beta)}, \\
    \kappa_{2,2}(z,w) &= \frac{\sqrt{zw} (z-w) (z+\beta)
      (w+\beta)}{(z^2-1) (w^2-1) (zw-1)}. & &
  \end{aligned}
\end{equation}
For bookkeeping purposes, we gather in
Appendix~\ref{sec:fermionic_expectations} the expressions for all
modified fermionic propagators with one free boundary state.  The case
of two free boundaries is left as an exercise to the reader.

\subsubsection{General correlation functions}

It turns out that the free fermion formalism used in this section
allows to derive an explicit $2n$-fold contour integral representation
for the general $n$-point correlation function of both
$\mathfrak{S}(\vec{\lambda})$ and $\mathfrak{S}_t(\vec{\lambda})$.

\begin{thm}
  \label{thm:rhonfold}
  Let $U=\{(i_1,k_1),\ldots,(i_n,k_n)\}$ be a finite subset of
  $\{1,\ldots,N\} \times \Z'$, with $i_1 \leq \cdots \leq i_n$.  The
  probability $\varrho(U)$ that the point process
  $\mathfrak{S}(\vec{\lambda})$ contains $U$ reads
  \begin{equation}
    \label{eq:rhonfold}
    \varrho(U) = \frac{1}{(2\im\pi)^{2n}} \oint \cdots \oint
    \left( \prod_{j=1}^n \frac{\dx z_j \dx w_j}{z_j^{k_j+1} w_j^{-k_j+1}} \cdot
      \frac{F(i_j,z_j)}{F(i_j,w_j)} \right) \Phi(z_1,\ldots,z_n;w_1,\ldots,w_n)
  \end{equation}
  where the contour integrals are taken over $2n$ nested circles
  $\min(R,u^{-1})>|z_1|>|w_1|>\cdots>|z_n|>|w_n|>\max(v,R^{-1})$,
  and where $F$ is as in
  Theorem~\ref{thm:main_thm} while
  \begin{multline}
    \label{eq:Phiexp}
    \Phi(z_1,\ldots,z_n;w_1,\ldots,w_n) =
    \sqrt{\frac{w_1\cdots w_n}{z_1 \cdots z_n}}
    \frac{((uv)^2;(uv)^2)_\infty^{2n}}{\prod_{i=1}^n (u z_i,-u w_i,-v z_i^{-1},v w_i^{-1};uv)_\infty} \\ \times
    \frac{\prod_{i,j=1}^n \theta_{(uv)^2}(u^2 z_i w_j)}{ \prod_{1 \leq
        i \leq j \leq n} \theta_{(uv)^2}(w_j/z_i) \prod_{1 \leq i < j
        \leq n} \theta_{(uv)^2}(z_j/w_i)} \prod_{1 \leq i < j \leq n}
    \frac{\theta_{(uv)^2}(z_j/z_i) \theta_{(uv)^2}(w_j/w_i)}{
      \theta_{(uv)^2}(u^2z_iz_j) \theta_{(uv)^2}(u^2w_iw_j)}.
  \end{multline}
  The probability $\varrho_t(U)$ that the point process
  $\mathfrak{S}_t(\vec{\lambda})$ contains $U$ admits the same
  expression, upon replacing $\Phi(z_1,\ldots,z_n;w_1,\ldots,w_n)$
  by
  \begin{equation}
    \label{eq:Phitexp}
    \Phi_t(z_1,\ldots,z_n;w_1,\ldots,w_n) = \theta_3 \left( \left(t \frac{z_1 \cdots z_n}{w_1 \cdots w_n} \right)^2 ; (uv)^4 \right) \frac{\Phi(z_1,\ldots,z_n;w_1,\ldots,w_n)}{\theta_3(t^2;(uv)^4)}.
  \end{equation}
\end{thm}

\begin{proof}
  We start from the fermionic representation \eqref{eq:rhopsi} of
  $\varrho(U)$, and plug in the contour integral representation
  \eqref{eq:psikcont} of $\Psi_k(i)$ and $\Psi^*_k(i)$: we obtain the
  $2n$-fold contour integral \eqref{eq:rhonfold} with
  \begin{equation}
    \Phi(z_1,\ldots,z_n;w_1,\ldots,w_n) =
    \frac{\ufcv \psi(z_1) \psi^*(w_1) \cdots \psi(z_n) \psi^*(w_n) \vfv}{\uv}.
  \end{equation}
  This quantity may be evaluated by the same strategy as for the
  fermionic propagators in Section~\ref{sec:corrker}, by using the
  boson--fermion correspondence (Proposition~\ref{prop:boson_fermion})
  to rewrite the $\psi/\psi^*$ in terms of vertex operators (and $R$
  and $C$ operators that are immediately factored out). We recognize
  the partition function of a certain free boundary Schur process of
  length $2N$, which after some massaging yields \eqref{eq:Phiexp}.
  The discussion of integration contours is easily adapted from that
  in Section~\ref{sec:corrker}.
  
  For $\varrho_t(U)$, we proceed in the same way starting from the
  fermionic representation \eqref{eq:rhotdef}, which amounts to
  replacing $\Phi$ by
  \begin{equation}
    \Phi_t(z_1,\ldots,z_n;w_1,\ldots,w_n) =
    \frac{\utfcv \psi(z_1) \psi^*(w_1) \cdots \psi(z_n) \psi^*(w_n) \vtfv}{\utfvtf}.
  \end{equation}
  By \eqref{eq:vtRv} and \eqref{eq:utfcvdef}, we see that $\Phi_t$
  differs from $\Phi$ by a simple charge-related factor, leading to
  \eqref{eq:Phitexp}.
\end{proof}

\begin{rem}
  By Wick's lemma for two free boundaries (Lemma~\ref{lem:wickgen}),
  we have
  \begin{equation}
    \label{eq:Phitpfaff}
    \Phi_t(z_1,\ldots,z_n;w_1,\ldots,w_n) = \pf
    \left[
      \begin{array}{cc}
        \kappa_{1,1}(z_i,z_j) & \kappa_{1,2}(z_i,w_j) \\
        -\kappa_{1,2}(z_j,w_i) & \kappa_{2,2}(w_i,w_j) \\
      \end{array}
    \right]_{1 \leq i,j \leq n}
  \end{equation}
  where the $\kappa$'s are as in Theorem~\ref{thm:main_thm}. Plugging
  in the explicit expression for $\Phi_t$ given in
  Theorem~\ref{thm:rhonfold}, we obtain a remarkable pfaffian identity
  which amounts to a particular case of an identity due to Okada
  \cite{oka}.  In the case $u=v=0$, this identity reduces to the
  well-known Cauchy determinant. In the case $u=0$, $v=1$ of
  Theorem~\ref{thm:onebound}, we obtain an identity equivalent to
  Schur's pfaffian
  $\pf_{1 \leq i < j \leq 2n} \frac{x_i-x_j}{x_i+x_j} = \prod_{1 \leq
    i < j \leq 2n} \frac{x_i-x_j}{x_i+x_j}$ (the equivalence goes as
  follows: substitute the expression \eqref{eq:kappasimp} for the
  $\kappa$'s, pull out the trivial row/column factors and take
  $x_{2i-1}=\frac{z_i-1}{z_i+1}$, $x_{2i}=\frac{1-w_i}{w_i+1}$ for
  $i=1,\ldots,n$).
\end{rem}

\section{Symmetric Last Passage Percolation}
\label{sec:slpp}
In this section we consider the last passage percolation (LPP) time with symmetric and up to symmetry independent geometric weights. For $(a_n)_{n\geq 1} \in (0,1)^{\mathbb{N}},$  $\alpha\in(0,\inf\{\frac{1}{ a_n},n\geq 1\})$ and $r,t \in \Z_{\geq 1}$ these weights are given by
\begin{equation}\label{symweights}\omega_{r,t}=\omega_{t,r}\sim
\begin{cases}\vspace{0.3cm}
g(a_r a_t), \quad & \mathrm{if}\quad r\neq t, \\
g(\alpha a_r), \quad& \mathrm{if}\quad r=t
\end{cases}
\end{equation}
where $\Prob(g(q)=k)=q^{k}(1-q)$ for $k\in \Z_{\geq 0}$.

For $(k,l), (m,n) \in \Z^{2}_{\geq 1}$ with $k\leq m, l\leq n$, consider up-right paths $\pi$ from $(k,l)$ to $(m,n)$, i.e. $\pi =(\pi(0), \pi(1),\ldots, \pi(m-k+n-l))$ with $\pi(0)=(k,l), \pi(m-k+n-l)=(m,n)$ and
$\pi(i)-\pi(i-1)\in \{(0,1),(1,0)\}.$ The symmetric LPP time with geometric weights \eqref{symweights} is then defined to be
\begin{equation}\label{symLPP}
L_{(k,l)\to (m,n)}:= \max_{\pi :(k,l)\to (m,n)}\sum_{(r,t)\in \pi} \omega_{r,t}
\end{equation}
where the maximum is taken over all up-right paths from $(k,l) $ to $(m,n)$. Note that we have the recursion
\begin{equation}\label{LPPrec}
L_{(k,l)\to (m,n)}=\max\{L_{(k,l)\to (m-1,n)}, L_{(k,l)\to (m,n-1)}\}+\omega_{m,n}.
\end{equation}

\subsection{Definition of distribution functions}\label{sec:defdist}

We start by defining the distribution functions which will appear later. The distributions defined below are mostly given in terms of contour integrals, which is why we make the following definition.
  For $\varphi \in [0, 2\pi]$ and $z\in \R$
 denote by $\mathfrak{G}_{z}^{\varphi}=\{  z+ |s| e^{\sgn(s) \im \varphi} , s\in \R\}$ the infinite curve oriented from $z+ \infty e^{- \im  \varphi}$ to $z+ \infty e^{ \im  \varphi}$. If $f$ is a function and $V \subset \C$, we denote by $\gamma_V$  any counterclockwise oriented simple closed curve containing all elements of $V$ in its interior and excluding all poles of $f$ that are not elements of $V$.

 Let $u_1>u_2 > \cdots >u_k \geq 0$ and $a,b\in \{1,\ldots,k\}$. For $v\in \R$ we define
 \begin{equation}
 K_{1,1}^{v}(u_a,\xi ; u_b,\xi')=\frac{1}{(2 \pi \im)^{2}}\int_{\mathfrak{G}_{1}^{\pi/3}}\dx Z\int_{\mathfrak{G}_{1}^{\pi/3}} \dx W \frac{(Z-W)(W+2v)(Z+2v)}{4ZW(Z+W)}\frac{e^{Z^{3}/3-Z^{2}u_{a}-Z\xi}}{e^{-W^{3}/3+W^{2}u_b+W\xi'}}
 \end{equation}
 and $K_{1,2}^{v}=K_{1,2}^{v,1}+K_{1,2}^{v,2},$ where
 \begin{align}\label{K12v}
 K_{1,2}^{v,1}(u_a,\xi; u_b,\xi')=\frac{1}{(2 \pi \im)^{2}}\int_{\mathfrak{G}_{A_Z}^{\pi/3}}\dx Z
 \int_{\mathfrak{G}_{A_W}^{2\pi/3}}\dx W\frac{(Z+W)(Z+2v)}{2(W+2v)Z(Z-W)}
 \frac{e^{Z^{3}/3-Z^{2}u_{a}-\xi Z}}{e^{W^{3}/3-W^{2}u_b-W\xi'}},
 \end{align}
 with $A_{Z}>A_{W}>-2v,A_{Z}>0$. For $a \leq b,$ we have $ K_{1,2}^{v,2}(u_a, \xi; u_b,\xi')=0,$ and if $a >b$, then 
 \begin{align}
  K_{1,2}^{v,2}(u_{a},\xi; u_{b},\xi')=\frac{-1}{2 \pi \im}\int_{\im\R}\dx Z e^{Z^{2}(u_b -u_{a})+Z(\xi' -\xi)}
 \end{align} 
 \noindent with $\im\R$ oriented with increasing imaginary part.
 Finally, we define $K_{2,2}^{v}=K_{2,2}^{v,1}+K_{2,2}^{v,2}$ through
 \begin{align}
 K_{2,2}^{v,1}(u_{a},\xi;u_b,\xi')=\frac{1}{(2 \pi \im)^{2}}\int_{\mathfrak{G}_{B_3}^{2\pi/3}}\dx Z\int_{\mathfrak{G}_{B_4}^{2\pi/3}}\dx W\frac{Z-W}{(W+2v)(Z+2v)(Z+W)}\frac{e^{-Z^{3}/3+u_{a}Z^{2}+Z\xi}}{e^{W^{3}/3-u_bW^{2}-W\xi' }},
 \end{align}
 \noindent with $B_3>-2v>B_4, B_3 <-B_4$.
  
 We define $K_{2,2}^{v,2}$  for $u_a=u_{b}=0$ to be 
 \begin{equation}
K_{2,2}^{v,2}(0,\xi;0,\xi')= \frac{e^{8v^{3}/3-2v\xi' }}{2\pi \im}\int_{\mathfrak{G}_{C_{1}}^{2 \pi/3}}\dx Z \frac{e^{-Z^{3}/3+Z\xi}}{Z-2v}+\sgn(\xi'-\xi)e^{-2v|\xi-\xi' |}
 \end{equation} 
 with $C_1 <2v,$ whereas if $u_a+u_{b}>0$ 
 \begin{align}
  \begin{split}
 K_{2,2}^{v,2}(u_a,\xi;u_b,\xi') &= \frac{e^{8 v^{3}/3+4v^{2}u_b -2v \xi' }}{2 \pi \im}\int_{\mathfrak{G}_{B_2}^{2\pi/3}}\dx Z\frac{1}{Z-2v}e^{-Z^{3}/3+Z^{2}u_{a}+\xi Z } \\
 &-\frac{1}{2 \pi \im}\int_{\mathfrak{G}^{2\pi/3}_{B_1}}\dx Z\frac{2Z}{(Z+2v)(Z-2v)}e^{Z^{2}(u_a+u_b)+Z(\xi-\xi')}
  \end{split}
\end{align}
with $B_1> 2|v|, $ $B_2>2v$.
We can now define the following   antisymmetric kernel, note that we introduce the prefactor $d_q=\frac{q^{1/6}}{2(1+\sqrt{q})^{2/3}}$ in its definition so we do not have to insert it later:
\begin{align}\label{Kv} K^{v}(u_a, \xi; u_b,\xi')=  \begin{pmatrix*}[l]
  K_{1,1}^{v}( d_q  u_a, \xi;  d_q  u_b,\xi' )        &   K_{1,2}^{v}( d_q u_a, \xi;  d_q u_b,\xi')  \\
-K_{1,2}^{v}( d_q u_b,\xi' ; d_q u_a, \xi)  &  K_{2,2}^{v}( d_q u_a, \xi;  d_q u_b,\xi')
  \end{pmatrix*}.
  \end{align}

  The Tracy--Widom GUE distribution is given by
  \begin{align}
  F_{\GUE}(s)=\pf(J-\hat{K}_{\mathcal{A}_2})_{(s,\infty)}
  \end{align}
with $\hat{K}_{\mathcal{A}_2}(x,y)=  \begin{pmatrix*}[l]
  0       &   K_{\mathcal{A}_2}(x,y)  \\
-K_{\mathcal{A}_2}(y,x) &  0
  \end{pmatrix*}
$
and  
\begin{align}
  K_{\mathcal{A}_2}(x,y)=\frac{1}{(2\pi \im)^{2}}\int_{\mathfrak{G}^{2 \pi/3}_{-1}}\dx z \int_{\mathfrak{G}^{ \pi/3}_{1}}\dx w
  \frac{e^{w^{3}-wx}}{e^{z^{3}-zy}}\frac{1}{w-z} 
\end{align}
\noindent is the Airy kernel.
  We define $F_{u,v}$ through
  \begin{align}
  F_{u,v}(s)=\pf(J-K^{v}( u,\xi; u,\xi'))_{(s,\infty)}.
  \end{align}
  The $F_{\GOE}$ and $F_{\GSE}$ distributions which appear in the following can be defined through Fredholm pfaffians --- see e.g. Lemmas 2.6, 2.7 in \cite{BBCS17}, but their explicit form will not be needed later and hence we omit giving it. $F_{u,v}$ interpolates between various distribution functions. First, we have $ F_{0,0}(s)=F_{\GOE}(s)$; the equivalence of $F_{0,0}$ with existing definitions of $F_{\GOE}$ was checked in e.g. Lemma 2.6 of \cite{BBCS17}. It follows from our Theorem \ref{LPPThm2} and (4.26)   of \cite{br99} that $F_{0,v}(s)=F^{\boxslash}(s;v)$,  where $F^{\boxslash}$ is defined in Definition 4 of \cite{br2}. This and (2.33) in \cite{br2} imply  $\lim_{v\to +\infty}F_{0,v}(s)=F_{\GSE}(s).$ Finally, if $v=v(u)$ is such that $u+2v(u)\to+\infty$ for $u \to +\infty,$ 
 then $\lim_{u \to +\infty}F_{u,v}(s-u^{2}d_{q}^{2})=F_{\GUE}(s)$ (with $ d_q=\frac{q^{1/6}}{2(1+\sqrt{q})^{2/3}}$). This follows from the convergence of $K^{v}$ (under conjugation) to $\hat{K}_{\mathcal{A}_2}$ and dominated convergence.

\subsection{Results and proofs}
The first result we present is a formula for the multipoint distribution of LPP times along down-right paths.
\begin{thm} \label{LPPmulti} Consider the LPP time \eqref{symLPP} with weights \eqref{symweights}.
Let  $r_l,t_l\in \Z_{\geq 1},r_l \leq t_l,l=1,\ldots,k$ with  $r_1 \leq \ldots \leq r_k , t_1 \geq \ldots \geq t_k .$  Then
\begin{align}\label{labelneeded}
\Prob\left(\bigcap_{l=1}^{k}\{L_{(1,1)\to (r_l,t_l) }\leq s_l \} \right)=\pf(J-K)_{B}
\end{align}
where $B=\{(i,x)\in \{1,\ldots,k\}\times \Z^{\prime}: x>s_i -1/2\}$ is equipped with the counting measure. The kernel $K$ is given by 
\begin{equation}
\begin{split}
K_{1,1}(a,x_a ; b,x_b) =& \frac{1}{(2\pi \im)^{2}}\oint \dx z \oint \dx w\frac{z-w}{(z^{2}-1)(w^{2}-1)(zw-1)}z^{-x_a +1/2}w^{-x_b +1/2}(1-\alpha/z)(1-\alpha/w) \\
& \times \prod_{i=1}^{r_a} \frac{1}{1-a_i z} \prod_{i=1}^{r_b} \frac{1}{1-a_i w} \prod_{i=1}^{t_a} (1-a_i /z) \prod_{i=1}^{t_b} (1-a_i/w)
\end{split}
\end{equation}
for counterclockwise oriented circle contours around $0$ satisfying $\alpha,1<|z|,|w|<\min_{i=1,\ldots, t_1} \frac{1}{a_i} $;
\begin{equation}
\begin{split}
K_{1,2}(a,x_a;b,x_b) =& \frac{1}{(2\pi \im)^{2}}\oint \dx z \oint \dx w\frac{zw-1}{(z-w)(z^{2}-1)}z^{-x_a +1/2} w^{x_b -3/2}\frac{1-\alpha/z}{1-\alpha/w} \\
& \times \prod_{i=1}^{r_a}\frac{1}{1-a_i z} \prod_{i=1}^{r_b} (1-a_i w) \prod_{i=1}^{t_a} (1-a_i /z) \prod_{i=1}^{t_b}\frac{1}{1-a_i /w}
\end{split}
\end{equation}
for counterclockwise oriented circle contours around $0$ with $\max_{i=1,\ldots,t_1} a_i ,\alpha < |z|,|w|< \min_{i=1,\ldots,t_1}\frac{1}{a_i}$ and, if $a\leq b$, $|w|<|z| >1,$
and, if $b<a$,  $1<|z| <|w|,$; and finally
\begin{equation}
\begin{split}
K_{2,2}(a,x_a;b,x_b) =& \frac{1}{(2\pi \im)^{2}}\oint \dx z \oint \dx w z^{x_a -3/2}w^{x_b -3/2}\frac{z-w}{zw-1}\frac{1}{(1-\alpha/z)(1-\alpha/w)} \\
&\times \prod_{i=1}^{r_a} (1-a_i z) \prod_{i=1}^{r_b} (1-a_i w) \prod_{i=1}^{t_a} \frac{1}{1-a_i /z} \prod_{i=1}^{t_b}\frac{1}{1-a_i /w}
\end{split}
\end{equation}
for counterclockwise oriented circle contours around $0$ with $\alpha,\max_{i=1,\ldots,t_1} a_i < |z|,|w| <  \min_{i=1,\ldots,t_1} \frac{1}{a_i}$ and $1<|zw|.$
\end{thm}

The following theorem will be obtained from the previous one by asymptotic analysis.
\begin{thm}\label{LPPThm2}
Consider the weights \eqref{symweights} with $a_j=\sqrt{q},q\in(0,1),j\geq 1$ and $\alpha=1-2vc_{q}N^{-1/3}, $ where  $c_q=\frac{1-\sqrt{q}}{q^{1/6}(1+\sqrt{q})^{1/3}},v\in\R$. Let $u_1 >\cdots >u_k \geq 0$.
Then 
\begin{equation}
  \begin{split}
\lim_{N \to \infty}\Prob\left(\bigcap_{i=1}^{k}\big\{L_{(1,1)\to (N-\lfloor u_i N^{2/3}\rfloor,N)}\leq \frac{2\sqrt{q}N}{1-\sqrt{q}}-u_{i}\frac{\sqrt{q}N^{2/3}}{1-\sqrt{q}}+c_{q}^{-1}s_i N^{1/3}\big\}\right)= \\
\qquad \qquad \qquad  = \pf( J - \chi_{s} K^{v} \chi_{s})_{\{u_1,\ldots,u_k\}\times \R}
  \end{split}
\end{equation}
where $\chi_{s}(u_i,x)=\mathbf{1}_{x>s_i}$ and $K^{v}$ is defined in \eqref{Kv}.
\end{thm}

\begin{rem}\label{LPPrem}As was already mentioned, further asymptotic regimes can be considered. The geometric LPP time $L_{(1,1)\to (\kappa N,N)}$ with weights \eqref{symweights} and $a_n =\sqrt{q}$  converges to $F_{\GUE}$   as long as $\alpha< \frac{1+\sqrt{\kappa q}}{\sqrt{\kappa}+\sqrt{q}}$. Also, as was already obtained in \cite{br99},  $L_{(1,1)\to ( N,N)}$ (rescaled)  converges to $F_{\GSE}$ for $\alpha<1$. We note that this case is more involved than the others, see Section 5 of \cite{BBCS17} for a discussion and solution of the arising difficulties in the exponential case. Finally, the law of large numbers limit of $L_{(1,1)\to ( N,N)}$ changes for $\alpha>1$ and  $L_{(1,1)\to ( N,N)}$ converges to a Gaussian random variable under $N^{1/2}$ scaling.
\end{rem}
\subsubsection{Proofs of Theorems~\ref{LPPmulti} and \ref{LPPThm2} }\label{Proofs}

We start by proving Theorem~\ref{LPPmulti}. The symmetric LPP time \eqref{symLPP} is a marginal of a Schur process with an \emph{even columns} free boundary partition. This can be seen using the framework developed in \cite{bbbccv} which we mostly follow and refer to for further references. For a word $w=(w_1 , \ldots, w_n)\in \{\prec, \succ\}^{n}$ and a sequence of partitions $\vec{\lambda}=(\emptyset=\lambda^{(0)},\ldots, \lambda^{(n)})$ we say that $\vec{\lambda}$ is $w-$interlaced if $\lambda^{(i-1)} w_i \lambda^{(i)}, i=1,\ldots,n$ and we define 
$w^{\star}=(w^{\star}_{n},\ldots, w^{\star}_1)\in \{\prec, \succ\}^{n}$ by imposing $w^{\star}_i \neq w_i$. Furthermore, given $w$, we set $\Gamma_{i}=\Gamma_{+}$ (resp. $\Gamma_{i}=\Gamma_{-}$) if $w_i $ equals $\prec$ (resp. $\succ$), and we define $w^{\mathrm{sym}}=w \cdot w^{\star}$ where $\cdot$ means concatenation. For $w\in \{\prec,\succ\}^{n},$ we label the elements of $\{i:\Gamma_i =\Gamma_+\}$
  as  $i_1 \leq \ldots \leq i_m$ and those of  $\{i:\Gamma_i =\Gamma_-\}$ as $j_{n-m} \leq \ldots\leq j_{1}.$ We now define $\mathfrak{s}_1,\ldots, \mathfrak{s}_n$  by setting
\begin{equation}\label{deutschess}
\mathfrak{s}_{i_k}=a_{k},\quad k=1,\ldots,m; \quad \mathfrak{s}_{j_k}= a_{k},\quad k=1,\ldots,n-m 
\end{equation} 
with the $(a_n)_{n \geq 1}$  from \eqref{symweights}.

To $w$ we associate an encoded shape: we construct a down-right path $\hat{\pi}=(\hat{\pi}(0),\ldots,\hat{\pi}(n)), \hat{\pi}(0)=(1,\#\{i:w_i =\succ\}+1), \hat{\pi}(n)=(\#\{i:w_i=\prec\}+1,1)$ of unit steps by setting $\hat{\pi}(i+1)-\hat{\pi}(i)=(1,0)$ if $w_{i+1}$ equals $\prec$, and $\hat{\pi}(i+1)-\hat{\pi}(i)=(0,-1)$ otherwise. This path can be seen as the boundary of a Young diagram drawn in French convention, denoted by $sh(w)$; the bottom left corner of this Young diagram is located at $(1,1)$ --- see Figure~\ref{GrowthDiag} left. For fixed $\lambda, \mu$ we have the bijective local growth rule
\begin{equation}\label{growthrule}
\mathcal{T}_{\mathrm{loc}}: \{\kappa: \kappa \prec \lambda, \kappa \prec \mu\} \times \mathbb{\Z}_{\geq 0 } \to \{\nu: \nu \succ \lambda, \nu \succ \mu\}
\end{equation}
where $\nu:=\mathcal{T}_{\mathrm{loc}}((\kappa,k))$ is given by
\begin{equation}\label{growthrules}
  \begin{split}
\nu_1 &=\max\{\lambda_1, \mu_1\}+k, \\
\nu_i &=\max\{\lambda_{i}, \mu_{i}\}+\min\{\lambda_{i-1}, \mu_{i-1}\}- \kappa_{i-1}, \quad i \geq 2 
  \end{split}
\end{equation}
and one has
\begin{equation}\label{volume}
|\nu|+ |\kappa|=|\mu|+|\lambda|+k.
\end{equation}

Let $\{G_{r,t}, r,t \geq 1\}$ be nonnegative integers with $G_{r,t}=G_{t,r}$ ($G_{r,t}$ is a possible realization of the $\omega_{r,t}$ from \eqref{symweights}). We now recursively construct partitions $\lambda^{r,t}, (r,t)\in \Z_{\geq 1}^{2}$ as follows: for $(r,t)\in \{(1,k):k\geq 1\}\cup\{(k,1):k \geq 1\}$ we set $\lambda^{r,t}=\emptyset$. Given $\lambda^{r,t}=:\kappa, \lambda^{r+1,t}=:\mu, \lambda^{r,t+1}=:\lambda$ we define 
\begin{align}\label{constructedpart}
\lambda^{r+1,t+1}:=\mathcal{T}_\mathrm{loc} ((\kappa,G_{r,t})).
\end{align} 

Note that, since $G_{r,t}=G_{t,r}$ and $\mathcal{T}_{\mathrm{loc}}$ is symmetric in $\lambda$ and $\mu$, we have that $\lambda^{r,t}=\lambda^{t,r}$. 

Given $w^{\mathrm{sym}}\in \{\prec,\succ\}^{2n}$ and the corresponding down-right path $\hat{\pi}=(\hat{\pi}(0),\ldots,\hat{\pi}(2n))$,  we denote
\begin{align}\label{lambdam}
\lambda^{(m)}=\lambda^{\hat{\pi}(m)}, m=0, \ldots,2n
\end{align}

\noindent and note that we have $\lambda^{(n+k)}=\lambda^{(n-k)},k=0,\ldots,n$ --- see Figure~\ref{GrowthDiag} right. 

\begin{SCfigure}
\begin{tikzpicture}[scale=0.8]
\begin{scope}[xshift=-200]
\draw[very thick] (0,0) -- (1,0) -- (1,-1) -- (1,-1) -- (2,-1) --(2,-2) -- (3,-2) --(3,-3);
\draw(0,0)-- (0,-1) -- (1,-1) -- (1,-2) -- (1,-2) -- (2,-2) --(2,-3) -- (3,-3);
\draw (0,-1) --(0,-2) --(1,-2) --(1,-3)--(2,-3);
\draw (0,-2)--(0,-3) --(1,-3);

\draw (0,0) node[anchor=east]{$\hat{\pi}$};

\draw (0.5, 0) node[anchor=south]{\mbox{\footnotesize$\prec$}};
\draw (1, -0.5) node[anchor=west]{\mbox{\footnotesize$\succ$}};
\draw (1.5, -1) node[anchor=south]{\mbox{\footnotesize$\prec$}};
\draw (2, -1.5) node[anchor=west]{\mbox{\footnotesize$\succ$}};
\draw (2.5, -2) node[anchor=south]{\mbox{\footnotesize$\prec$}};
\draw (3, -2.5) node[anchor=west]{\mbox{\footnotesize$\succ$}};
\draw(0,-3) node[anchor=north]{(1,1)};
\end{scope}

\draw (0,0) -- (1,0) -- (1,-1) -- (1,-1) -- (2,-1) --(2,-2) -- (3,-2) --(3,-3);
\draw(0,0)-- (0,-1) -- (1,-1) -- (1,-2) -- (1,-2) -- (2,-2) --(2,-3) -- (3,-3);
\draw (0,-1) --(0,-2) --(1,-2) --(1,-3)--(2,-3);
\draw (0,-2)--(0,-3) --(1,-3);
\foreach \x in {0.5}
\foreach \y in {-2.5}
{
\draw (\x,\y) node{\mbox{\footnotesize$G_{1,1}$}};
}
\foreach \x in {1.5}
\foreach \y in {-2.5}
{
\draw (\x,\y) node{\mbox{\footnotesize$G_{1,2}$}};
}
\foreach \x in {2.5}
\foreach \y in {-2.5}
{
\draw (\x,\y) node{\mbox{\footnotesize$G_{1,3}$}};
}
\draw(1.5, -1.5) node{\mbox{\footnotesize$G_{2,2}$}};
\draw(0.5, -1.5) node{\mbox{\footnotesize$G_{1,2}$}};
\draw(0.5, -0.5) node{\mbox{\footnotesize$G_{1,3}$}};
\foreach \x in {0}
\foreach \y in {-1,-2}
\draw (\x, \y) node[anchor=east]{\mbox{\footnotesize$\emptyset$}};
\draw (0, 0) node[anchor=east]{\mbox{\footnotesize$\lambda^{(0)}=\emptyset$}};
\draw(0,-3) node[anchor= north east]{\mbox{\footnotesize$\emptyset$}};
\foreach \y in {-3}
\foreach \x in {1,2}
\draw (\x, \y) node[anchor=north]{\mbox{\footnotesize$\emptyset$}};
\draw(3.6,-3) node[anchor=north]{\mbox{\footnotesize$\emptyset=\lambda^{(0)}$}};
\draw(1,0) node[anchor=south]{\mbox{\footnotesize$\lambda^{(1)}$}};
\draw(1,-1) node[anchor=south west]{\mbox{\footnotesize$\lambda^{(2)}$}};
\draw(2,-2) node[anchor=south west]{\mbox{\footnotesize$\lambda^{(2)}$}};
\draw(2.3,-1.1) node[anchor=south ]{\mbox{\footnotesize$\lambda^{(3)}$}};
\draw(3.3,-2.1) node[anchor= south]{\mbox{\footnotesize$\lambda^{(1)}$}};
\end{tikzpicture}
\caption{\fs{Left: the encoded shape $sh(w^{\mathrm{sym}})=(3,2,1)$ for $w=(\prec, \succ, \prec)$ together with the down-right path $\hat{\pi}$ in bold; right:
the partitions $\lambda^{(0)},\ldots,\lambda^{(3)},\ldots,\lambda^{(0)}$ associated to the points of $\hat{\pi}$ are constructed recursively using the empty partitions on the boundary, the $G_{r,t}$ and $\mathcal{T}_{\mathrm{loc}}$. }}
\label{GrowthDiag}
\end{SCfigure}
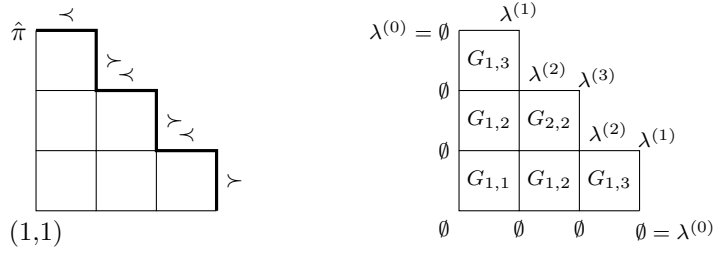
The following proposition is an elementary induction on $|sh(w^{\mathrm{sym}})|$ (see Theorem 3.2 of \cite{bbbccv} for a similar proof), which we omit carrying out (the partition function $Z_w$ appearing in \eqref{symmeasure} was computed in \eqref{eq:fbz}).
\begin{prop}\label{Fomin}
Let $w\in \{\prec, \succ\}^{n}, w^{\mathrm{sym}}=w\cdot w^{\star}$ and  $\mathfrak{s}_{i},i=1,\ldots,2n,$ be the variables \eqref{deutschess} for $w^{\mathrm{sym}}$. With the symmetric weights \eqref{symweights} we obtain by \eqref{lambdam} a probability distribution on
\begin{equation}\label{setLambda}
\{\vec{\lambda}: \vec{\lambda}=(\emptyset=\lambda^{(0)},\ldots,\lambda^{(n-1)},\lambda^{(n)},\lambda^{(n-1)},\ldots,\lambda^{(0)}), \vec{\lambda} \text{\,\,is\,\,} w^{\mathrm{sym} }\text{-interlaced\,} \}
\end{equation}
which is given by
\begin{align}\label{symmeasure}
\Prob(\{\vec{\lambda}\})=\frac{1}{Z_{w}} \bra{\lambda^{(n)}}\Gamma_{-}(\{\alpha\})\oneerv \cdot \prod_{i=1}^{n}\bra{\lambda^{(i-1)}}\Gamma_{i}(\{  \mathfrak{s}_i   \})\ket{\lambda^{(i)}}.
\end{align}
\end{prop}

Now we can prove Theorem~\ref{LPPmulti}.

\begin{proof}[Proof of Theorem~\ref{LPPmulti}] 
  Note that, by construction, $L_{(1,1)\to (m,n)}=\lambda^{m+1,n+1}_{1}$, where $\lambda^{m+1,n+1}_{1}$ is from \eqref{constructedpart} and  $G_{r,t}=\omega_{r,t}$. Consequently, the LPP times become a gap probability for the point process $\mathfrak{S}(\vec{\lambda})$ where $ \vec{\lambda}$ is distributed as  \eqref{symmeasure}. By \eqref{gap}, the left hand side of \eqref{labelneeded} is given as a Fredholm pfaffian, and the corresponding correlation functions were computed in Theorem~\ref{thm:onebound} and \eqref{eq:kapoc}, leading to the identity \eqref{labelneeded}.
\end{proof}

\begin{proof}[Proof of Theorem~\ref{LPPThm2}]
We have to show the convergence, as $N \to \infty$, of the  Fredholm pfaffian provided by Theorem \ref{LPPmulti}. By Proposition \ref{pointconvv}, the kernel $K$ from Theorem \ref{LPPmulti} converges pointwise to $K^{v},$  and by Proposition \ref{expprop} we can apply Lemma \ref{lemexp} which yields an integrable upper bound, allowing us to apply dominated convergence to show the convergence of the Fredholm pfaffian. 
\end{proof}

\subsubsection{Asymptotics}

In the following asymptotics, we will ignore the integer parts in \eqref{x1} when it leads to no confusion.
\begin{prop}\label{pointconvv}
Set $a_j=\sqrt{q}$ for a $q\in (0,1)$ and let $u_1>\cdots >u_k\geq  0$ and $c_q=\frac{1-\sqrt{q}}{q^{1/6}(1+\sqrt{q})^{1/3}},d_q=\frac{q^{1/6}}{2(1+\sqrt{q})^{2/3}}$. Set $\alpha=1-2v c_q N^{-1/3}, v \in \R.$ We take $t_l=N,  r_l=N-\lfloor u_l N^{2/3} \rfloor,l=1,\ldots,k$ and 
\begin{equation}\label{x1}
  \begin{split}
 x_a &= \lfloor \frac{2\sqrt{q}}{1-\sqrt{q}}N \rfloor-\lfloor  u_a N^{2/3}\frac{\sqrt{q}}{1-\sqrt{q}}\rfloor      +  \lfloor \xi N^{1/3}\rfloor -1/2, \\
 x_b &= \lfloor \frac{2\sqrt{q}}{1-\sqrt{q}}N \rfloor-\lfloor  u_b N^{2/3}\frac{\sqrt{q}}{1-\sqrt{q}}\rfloor      +  \lfloor \xi'  N^{1/3}\rfloor -1/2.
  \end{split}
\end{equation}

We then have for the kernels $K$ from Theorem~\ref{LPPmulti} and $K^{v}$ from \eqref{Kv}:
\begin{equation}
  \begin{split}
&\lim_{N\to \infty}N^{2/3}e^{-N^{2/3}\ln(1-\sqrt{q})(u_a+u_b)}c_{q}^{-1}K_{1,1}(a,x_a ; b,x_b)=c_q K_{1,1}^{v}(d_{q}u_a, c_q \xi  ; d_{q} u_b,c_q \xi'), \\
&\lim_{N\to \infty}N^{1/3}e^{N^{2/3}\ln(1-\sqrt{q})(u_b-u_a)}K_{1,2}(a,x_a ; b,x_b)=c_q K_{1,2}^{v}(d_{q}u_a, c_q \xi ; d_{q} u_b,c_q \xi'), \\
&\lim_{N\to \infty}e^{N^{2/3}\ln(1-\sqrt{q})(u_b+u_a)}c_q K_{2,2}(a,x_a ; b,x_b)=c_q K_{2,2}^{v}(d_{q}u_a, c_q \xi ; d_{q} u_b,c_q \xi').
  \end{split}
\end{equation}
\end{prop}

\begin{proof}
We start with $K_{1,2}$. We get for $a \leq b$ 
\begin{align}
K_{1,2}(a,x_a;b, x_b)&=\frac{1}{(2\pi \im)^{2}}\oint_{1/\sqrt{q}>|z|>\alpha,1}\dx z \oint_{1/\sqrt{q},|z|>|w|>\alpha, \sqrt{q}}\dx w \frac{z-\alpha}{w-\alpha}\frac{zw-1}{w(z-w)(z-1)(z+1)}\\
&\nonumber \times \frac{e^{N f_{1}(z)+N^{2/3}u_a f_{2}(z)-N^{1/3}f_{3}(z)\xi}}{e^{N f_{1}(w)+N^{2/3}u_{b} f_{2}(w)-N^{1/3}f_{3}(w)\xi'}}\\ 
&\label{K12v2}=\frac{1}{(2\pi \im)^{2}} \oint_{\gamma_{-1,0,1,w}}\dx z \oint_{\gamma_{\sqrt{q},\alpha}}\dx w \frac{z-\alpha}{w-\alpha}\frac{zw-1}{w(z-w)(z-1)(z+1)}\\
&\nonumber \times \frac{e^{N f_{1}(z)+N^{2/3}u_a f_{2}(z)-N^{1/3}f_{3}(z)\xi}}{e^{N f_{1}(w)+N^{2/3}u_{b} f_{2}(w)-N^{1/3}f_{3}(w)\xi'}}
\end{align}
where we have set
 \begin{equation}
 \begin{aligned}
 f_{1}(z) &= \ln(1-\sqrt{q}/z)-\ln(1-\sqrt{q}z)-\ln(z)\frac{2\sqrt{q}}{1-\sqrt{q}}, \\
 f_{2}(z) &= \ln(z)\frac{\sqrt{q}}{1-\sqrt{q}}+\ln(1-\sqrt{q}z),\\ 
 f_{3}(z) &= \ln(z).
 \end{aligned}
 \end{equation}
 
 One readily computes $f_{1}(1)=f_{1}^{\prime}(1)=f_{1}^{\prime \prime}(1)=0, f_{1}^{\prime \prime \prime}(1)=2 c_{q}^{-3}$ and $f_{2}(1)=\ln(1-\sqrt{q}), f_{2}^{\prime}(1)=0,f_{2}^{\prime\prime}(1)=-\frac{\sqrt{q}}{(1-\sqrt{q})^{2}}$, also $d_q=-f^{\prime\prime}_{2}(1)c_{q}^{2}/2$.

 Let us briefly outline the strategy for the asymptotics. We choose contours $\gamma_{-1, 0,1,w}$ for $z$ and $\gamma_{\sqrt{q},\alpha}$ for $w$ such that they pass (almost) through the critical point $1$. Furthermore, for $\delta>0$ small (but independent of $N$) we want $\Re(f_{1}(z))\leq - c_{0,0}$ and $-\Re(f_{1}(w))\leq -c_{0,0}$ for some $c_{0,0} >0$ (depending on $\delta$) for $z \in \gamma_{-1,0,1,w}^{\delta}, w \in  \gamma_{\sqrt{q}, \alpha}^{\delta}$ where we define $\gamma_{-1,0,1,w}^{\delta}=\{z\in \gamma_{-1,0,1,w}:|z-1|\leq \delta\}, \gamma_{\sqrt{q},\alpha}^{\delta}=\{w\in \gamma_{\sqrt{q},\alpha}: |w-1|\leq \delta\}$.

 Given this, the integral over $ (\gamma_{-1,0,1,w}^{\delta}\cup \gamma_{\sqrt{q}, \alpha}^{\delta})^{c}$ will vanish as $N\to \infty$. On $ \gamma_{-1,0,1,w}^{\delta}\cup  \gamma_{\sqrt{q}, \alpha}^{\delta}$ we use Taylor to obtain $Nf_{1}(1+Zc_{q}N^{-1/3})=Z^{3}/3+\mathcal{O}(Z^{4}N^{-1/3})$. For this reason we want $(\gamma_{-1,0,1,w}^{\delta}-1)N^{1/3} $ respectively $(\gamma_{\sqrt{q},\alpha}^{\delta}-1)N^{1/3}$ to lie (up to a part of length $\mathcal{O}(1)$) in $\{z \in \C: \Re(z^{3})<0\}$ respectively  $\{w \in \C: \Re(w^{3})>0\}.$ 

 We first note that $\Re (f_{1}(e^{\im s}))=0$ for all $s$. Furthermore, we compute
 \begin{equation}
 L(z)=z f_{1}^{\prime}(z)=\frac{\sqrt{q}(1+\sqrt{q})}{1-\sqrt{q}}\frac{(1-z)^{2}}{(\sqrt{q}-z)(-1+\sqrt{q}z)}.
 \end{equation}

An elementary computation shows 
\begin{equation}
\Im(L(e^{\im s}))=0.
\end{equation}

Next we treat $\Re(L(e^{\im s}))=L(e^{\im s})$. We have  
\begin{equation}\label{REL}
\Re(L(z))=\frac{\sqrt{q}(1+\sqrt{q})}{1-\sqrt{q}}\frac{1}{|(\sqrt{q}-z)(-1+\sqrt{q}z)|^{2}}\Re((1-z)^{2}(\sqrt{q}-\bar{z})(-1+\sqrt{q}\bar{z})).
\end{equation}
This implies by a simple computation
\begin{equation}\label{L}
L(e^{\im s})=\frac{\sqrt{q}(1+\sqrt{q})}{1-\sqrt{q}}\frac{2(-1+\cos(s))}{|\sqrt{q}-e^{\im s}|^{2}}.
\end{equation}

Let now $\varepsilon>0,$  and $\phi \in [0, 2 \pi]$. Then,  by Taylor approximation for some $t_s \in [0,1]$, we have
\begin{equation}
\begin{aligned}\label{Taylor}
\Re(f_{1}(e^{\im s}+\varepsilon e^{\im \phi}))&=\Re(f_{1}(e^{\im s}))+ \Re(f_{1}^{\prime}(e^{\im s})\varepsilon e^{\im \phi})+\frac{1}{2}\Re(f_{1}^{\prime \prime}(e^{\im s}+t_{s}\varepsilon e^{\i\phi})\varepsilon^{2} e^{2\im \phi})\\
&=L(e^{\im s})\varepsilon \cos(\phi-s)+\frac{\varepsilon^{2}}{2}\Re(f_{1}^{\prime \prime}(e^{\im s}+t_{s}\varepsilon e^{\im \phi}) e^{2\im \phi}).
\end{aligned}
\end{equation}

Take $\varepsilon=-\varepsilon_1 L(e^{\im s})$ for $\varepsilon_1 >0$ small
and   $\phi=\pi+s$  respectively $\phi=s$ to get
\begin{equation}
\begin{aligned}\label{crucial}
&\Re(f_{1}((1+\varepsilon_1 L(e^{\im s}))e^{\im s}))\geq L(e^{\im s})^{2} \varepsilon_1  / 2, \\
&\Re(f_{1}((1-\varepsilon_1 L(e^{\im s}) )e^{\im s}))\leq -L(e^{\im s})^{2} \varepsilon_1  / 2.
\end{aligned}
\end{equation}

Choose now the contour (Figure~\ref{contourlpp})
\begin{equation}\label{gammaq}
\gamma_{\sqrt{q},\alpha}(s)=(1+N^{-1/3}(2|v|c_q +1/2))(1+\varepsilon_1 L(e^{\im s}))e^{\im s}, \quad s \in [0, 2\pi].
\end{equation}
The prefactor $(1+N^{-1/3}(2|v|c_q +1/2))$ makes sure that $\alpha$ lies inside the contour. Furthermore, by \eqref{crucial} and \eqref{L}, for any $\delta>0$ independent of $N$ there is a $c_{0,0}>0$ for which we  have $-\Re (f_{1}(w))\leq - c_{0,0}$ for $w\in  \gamma_{\sqrt{q},\alpha}\setminus \gamma_{\sqrt{q},\alpha}^{\delta}=\{w\in \gamma_{\sqrt{q},\alpha}:|w-1|>\delta\}.$ To choose $\gamma_{-1,0,1,w}$ consider first
  \begin{equation}
 \tilde{\gamma}_{-1,0,1,w}(s)=(1+N^{-1/3}(2|v|c_q +1))(1-\varepsilon_1 L(e^{\im s}))e^{\im s}, \quad s \in [0, 2\pi].
 \end{equation}
By \eqref{crucial} and \eqref{L}, $\Re(f_{1}(\tilde{\gamma}_{-1,0,1,w}(s)))$ behaves as desired, but  $\Re(\tilde{\gamma}_{-1,0,1,w}(s))< 1+N^{-1/3}(2|v|c_q +1))$ if $\varepsilon_1 L(e^{\im s})(-1+\cos(s))^{-1} $ is too small, i.e. the contour may not cross the positive real axis with the right angles. We thus do a local modification. Choose  $\eta_1 >0$ small and let $\eta_2 >0$ be the number such that  $e^{\im  \arcsin(\eta_1)}+\eta_2 \in\{(1-\varepsilon_1 L(e^{\im s}))e^{\im s},s\in [0, 2 \pi]\}$, and let $ e^{\im  \arcsin(\eta_1)}+\eta_2  =(1-\varepsilon_1 L(e^{\im s_{\eta_{2}}}))e^{\im s_{\eta_2}}$. Let $M>1$ and define
\begin{equation}
  \begin{split}
& \gamma_{-1,0,1,w}^{loc,1}(s)= 1+\im s+s/M, \, s\in[0,\eta_1]\mathrm{\quad and } \\
& \gamma_{-1,0,1,w}^{loc,2}(s)=e^{\im \arcsin(\eta_1)}+s, \, s\in [\eta_2, 1+\eta_{1}/M- \cos(\arcsin(\eta_1 ))].
  \end{split}
\end{equation} 

Choosing $M$ sufficiently large and $\eta_1$ sufficiently small we get by Taylor approximation around $1$ and $e^{\arcsin(\eta_1)\im}$ that for some constants $d_1, d_2 >0$
\begin{align}\label{Taylor2} 
\Re(f_{1}(\gamma_{-1,0,1,w}^{loc,1}(s)))\leq -s^{3} d_1 \mathrm{\quad and \quad} \Re(f_{1}(\gamma_{-1,0,1,w}^{loc,2}(s)))\leq  - \eta_{1}^{2}s d_2.
\end{align}

Define finally $\gamma_{-1,0,1,w}^{loc,3}= \{(1-\varepsilon_1 L(e^{\im s}))e^{\im s},s\in [s_{\eta_2},  \pi]\}$.

Since $\Re(f_{1}(z))=\Re(f_{1}(\bar{z}))$ it suffices to define $\gamma_{-1,0,1,w}$ on the upper half plane. We set
\begin{align}\label{gamma-101}
\gamma_{-1,0,1,w}=(1+N^{-1/3}(2|v|c_q +1))(\gamma_{-1,0,1,w}^{loc,1}+   \gamma_{-1,0,1,w}^{loc,2}       +  \gamma_{-1,0,1,w}^{loc,3} )\end{align}
where the $+$ means we concatenate the curves such that  $\gamma_{-1,0,1,w}$ is counterclockwise oriented. On the lower half plane, we simply take the image of \eqref{gamma-101} under complex conjugation.  We now choose $\delta < \eta_1$ such that $\gamma_{-1,0,1,w}^{\delta} \subset \{1+\im s+|s|/M, s \in [-\eta_1, \eta_1]\} $.

\begin{SCfigure} 
   \begin{tikzpicture}[scale=0.7]
  \draw [very thick](-2,-0.1) -- (-2,0.1) node[below=4pt] {\small{$1$}};
    \draw [very thick](0.3,-0.1) -- (0.3,0.1) ;
    \draw (-0.1,0) node[below=2pt] {\small{$w_c$}};
     \draw [very thick](2.3,-0.1) -- (2.3,0.1) ;
     \draw (2.1,0) node[below=2pt ] {\small{$\tilde{w}_c $}};
     \draw (2.9,1.2) -- (3.1,1.56);
          \draw(2.9,-1.2) -- (3.1,-1.56);
     \draw(3.1,-1.56) --(2.117,-1.56);
         \draw (3.1,1.56) --(2.117,1.56);
         \begin{scope}[scale=5]
\foreach \x in {14,15,16,17,18,19,20,21,22,23,24,25,26,27,28,29,30}
\draw[] ({(1+(-1+cos(\x))/5)*cos(\x)-0.94},{(1+(-1+cos(\x+1))/5)*sin(\x)}) -- ({(1+(-1+cos(\x+1))/5)*cos(\x+1)-0.94},{ (1+(-1+cos(\x+1))/5)*sin(\x+1)});
\foreach \x in {14,15,16,17,18,19,20,21,22,23,24,25,26,27,28,29,30}
\draw[] ({(1+(-1+cos(\x))/5)*cos(\x)-0.94},{-((1+(-1+cos(\x+1))/5)*sin(\x))}) -- ({(1+(-1+cos(\x+1))/5)*cos(\x+1)-0.94},{- (1+(-1+cos(\x+1))/5)*sin(\x+1)});
\foreach \x in {30}
\draw[->, very thick] ({(1+(-1+cos(\x))/5)*cos(\x)-0.94},{(1+(-1+cos(\x+1))/5)*sin(\x)}) -- ({(1+(-1+cos(\x+1))/5)*cos(\x+1)-0.94},{ (1+(-1+cos(\x+1))/5)*sin(\x+1)});
\end{scope}

         \begin{scope}[scale=5]
\foreach \x in {0,1,2,3,4,5,6,7,8,9,10,11,12,13,14}
\draw[dotted, thick] ({(1+(-1+cos(\x))/5)*cos(\x)-0.94},{(1+(-1+cos(\x+1))/5)*sin(\x)}) -- ({(1+(-1+cos(\x+1))/5)*cos(\x+1)-0.94},{ (1+(-1+cos(\x+1))/5)*sin(\x+1)});
\foreach \x in {0,1,2,3,4,5,6,7,8,9,10,11,12,13,14}
\draw[ dotted, thick ] ({(1+(-1+cos(\x))/5)*cos(\x)-0.94},{-((1+(-1+cos(\x+1))/5)*sin(\x))}) -- ({(1+(-1+cos(\x+1))/5)*cos(\x+1)-0.94},{- (1+(-1+cos(\x+1))/5)*sin(\x+1)});
\end{scope}

  \draw[dotted, thick] (2.3,0) -- (2.9,1.2);
  \draw[dotted, thick] (2.3,0) -- (2.9,-1.2);

     \begin{scope}[scale=5]
\foreach \x in {18,19,20,21,22,23,24,25,26,27,28,29,30}
\draw[ smooth] ({(1+(1-cos(\x))/5)*cos(\x)-0.536},{(1+(1-cos(\x))/5)*sin(\x)}) -- ({(1+(1-cos(\x+1))/5)*cos(\x+1)-0.536},{ (1+(1-cos(\x+1))/5)*sin(\x+1)});
\foreach \x in {30}
\draw[->, very thick] ({(1+(1-cos(\x))/5)*cos(\x)-0.536},{(1+(1-cos(\x))/5)*sin(\x)}) -- ({(1+(1-cos(\x+1))/5)*cos(\x+1)-0.536},{ (1+(1-cos(\x+1))/5)*sin(\x+1)});
\foreach \x in {18,19,20,21,22,23,24,25,26,27,28,29,30}
\draw[smooth] ({(1+(1-cos(\x))/5)*cos(\x)-0.536},{-((1+(1-cos(\x))/5)*sin(\x))}) -- ({(1+(1-cos(\x+1))/5)*cos(\x+1)-0.536},{- (1+(1-cos(\x+1))/5)*sin(\x+1)});
\end{scope}
  \node[label = above left: $\gamma_{\sqrt{q},\alpha}^{\delta}$] at (1.9, 0.3) {};
    \node[label = above left: $\gamma_{-1,0,1,w}^{\delta}$] at (5.2, 0.3) {};
          \draw [very thick, ->] (-5,0) -- (5,0) node[below=8pt] {\small{$\R$}};
   \end{tikzpicture}    
  \caption{\fs{Parts of the contours $\gamma_{\sqrt{q},\alpha}$ and $\gamma_{-1,0,1,w}$  from \eqref{gammaq} and \eqref{gamma-101} close to where they cross the positive real axis in
    $w_c=1+N^{-1/3}(2|v|c_q +1/2)$, $\tilde{w}_c=w_c+N^{-1/3}/2 $. The curve segments $\gamma_{\sqrt{q},\alpha}^{\delta},\gamma_{-1,0,1,w}^{\delta}$ are shown in dotted lines.}}
\label{contourlpp} 
\end{SCfigure}
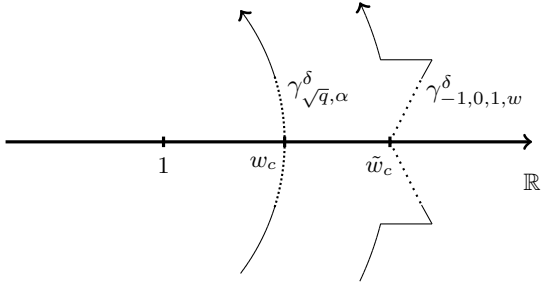

By virtue of \eqref{Taylor2}, \eqref{crucial} and \eqref{L}, for $z\in \gamma_{-1,0,1,w}\setminus \gamma_{-1,0,1,w}^{\delta}=\{z\in \gamma_{-1,0,1,w}: |z-1|>\delta\}$ we have $\Re(f_{1}(z))\leq - c_{0,1}$ for a $c_{0,1}>0$  and we can take $c_{0,0}=c_{0,1}$. We choose $\Sigma$ to be the part of the contours where $ z \notin \gamma^{\delta}_{-1,0,1,w}$ and/or $w\notin \gamma_{\sqrt{q},\alpha}^{\delta}$. Then the integral \eqref{K12v2} is on
\begin{align}
\Sigma \cup ( \gamma^{\delta}_{-1,0,1,w} \cup \gamma_{\sqrt{q},\alpha}^{\delta})=  \gamma_{-1,0,1,w} \cup \gamma_{\sqrt{q},\alpha}.
\end{align}
On $\Sigma$, $\exists\,c_{0,0}>0$ such that $\Re(f_{1}(z))<-c_{0,0}$ and/or $-\Re(f_{1}(w))<-c_{0,0}$, and also $\Re(f_{1}(z)), -\Re(f_{1}(w))<c_{0,0}/4$. Furthermore, we can bound $\left| \frac{z-\alpha}{w(w-\alpha)}\frac{zw-1}{(z-w)(z^{2}-1)}\right| < C(\delta) N^{1/3}$ where $C(\delta)$ is a constant depending on $\delta$. So overall we may bound
\begin{equation}
N^{1/3}e^{N^{2/3}f_{2}(1)(u_b -u_a)}\left|\int_{\Sigma}\dx z\dx w   \frac{zw-1}{(z^{2}-1)(z-w)}\frac{z-\alpha}{w(w-\alpha)}
\frac{e^{N f_{1}(z)+N^{2/3}u_a f_{2}(z)-N^{1/3}f_{3}(z)\xi}}{e^{N f_{1}(w)+N^{2/3}u_b f_{2}(w)-N^{1/3}f_{3}(w)\xi'}}    \right|\leq C e^{-N c_{0,0}/2}
\end{equation}
for a $C>0$.

For the integral on $ \gamma^{\delta}_{-1,0,1,w} \cup \gamma_{\sqrt{q},\alpha}^{\delta},$ we do the change of variable
\begin{equation}\label{ZW}
z=1+Zc_{q}N^{-1/3} \quad w=1+Wc_{q}N^{-1/3}.
\end{equation}

Use now Taylor and define $F(Z,W)$ via
\begin{equation}
\begin{aligned}
\label{errorterms}
&e^{N(f_{1}(z)-f_{1}(w))}e^{N^{2/3}(u_a f_{2}(z)-u_b f_{2}(w))}e^{N^{1/3}(\xi' f_{3}(w)- \xi f_{3}(z))}\\&=e^{N^{2/3}(u_a f_{2}(1)-u_b f_{2}(1))}e^{Z^{3}/3
+Z^{2}u_ac_{q}^{2}f^{\prime\prime}_{2}(1)/2-\xi Z c_q     -W^{3}/3 -W^{2}u_b c_{q}^{2}f^{\prime\prime}_{2}(1)/2+ W c_q \xi'} 
\\&\times e^{\mathcal{O}(Z^{4}N^{-1/3})+\mathcal{O}(Z^{3}N^{-1/3})+ \mathcal{O}(Z^{2}N^{-1/3}) +\mathcal{O}(W^{4}N^{-1/3})+\mathcal{O}(W^{3}N^{-1/3})+ \mathcal{O}(W^{2}N^{-1/3})}
\\& = F(Z,W)e^{N^{2/3}(u_a f_{2}(1)-u_b f_{2}(1))}
\end{aligned}
\end{equation}
and denote $\tilde{F}(Z,W)=e^{Z^{3}/3
+Z^{2}u_ac_{q}^{2}f^{\prime\prime}_{2}(1)/2-\xi Z c_q     -W^{3}/3 -W^{2}u_b c_{q}^{2}f^{\prime\prime}_{2}(1)/2+ W c_q \xi'}  $.
To control the contribution from the error terms in \eqref{errorterms}, we use the inequality $|e^{x}-1|\leq e^{|x|}|x|$.

With the change of variable \eqref{ZW}, we have to control
\begin{multline}
A:= \Big|\frac{c_q}{(2\pi \im )^{2}}\int_{(\gamma_{-1,0,1,w}^{\delta}-1)\frac{N^{1/3}}{c_q}}\dx Z \int_{(\gamma_{\sqrt{q},\alpha}^{\delta}-1)\frac{N^{1/3}}{c_q}}\dx W \frac{Z+2v}{W+2v}\frac{Z+W+ZWN^{-1/3}c_{q}}{(1+W c_{q}N^{-1/3})(Z-W)Z(2+Zc_{q}N^{-1/3})} \\ 
\times (F(Z,W)-\tilde{F}(Z,W))\Big|.
\end{multline}
We have 
\begin{equation}
  \begin{split}
  A \leq 
&  \frac{c_q}{(2\pi )^{2}}\int_{(\gamma_{-1,0,1,w}^{\delta}-1)\frac{N^{1/3}}{c_q}}\left|\dx Z \right|\int_{(\gamma_{\sqrt{q},\alpha}^{\delta}-1)\frac{N^{1/3}}{c_q}}\left|\dx W \right| \left| \frac{Z+2v}{W+2v}\frac{Z+W +ZWN^{-1/3}c_{q}}{(Z-W)Z}   \tilde{F}(Z,W) \right| \\
&  \times e^{|\mathcal{O}(Z^{4}N^{-1/3})+\mathcal{O}(Z^{3}N^{-1/3})+ \mathcal{O}(Z^{2}N^{-1/3}) +\mathcal{O}(W^{4}N^{-1/3})+\mathcal{O}(W^{3}N^{-1/3})+ \mathcal{O}(W^{2}N^{-1/3})|} \\
&  \times |\mathcal{O}(Z^{4}N^{-1/3})+\mathcal{O}(Z^{3}N^{-1/3})+ \mathcal{O}(Z^{2}N^{-1/3}) +\mathcal{O}(W^{4}N^{-1/3})+\mathcal{O}(W^{3}N^{-1/3})+ \mathcal{O}(W^{2}N^{-1/3})| .
  \end{split}
\end{equation}
This implies 
\begin{equation}
  \begin{split}
A \leq 
&  \frac{c_q}{(2\pi )^{2}}\int_{(\gamma_{-1,0,1,w}^{\delta}-1)\frac{N^{\frac{1}{3}}}{c_q}}\left|\dx Z \right|\int_{(\gamma_{\sqrt{q},\alpha}^{\delta}-1)\frac{N^{\frac{1}{3}}}{c_q}}\left|\dx W \right| \left| \frac{Z+2v}{W+2v}\frac{Z+W+ZWN^{-\frac{1}{3}}c_{q}}{(Z-W)Z}   \frac{e^{Z^{3}(1+\chi_1)/3-\xi Z(1+\chi_2) c_q}}{ e^{W^{3}(1+\chi_3)/3- W(1+\chi_4) c_q\xi'} } \right| \\ 
&  \times \left| \frac{e^{Z^{2}(1+\chi_5) u_a c_{q}^{2}f^{\prime\prime}_{2}(1)/2    }}{ e^{W^{2}(1+\chi_6)u_b c_{q}^{2}f^{\prime\prime}_{2}(1)/2      } } \right| N^{-1/3}|\mathcal{O}(Z^{4})+\mathcal{O}(Z^{3})+ \mathcal{O}(Z^{2}) +\mathcal{O}(W^{4})+ \mathcal{O}(W^{3})+\mathcal{O}(W^{2})|
  \end{split}
\end{equation}
where the $\chi_i \in \R, i=1, \ldots,6$ can be taken as small in absolute value as desired by taking $\delta $ small. Now for large $N$, the $|e^{Z^{3}(1+\chi_1)/3}e^{-W^{3}(1+\chi_3)/3}|$ term dominates the integral. At the integration boundary, it is of order $e^{-\mathcal{O}(\delta^{3}N)}$. This easily implies that $N^{1/3}A $ remains bounded as $N \to \infty$. Consequently, $ A$ vanishes like $N^{-1/3}$. We can thus take $\tilde{F}(Z,W)$ instead of $F(Z,W)$ and only make an error of $\mathcal{O}(N^{-1/3})$. By doing so, we are left with
\begin{equation}\label{almostint}
\frac{c_q}{(2\pi \i)^{2}}\int_{(\gamma_{-1,0,1,w}^{\delta}-1)\frac{N^{1/3}}{c_q}}\dx Z \int_{(\gamma_{\sqrt{q},\alpha}^{\delta}-1)\frac{N^{1/3}}{c_q}}\dx W \frac{Z+2v}{W+2v}\frac{Z+W+ZWc_q N^{-1/3}}{(Z-W)Z}\frac{(2+Zc_q N^{-1/3})^{-1}}{1+WN^{-1/3}c_q}\tilde{F}(Z,W) .
\end{equation}
Finally, in \eqref{almostint} we can extend the curves to infinity (inside $\{z\in\C: \Re(z^{3})<0\}$ and $\{w\in\C: \Re(w^{3})>0\}$) and thus only make an error $e^{-\mathcal{O}(N)}.$ 
We can then deform the contours to be as in \eqref{K12v} without errors. 

To summarize, we have shown that
\begin{multline}
N^{1/3}e^{f_{2}(1) N^{2/3}(u_b  -   u_a )}K_{1,2}(a,x_a;b, x_b)=\frac{c_q}{(2\pi \i)^{2}}\int_{\mathfrak{G}_{A_Z}^{\pi/3}}\dx Z \int_{\mathfrak{G}_{A_W}^{2\pi/3}}\dx W \frac{(Z+W+ZWc_q N^{-1/3})(Z+2v)}{(W+2v)Z(Z-W)}
\\
\times \frac{e^{Z^{3}/3
+Z^{2}u_ac_{q}^{2}f^{\prime\prime}_{2}(1)/2-\xi Z c_q     -W^{3}/3 -W^{2}u_b c_{q}^{2}f^{\prime\prime}_{2}(1)/2+ W c_q \xi'} }{(1+WN^{-1/3}c_q)(2+Zc_q N^{-1/3})} 
+ \label{errors}\mathcal{O}(N^{-1/3}+ e^{-N c_{0,0}/2}+e^{-\mathcal{O}(N)}),
\end{multline}
implying
\begin{equation}
  \begin{split}
\lim_{N \to \infty}N^{1/3}e^{f_{2}(1) N^{2/3}(u_b  -   u_a )}K_{1,2}(a,x_a;b, x_b)=&\frac{c_q}{(2\pi \i)^{2}}\int_{\mathfrak{G}_{A_Z}^{\pi/3}}\dx Z \int_{\mathfrak{G}_{A_W}^{2\pi/3}}\dx W \frac{(Z+W)(Z+2v)}{2(W+2v)Z(Z-W)} \\
& \times e^{Z^{3}/3
+Z^{2}u_ac_{q}^{2}f^{\prime\prime}_{2}(1)/2-\xi Z c_q -W^{3}/3 -W^{2}u_b c_{q}^{2}f^{\prime\prime}_{2}(1)/2+ W c_q \xi'} 
  \end{split}
\end{equation}
for $A_Z>A_W>-2v, A_Z>0$.

If now $a>b,$ we have the condition $|z|<|w|$ on our contours. Deforming them so as to equal 
\eqref{gamma-101}, \eqref{gammaq}, we pick up an extra residue, which equals
\begin{equation}\label{reszw}
N^{1/3}e^{f_{2}(1) N^{2/3}(u_b  -   u_a )}\frac{-1}{2 \pi \im}\oint_{\gamma_{0}}\dx z e^{N^{2/3}f_{2}(z)(u_a-u_b)+N^{1/3}f_{3}(z)(\xi'-\xi)}z^{-1}.
\end{equation} 
For $\gamma_{0}(s)=e^{\im s}$ the contribution of the integral over $\gamma_{0}^{\delta}$ in \eqref{reszw} clearly vanishes. Note $u_a-u_b<0$. On $\gamma_{0}^{\delta}$ we do the same change of variable $z=1+Zc_q N^{-1/3}$. Next we Taylor-expand $f_{2}(1+Zc_q N^{-1/3}),f_{3}(1+Zc_q N^{-1/3})$ as before, and control the contribution from the remainder terms as before. Sending then $N \to \infty$ shows that \eqref{reszw} converges to
\begin{equation}
\frac{- c_q }{2 \pi \im}\int_{\im \R}\dx Z e^{Z^{2}c_{q}^{2}f_{2}^{\prime\prime}(1)(u_a-u_b)/2}e^{Z c_q (\xi'-\xi)},
\end{equation}
with $\im \R$ oriented with increasing imaginary part.

Next we come to $K_{1,1}(a,x_a;b, x_b)$. We have
\begin{equation}\label{K11}
  \begin{split}
K_{1,1}(a,x_a;b, x_b) = &\frac{1}{(2 \pi \i)^{2}}\oint_{\gamma_{-1,0,1}}\dx z \oint_{\gamma_{-1,0,1}} \dx w e^{N(f_1(z)+f_1 (w))+N^{2/3}(u_a   f_{2}(z)+u_b f_{2}(w))} \\
& \times \frac{e^{-N^{1/3}(f_{3}(z)\xi+f_{3}(w)\xi')}(z-w)}{(z^{2}-1)(w^{2}-1)(zw-1)}(z-\alpha)(w-\alpha).
  \end{split}
\end{equation}

We can choose the contour \eqref{gamma-101} for both $z$ and $w$ in \eqref{K11}. Redoing all the steps made for $K_{1,2}$ we obtain
\begin{equation}
  \begin{split}
\lim_{N \to \infty}N^{2/3}e^{-N^{2/3}f_{2}(1)(u_a +u_b)}K_{1,1}(a,x_a;b, x_b)
=& \frac{c_{q}^{2}}{(2 \pi \i)^{2}}\int_{\mathfrak{G}^{\pi/3}_1}\dx Z
\int_{\mathfrak{G}^{\pi/3}_1}\dx W\frac{(Z-W)(W+2v)(Z+2v)}{4ZW(Z+W)}\\
& \times e^{Z^{3}/3 +Z^{2}u_ac_{q}^{2}f^{\prime\prime}_{2}(1)/2-\xi Z c_q +W^{3}/3 + W^{2}u_b c_{q}^{2}f^{\prime\prime}_{2}(1)/2- W c_q \xi'}.
  \end{split}
\end{equation}

Finally, we come to $K_{2,2}(a,x_a;b, x_b)$. We have
\begin{equation}
  \begin{split}
K_{2,2}(a,x_a;b, x_b) &=\frac{1}{(2 \pi \im)^{2}}\oint_{\gamma_{\sqrt{q},\alpha}}\dx z \oint_{\gamma_{\sqrt{q},\alpha,1/z}}\dx w e^{-N(f_{1}(z)+f_{1}(w))-N^{2/3}(u_a f_{2}(z)+u_{b}f_{2}(w))+N^{1/3}(f_{3}(z)\xi+f_{3}(w)\xi')} \\
& \times \frac{z-w}{zw(z-\alpha)(w-\alpha)(zw-1)}.
  \end{split}
\end{equation}

We can choose the contour for $z$ to also contain $1/\alpha$, e.g. $\gamma_{\sqrt{q},\alpha}(s)=(1+N^{-1/3}(2|v| c_q +1/2))e^{\im s}.$ Note that on the unit circle, $\Re(-f_{2}(z)),\Re(-f_{2}(1/z))$ are maximal at $z=1$ and decrease until they reach $z=-1$.

We consider first the case $u_a+u_b>0$. We start with the residue of $w$ at $1/z$. It equals 
\begin{equation}\label{res1/z}
\frac{1}{2 \pi \im}\oint_{(1+N^{-1/3}(2|v| c_q +1/2))e^{\im s}}\dx z \frac{z^{2}-1}{z(z-\alpha)(-\alpha)(z-1/\alpha)}e^{-N^{2/3}(f_{2}(z)u_a +f_{2}(1/z)u_b)}e^{N^{1/3}(f_{3}(z)\xi-f_{3}(z)\xi')}.
\end{equation}
Repeating the steps of the asymptotics for $K_{1,2},$ one then obtains  
\begin{equation}
\lim_{N \to \infty}e^{N^{2/3}f_{2}(1)(u_a+u_b)}\eqref{res1/z}=\frac{-1}{2 \pi \im}\int_{\mathfrak{G}^{2\pi/3}_{B_1}}\dx Z\frac{2Z}{(Z+2v)(Z-2v)}e^{-Z^{2}f_{2}^{\prime\prime}(1)c_{q}^{2}(u_a+u_b)/2+Zc_{q}(\xi-\xi')}
\end{equation}
with $B_1> 2|v|$.

Next we consider the contribution from the pole at $w=\alpha$, given by 
\begin{equation}\label{polewa}
\frac{e^{N^{2/3}f_{2}(1)(u_a+u_b)}}{2 \pi \im}\oint_{(1+N^{-1/3}(2|v| c_q +1/2))e^{\im s}}\dx z \frac{e^{-N(f_{1}(z)+f_{1}(\alpha))-N^{2/3}(f_{2}(z)u_a +f_{2}(\alpha)u_b)+N^{1/3}(f_{3}(z)\xi+f_{3}(\alpha)\xi') }}{z \alpha(z\alpha-1)}
\end{equation}
We may deform the contour  $(1+N^{-1/3}(2|v| c_q +1/2))e^{\im s}$ in \eqref{polewa} to 
\begin{equation}\label{gammaqal}
(1+N^{-1/3}(2|v| c_q +1/2))(1+\varepsilon_{1}L(e^{\im s}))e^{\im s}, s\in[0, 2 \pi].
\end{equation}
without errors. By the same asymptotic analysis performed for $K_{1,2}$ we then see that 
\begin{equation}
\lim_{N \to \infty}   \eqref{polewa}= \frac{e^{8 v^{3}/3-2v^{2}c_{q}^{2}f_{2}^{\prime\prime}(1)u_b-2vc_q \xi'}}{2 \pi \im}\int_{\mathfrak{G}_{B_2}^{2\pi/3}}\dx Z\frac{1}{Z-2v}e^{-Z^{3}/3-Z^{2}u_{a}c_{q}^{2}f_{2}^{\prime \prime}(1)/2+\xi Z c_{q}}, \qquad B_2 > 2v.
\end{equation}
We compute the remaining term 
\begin{equation}\label{limq}
  \begin{split}
\lim_{N \to \infty} \frac{e^{N^{2/3}f_{2}(1)(u_a+u_b)}}{(2 \pi \im)^{2}}\oint_{\gamma_{\sqrt{q},\alpha}}\dx z \oint_{\gamma_{\sqrt{q}}}\dx w & e^{-N(f_{1}(z)+f_{1}(w))-N^{2/3}(u_a f_{2}(z)+u_{b}f_{2}(w))+N^{1/3}(f_{3}(z)\xi+f_{3}(w)\xi')} \\
& \times \frac{z-w}{zw(z-\alpha)(w-\alpha)(zw-1)}.
  \end{split}
\end{equation}

As contours, we can choose \eqref{gammaqal} for $\gamma_{\sqrt{q},\alpha}$, and $\frac{1-N^{-1/3}(4|v| c_q +1)}{1+N^{-1/3}(2|v| c_q +1/2)}\times \eqref{gammaqal}$ for $\gamma_{\sqrt{q}}$ (note that with this choice, $\gamma_{\sqrt{q},\alpha}$ has no point of $\gamma_{\sqrt{q}}$ in its interior, and vice versa).

With this choice, we can now proceed exactly as before. This then yields that 
\begin{align}\label{doubleintk22}
\eqref{limq}=\frac{1}{(2 \pi \im)^{2}}\int_{\mathfrak{G}^{2 \pi/3}_{B_3}} \dx Z \int_{\mathfrak{G}^{2 \pi/3}_{B_4}} \dx W
&\frac{e^{-Z^{3}/3-u_{a}Z^{2}c_{q}^{2}f_{2}^{\prime \prime}(1)/2+Zc_{q}\xi}}{e^{W^{3}/3+u_{b}W^{2}c_{q}^{2}f_{2}^{\prime\prime}(1)/2-Wc_q \xi'}} \frac{Z-W}{(W+2v)(Z+2v)(Z+W)}
\end{align}
for $B_3 >-2v, B_4 < -2v$ and $B_3<-B_4$.

For $u_a=u_b=0,$ we proceed as follows. Writing the integer parts, the residue  of $w=1/z$ is given by 
\begin{equation}\label{res(z,1/z)}
\frac{1}{2 \pi \im}\oint_{(1+N^{-1/3}(2|v|c_q+1/2))e^{\im s}   }\dx z \frac{z^{2}-1}{(-\alpha)(z-1/\alpha)(z-\alpha)}z^{\lfloor N^{1/3}\xi\rfloor -\lfloor N^{1/3}\xi' \rfloor-1}.
\end{equation}
Now, if $\lfloor N^{1/3}\xi\rfloor -\lfloor N^{1/3}\xi' \rfloor\leq -1,$ we can send $z \to \infty$ in \eqref{res(z,1/z)} which shows $ \eqref{res(z,1/z)} =0.$ If $\lfloor N^{1/3}\xi\rfloor -\lfloor N^{1/3}\xi' \rfloor > 0$, \eqref{res(z,1/z)} equals the sum of the residues at $z=\alpha,1/\alpha.$ and if $
\lfloor N^{1/3}\xi\rfloor -\lfloor N^{1/3}\xi' \rfloor =0,$ \eqref{res(z,1/z)} equals the sum of the residues at $z=0,\alpha,1/\alpha$. Thus 
 \begin{equation}
 \begin{aligned}
 \lim_{N \to \infty} \eqref{res(z,1/z)}&=\lim_{N\to\infty}\big(-\mathbf{1}_{\{\frac{\lfloor N^{1/3}\xi\rfloor-\lfloor N^{1/3}\xi'\rfloor}{N^{1/3}}>0\}}\left(e^{-2v c_q\frac{\lfloor N^{1/3} \xi\rfloor-\lfloor N^{1/3} \xi'\rfloor}{N^{1/3}}}+e^{2v c_q\frac{\lfloor N^{1/3} \xi\rfloor-\lfloor N^{1/3} \xi'\rfloor}{N^{1/3}}}\right)
 \\&-\mathbf{1}_{\{\lfloor N^{1/3}\xi\rfloor-\lfloor N^{1/3}\xi'\rfloor=0\}}\alpha^{-1}\big)
 \\&=-\mathbf{1}_{\{\xi-\xi'>0\}}(e^{-2v c_q(\xi-\xi')}+e^{2v c_q(\xi-\xi')})-\mathbf{1}_{\{\xi=\xi'\}}.
 \end{aligned}
 \end{equation}
 
 The contribution \eqref{polewa} from the pole at $w=\alpha$ is as before, but here we compute separately  the residue $z=1/\alpha$ in \eqref{polewa}, which equals
 \begin{equation}\label{reswalpha1}
 \frac{1}{\alpha}e^{f_{3}(\alpha)N^{1/3}(\xi'  - \xi )}\to_{ N \to \infty}e^{2v c_q(\xi-\xi')}.
 \end{equation} 
  What thus remains to compute from \eqref{polewa} is the  limit of 
 \begin{equation}\label{reswalpha2} 
  \frac{1}{2\pi \im}\oint_{\gamma_{\sqrt{q}}}\dx z\frac{1}{z\alpha(z\alpha-1)}e^{-N  (f (\alpha) +f (z)) }e^{ N^{1/3} \xi' f_{3}(\alpha) }
 e^{ N^{1/3} \xi  f_{3}(z) },
\end{equation}
 which equals 
 \begin{equation}\label{BBB}
\frac{1}{2\pi \im}\int_{\mathfrak{G}_{C_1}^{2\pi/3}}\dx Z e^{-Z^{3}/3+Z \xi c_q}\frac{e^{8v^{3}/3-2v c_q \xi'}}{Z-2v}, \qquad C_1 < 2v. 
\end{equation}
 So in total, the sum of the residues of $w $ in $1/z$ and $\alpha$ converges to 
\begin{equation}\label{midterm}
\eqref{BBB}+\sgn(\xi' -\xi)e^{-2v c_q|\xi-\xi'|}.
\end{equation}
The remaining term is exactly $\eqref{doubleintk22}$ for $u_a=u_b=0$.
\end{proof}

The following proposition will provide the required integrable upper bound.
\begin{prop}\label{expprop}
Let $c_q, K,x_a, x_b, \alpha $ be as in Proposition~\ref{pointconvv}. Let $L >0$ and $\xi,\xi'>-L$ and $a\leq b$. Then there is an $N_0$ and a $C>0$ (which may depend on $L,u_a, u_b$, but not on $\xi, \xi'$) such that for $N >N_0$ 
\begin{align}\label{K12exp}
& N^{1/3}e^{N^{2/3}\ln(1-\sqrt{q})(u_b-u_a)} |K_{1,2}(a,x_a ;b,x_b)|\leq C e^{-c\xi+d\xi'}, \\
& \label{K11exp}N^{2/3}e^{-N^{2/3}\ln(1-\sqrt{q})(u_b+u_a)} |K_{1,1}(a,x_a ;b,x_b)|\leq C e^{-c\xi-c\xi'}, \\
& e^{N^{2/3}\ln(1-\sqrt{q})(u_b+u_a)} |K_{2,2}(a,x_a ;b,x_b)|\leq C e^{d\xi+d\xi'},
\end{align}
with  $c=2|v|c_q+3/4, d=2|v|c_q+2/3.$
\end{prop}
\begin{proof}
We assume first that $\xi,\xi'\geq 0$. To show \eqref{K11exp}, note that the proof of pointwise convergence easily implies that for $\xi=\xi'=0$  and the contour \eqref{gamma-101} for $\gamma_{-1,0,1}$ one has 
\begin{equation}
  \begin{split}
N^{2/3}\frac{|K_{1,1}(a,x_a;b, x_b)|}{e^{N^{2/3}f_{2}(1)(u_a +u_b)}}\leq 
 \frac{N^{2/3}}{e^{N^{2/3}f_{2}(1)(u_a +u_b)}}\oint_{\gamma_{-1,0,1}}|\dx z |\oint_{\gamma_{-1,0,1}}&|\dx w |\big|e^{N(f_1(z)+f_1 (w))+N^{2/3}(u_a   f_{2}(z)+u_b f_{2}(w))}\big| \\
 & \times \left|\frac{(z-\alpha)(w-\alpha)(z-w)}{(z^{2}-1)(w^{2}-1)(zw-1)}\right|\leq C
  \end{split}
\end{equation}
where $f_1,f_2$ are as in the proof of Proposition ~\ref{pointconvv}. If now $\xi,\xi' \geq 0$, this creates an extra factor  
\begin{equation}
|e^{-N^{1/3}(\ln(z)\xi+\ln(w)\xi')}|.
\end{equation}
For the contour \eqref{gamma-101}, we have that for $z,w \in \gamma_{-1,0,1}$ $|z|,|w|\geq 1+N^{-1/3}(2|v|c_q+1)$ . Hence we can bound for $N$ large enough $|e^{-N^{1/3}(\ln(z)\xi+\ln(w)\xi')}|\leq e^{-c(\xi+\xi')}.$
Next we come to \eqref{K12exp}. As contours we choose again $\gamma_{\sqrt{q},\alpha}=\eqref{gammaq},\gamma_{-1,0,1,w}=\eqref{gamma-101}.$ The proof of pointwise convergence easily implies that  for $\xi=\xi'=0$ we have a constant upper bound.
If now $\xi, \xi' \geq 0,$ then we get an extra factor which equals
\begin{equation}\label{factor1}
|e^{N^{1/3}\ln(w)\xi'-N^{1/3}\ln(z)\xi}|.
\end{equation}
Now for $w \in \gamma_{\sqrt{q},\alpha}$, we have $|w|\leq 1+N^{-1/3}(2|v|c_q+1/2)$. Consequently, for $N$ large enough, we have $N^{1/3}\ln(|w|)\leq 2|v|c_q +2/3$. Since for $z\in \gamma_{-1,0,1,w}$ we have $|z|\geq 1+N^{-1/3}(2|v|c_q+1)$ we get $-N^{1/3} \ln(|z|)\leq -(2|v|c_q+3/4)$. This implies that $\eqref{factor1}\leq e^{-c\xi+d\xi'}$. As for $K_{2,2}$, consider the case $u_a+u_b>0$ first. In \eqref{polewa}, by definition of $\alpha$ and choice of contour, we obtain $|\eqref{polewa}|\leq Ce^{d(\xi+\xi')}$. Equally, one obtains $|\eqref{res1/z}|e^{N^{2/3}\ln(1-\sqrt{q})(u_b+u_a)}\leq e^{d(\xi+\xi')}$. Finally, \eqref{limq} can be bounded by $e^{d(\xi+\xi')}$ as well, simply by the choice of contours. The case $u_a+u_b=0$ is treated similarly.
Finally, if $\xi,\xi'\in [-L,0],$ the proof(s) of pointwise convergence give an upper bound $C$, where the constant $C$ may depend on $L$.   If e.g. $\xi \in [-L,0],\xi'>0$ we multiply the bound obtained in the case $\xi\leq 0,\xi'=0$ by the bound obtained for $\xi'\geq 0, \xi=0,$ finishing the proof.
\end{proof}

\section{Symmetric plane partitions} \label{sec:symm_pp}

A \emph{free boundary plane partition of length $N$} is an array $(\pi_{i,j})_{ 1 \leq j \leq i \leq N}$ of nonnegative integers satisfying the properties $\pi_{i,j} \geq \pi_{i+1, j}, \pi_{i,j} \geq \pi_{i, j+1}$
for all meaningful $i,j$. Its \emph{volume} is the sum of its entries: $|\pi| = \sum_{1 \leq j \leq i \leq N} \pi_{i,j}.$

Clearly a free boundary plane partition of length $N$ is half of a \emph{symmetric plane partition} with base in the square $N \times N$, which by definition is an array $(\pi_{i,j})_{ 1 \leq i, j\leq N}$ satisfying the above constraints plus the symmetry constraint $\pi_{i,j} = \pi_{j, i}$. An example of length 5 and volume 79 is depicted in Figure~\ref{free_pp}.

\begin{SCfigure}[][ht]
\begin{tikzpicture}[scale=0.4]
  \draw[thick, smooth] (-1,0)--(4,5);
    \draw[thick, smooth] (-1,0)--(4,-5);
  \foreach \x in {0,...,4} {
      \draw[thick, smooth] (\x, {-\x-1})--(5, {-2*\x+4});
      \draw[thick, smooth] (\x, {\x+1})--(5, {2*\x-4});
  }
  \node at (0,0) {$6$};
  \node at (1,1) {$7$};
  \node at (1,-1) {$3$};
  \node at (2,2) {$9$};
  \node at (2,0) {$5$};
  \node at (2,-2) {$2$};
  \node at (3,3) {$9$};
  \node at (3,1) {$7$};
  \node at (3,-1) {$3$};
  \node at (3,-3) {$1$};
  \node at (4,4) {$10$};
  \node at (4,2) {$8$};
  \node at (4,0) {$6$};
  \node at (4,-2) {$2$};
  \node at (4,-4) {$1$};
\end{tikzpicture}
\qquad
\includegraphics[scale=0.3]{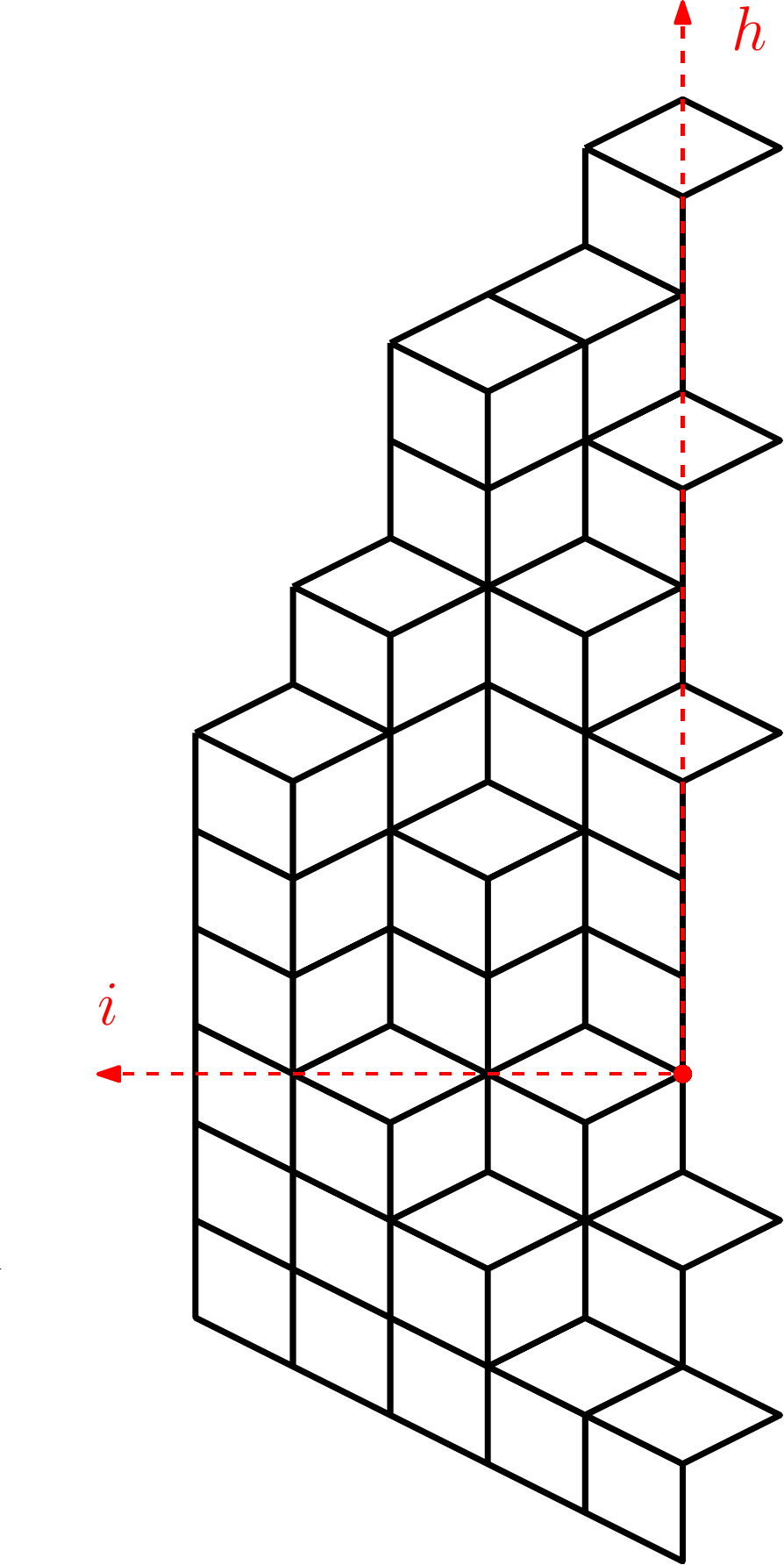}
\qquad
\includegraphics[scale=0.3]{symm_pp.pdf}
\caption{\fs{Left: a free boundary plane partition $\pi$ of length 5, with $\pi_{1,1} = 10, \pi_{4,2} = 5, \dots$, corresponding to interlacing partitions $\es \prec (6) \prec (7,3) \prec (9,5,2) \prec (9,7,3,1) \prec (10, 8, 6, 2, 1)$; middle: a stack of cubes depicting $\pi$ via the heights of the horizontal lozenges; right: the corresponding symmetric plane partition with base contained in a $5 \times 5$ square.}}
\label{free_pp}
\end{SCfigure}

Let us fix a real parameter $0 < q < 1$. We will study the asymptotics of free boundary plane partitions weighted according to their volume; that is, distributed as
\begin{align}
 \Prob(\pi) \propto q^{|\pi|}
\end{align}

\noindent in the limit $q \to 1$ and $N \to \infty$.

Free boundary plane partitions $\pi$ of length $N$ weighted by their volume are in bijection with the H-ascending Schur process on
\begin{align}
 \vec{\lambda} = (\es \prec \lambda^{(1)} \prec \dots \prec \lambda^{(N)})
\end{align}

\noindent with parameters $x_i = q^{N+1-i}$ (the free boundary is $\lambda^{(N)}$) via the following identification:
\begin{align}
 \lambda^{(i)}_k = \pi_{N-i+k, k},\ 1 \leq k \leq i.
\end{align}

In this setting we have
\begin{align}
 \Prob(\vec{\lambda}) \propto q^{\sum_{1 \leq i \leq N} |\lambda^{(i)}|}.
\end{align}

The sequence of partitions $\vec{\lambda}$, itself induced by $\pi$, gives a point process via the identification $\lambda^{(i)} \mapsto \{ k^{(i)}_s = \lambda^{(i)}_s - s + \frac{1}{2} \}$. It turns out the computations and formulas are simpler if we reverse time. We introduce a new sequence of partitions $\mu^{(i)}$ defined by $\mu^{(i)} = \lambda^{(N-i)}$ so that our Schur process becomes $\es = \mu^{(N)} \prec \dots \prec \mu^{(0)}$. Proposition~\ref{prop:H-HV} becomes the following.

\begin{thm}\label{thm:pp_corr}
 The point process induced by $\pi$ via the sequence of partitions $\mu^{(i)}$ is pfaffian with $2 \times 2$ correlation kernel given by
 \begin{equation}
  K(i, k; i', k') = \left[ \begin{array}{cc} K_{1,1} (i, k; i', k') & K_{1,2} (i, k; i', k') \\ -K_{1,2} (i', k'; i, k) & K_{2,2} (i, k; i', k') \end{array} \right]
 \end{equation}
where
\begin{align}
   K_{1,1}(i, k; i', k') &= \frac{1}{(2 \pi \im)^2} \int_{C_z} \int_{C_w} F(N-i, z) F(N-i', w) \frac{1}{z^{k+1} w^{k'+1}} \frac{\sqrt{zw} (z-w)}{(z+1)(w+1)(zw-1)} \dx w \dx z, \notag \\
 K_{1,2}(i, k; i', k') &= \frac{1}{(2 \pi \im)^2} \int_{C_z} \int_{C_w} \frac{F(N-i, z)}{F(N-i', w)} \frac{w^{k'-1}}{z^{k+1}} \frac{\sqrt{zw} (zw-1)}{(z+1)(w-1)(z-w)} \dx w \dx z, \\
 K_{2,2}(i, k; i', k') &= \frac{1}{(2 \pi \im)^2} \int_{C_z} \int_{C_w} \frac{1}{F(N-i, z) F(N-i', w)} z^{k-1} w^{k'-1} \frac{\sqrt{zw} (z-w)}{(z-1)(w-1)(zw-1)} \dx w \dx z,\notag
\end{align}
\noindent the contours are counterclockwise oriented circles centered at the origin of radii slightly larger than 1 with the additional constraint that for $K_{1,2}$ $C_z$ surrounds $C_w$ if and only if $i'\leq i$ and finally where we have denoted
\begin{align}
 F(N-i,z) := \frac{(q/z;q)_N}{(q^{i+1}z; q)_{N-i}}.
\end{align}
\end{thm}

\begin{rem}
 Every particle in our process corresponds to a \emph{horizontal} lozenge in the middle picture from Figure~\ref{free_pp}, and we imagine tiling the floor of the room with infinitely many horizontal lozenges going down. Some authors (e.g., \cite{or}) prefer to write the kernel in terms of heights $h$ of these lozenges. For a given particle at position $(i,k)$, the coordinate of the corresponding lozenge (in the axes depicted, with the origin at the hidden corner of the box) is $(i, h) = (i, k-\frac{i}{2})$. We opt to work throughout with the ordinate $k$ of particle positions, but the reader interested in the lozenge picture has only to keep in mind $k = h + \frac{i}{2}.$
\end{rem}

We now pick a real number $\a > 0$. We zoom in around a point $(i,k)$ in the point process induced by $\mu^{(i)}$ for large $N, i, k$ as $q$ approaches 1. Precisely, we are interested in the asymptotic regime $r \to 0+$ with
\begin{equation} \label{eq:var_scaling_pp}
    q = e^{-r} \to 1, \qquad r N \to \a, \qquad r i \to \x, \qquad r k \to \y.
\end{equation}

The coordinate system $(\x, \y)$ represents our macroscopic coordinates, with $\y \in \R$ and $0 \leq \x \leq \a$. The positive real number $\a$ plays the role of a boundary parameter. We furthermore introduce macroscopic exponential coordinates (and exponential boundary parameter)
\begin{align}
  \X = \exp(-\x), \qquad \Y = \exp(-\y), \qquad \A = \exp(-\a).
\end{align}

Notice $0 < \A < 1, \A \leq \X \leq 1$. We first focus on the case $0<\x<\a$ and thus $\A < \X < 1$.

Analyzing the kernel asymptotically in the regime \eqref{eq:var_scaling_pp} (following \cite{or}) can be reduced (more details will be given below) to the analysis of the $z$ integrand $F(N-i, z)/z^{k+1/2}$.

We first note in the limit $r \to 0+$, the zeros of $F(N-i, z)$ accumulate in the set $\mathcal{Z} = [0,1]$ while the poles in the set $\mathcal{P} = [\frac{1}{\X}, \frac{1}{\A}]$. Moreover as $r \to 0+$ we have by \eqref{eq:zqnas}
\begin{align}
 \frac{F(N-i, z)}{z^{k + 1/2}} \sim \exp \frac{1}{r} S(z; \x, \y)
\end{align}

\noindent where $S(z; \x, \y) = \dilog(\A/z)-\dilog(1/z)+\dilog(\X z)-\dilog(\A z) - y \log z.$

To apply the method of steepest descent to $S$, we look for its critical points: that is, solutions of $z \frac{\partial}{\partial z} S(z; \x, \y) = 0$. There are two of them:
\begin{align} \label{eq:pp_crit_pts}
     z_{\pm} = z_{\pm} (\x, \y) = \frac{1+\X-(1+\A^2)\Y \pm \sqrt{D(\X, \Y)}} {2(\X-\A\Y)}
\end{align}

\noindent where
\begin{equation}
  \begin{split}
  D(\X, \Y) &= -4 (1 - \A \Y) (\X - \A \Y) + (-1 - \X + \Y + \A^2 \Y)^2 \\
            &= (\A^2-1)^2 \Y^2 -2 (\A-1)^2 \X \Y + \X^2 - 2 (\A-1)^2 \Y - 2 \X +1.
  \end{split}
\end{equation}

\begin{rem}
 The equation $D(\X, \Y)=0$, the locus where $S$ has a real double critical point, is a conic in the $(\X, \Y)$ plane. Its discriminant is $-16 \A (\A-1)^2 < 0$ and so it is the equation of an ellipse. When $\A=0$ (equivalently, $\a \to \infty$), it becomes a parabola.
\end{rem}

\begin{rem}
 The two critical points in equation \eqref{eq:pp_crit_pts} have a singularity on the line $\X = \A \Y$, where one of them becomes $\pm \infty$ (depending on the sign of $\X - \A \Y$) while the other stays finite at $\frac{\A \Y - 1}{(1-\A+\A^2) \Y - 1}$ as can be easily verified by an application of l'H\^opital's rule.
\end{rem}

Keeping $\x$ and hence $\X$ fixed, we solve for $\Y$ (and hence for $\y$) from $D(\X, \Y) = 0$ to obtain two solutions:
\begin{align}
  \Y_{\pm} = \frac{(1+\X)(1-\A)^2 \pm 2 (1-\A)\sqrt{-\A \X^2+\X(1+\A^2)-\A}}{(1-\A^2)^2}, \qquad \y_{\mp} = -\log(\Y_{\pm}).
\end{align}

 We note $\y_- \leq \y_+$ and both depend on $\x$; when $\x = \a$ we have $\y_- = \y_+ = \log(1+\A)$; when $\x = 0$ (hence $\X = 1$) we have $\y_+ = \infty$ (as $\Y_- = 0$) and $\y_- = -\log \frac{4}{(1+\A)^2}$.

 \begin{SCfigure}[][t]
    \includegraphics[scale=0.4]{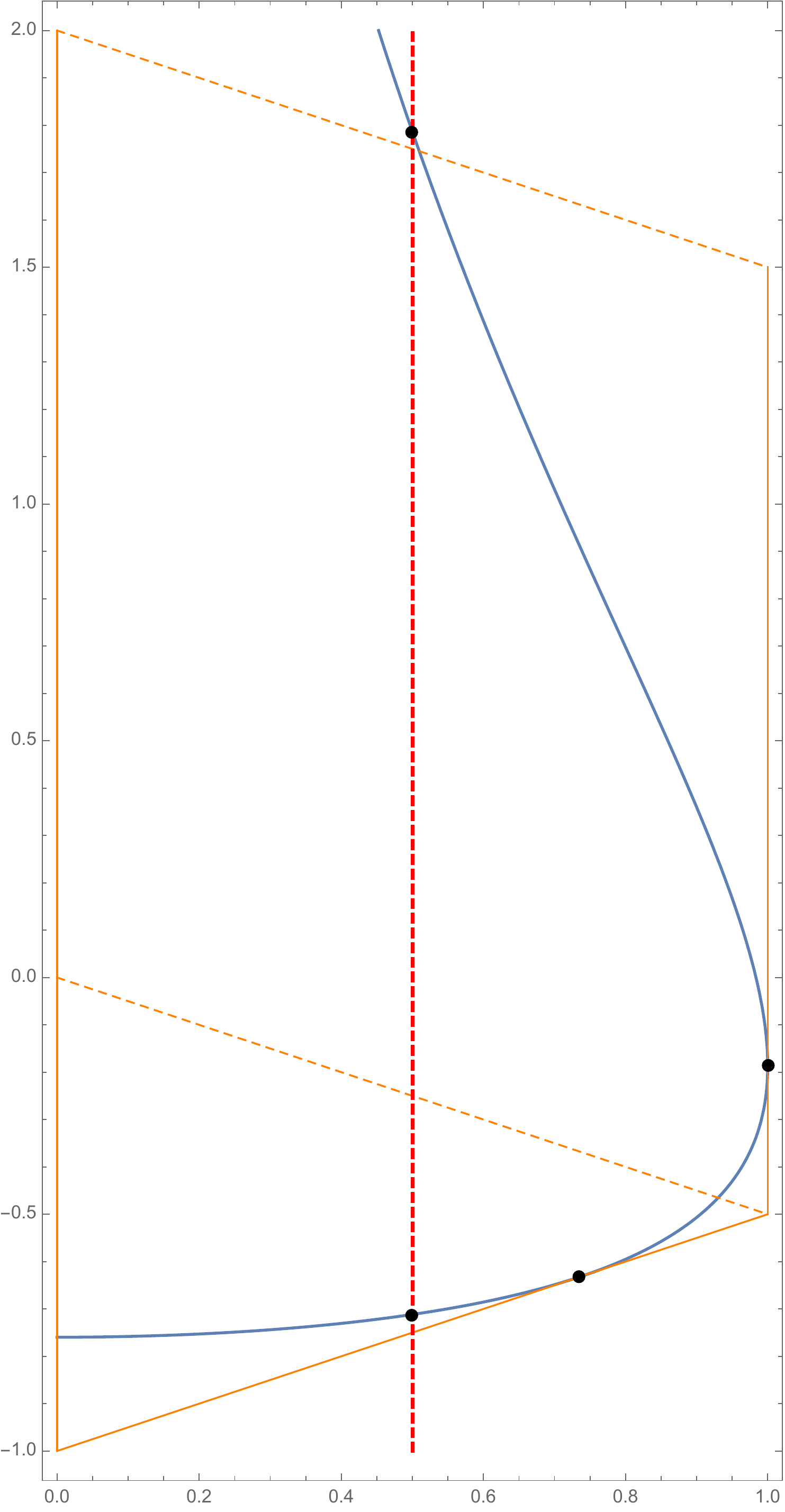}
    \includegraphics[scale=0.062]{pp_symm_3D_crop.png}
    \caption{\fs{Left: a portion of the arctic curve for $\a=1$ in the lozenge $(\x, \h) = (\x, \y - \x/2)$ coordinates; the vertical dashed line is of abscissa $\x=1/2$ and intersects the arctic curve in the two points $\h_{\pm} = \y_{\pm}-\x/2$; the points of tangency of the arctic curve to the boundary of the domain on the bottom and right are also pictured. Right: a simulation (after symmetrization) of a boxed plane partition with base of size $100 \times 100$ and $q \approx 0.98$ using the algorithm from \cite{bbbccv}.}}
    \label{fig:arctic_curve_pp}
  \end{SCfigure}
 
We call the curve
\begin{equation}
  \mathcal{C} = \{(\x, \y) \in \R^2: \a \geq \x \geq 0, D(\X, \Y) = 0\}
\end{equation}
the \emph{arctic curve}. It is depicted in Figure~\ref{fig:arctic_curve_pp}. We denote by $\mathcal{L}$ --- which we call \emph{the liquid region}, the inside of the arctic curve; that is, the domain
 \begin{align}
   \mathcal{L} = \{(\x, \y) \in \R^2: \a \geq \x \geq 0, \y_- \leq \y \leq \y_+ \}.
 \end{align}

 We now discuss what happens to the critical points at various $(\x, \y)$ positions in the plane. To do that, we first take care of the singularity on the line $\X = \A \Y$ or, equivalently, the line $\x = \a + \y$. The curve $\mathcal{C}$ is tangent to this line, at the tangency points $(\x_T, \y_T)$ given by:
 \begin{equation}
  (\X_T, \Y_T) = \left(\frac{\A}{1-\A+\A^2}, \frac{1}{1-\A+\A^2}\right), \qquad (\x_T, \y_T) = (-\log \X_T, -\log \Y_T)
\end{equation}
where by construction $0 < \x_T < \a, \y_T = \x_T - \a < 0$. We have three cases:

 \begin{itemize}
\item when $D(\X, \Y) = 0$ and thus $(\x, \y) \in \mathcal{C}$, $S$ has a double critical point: $z = \frac{1+\X-(1+\A^2) \Y}{2(\X-\A \Y)}$. This happens twice for fixed $\x$: when $\y = \y_-$ or $\y = \y_+$ (the two cases coinciding at $\x=\a, \y_- = \y_+ = \log(1+\A)$). First, the case $\y = \y_+ > 0$. We have $z \in (1, \frac{1}{\X})$. When $(\x, \y) \to (0, \infty)$ along $\mathcal{C}$, $z \to 1$ from above (so $z > 1$ always). When $(\x, \y) \to (\a, \log(1+\A))$ from above along $\mathcal{C}$, $z \to \frac{1}{\A}$ from below. $\mathcal{C}$ is tangent at $(\a, \log(1+\A))$ to the line $\x=\a$. Second, the case $\y = \y_- \leq \log(1+\A)$.  If $0 \leq \x < \x_T$, $z \in (-\infty, -1]$ with $z = -1$ at $\x = 0$ and $z \to -\infty$ when $\x \to \x_T$ from below.  When $\x \in (\x_T, \a]$, $z \in [\frac{1}{\A}, \infty)$ with $z \to \infty$ when $\x \to \x_T$ from above and $z \to \frac{1}{\A}$ from below as $\x \to \a$;

\item  when $D(\X, \Y) < 0$, which means $(\x, \y)$ is in the interior of $\mathcal{L}$, the two distinct critical points are complex conjugate with arguments $\pm \theta(\x, \y)$ where we take $\theta(\x, \y) \in (0, \pi)$. $\theta(\x, \y)$ varies from 0 to $\pi$ as $\y$ descends from $\y_+$ to $\y_-$ through the liquid region $\mathcal{L}$ for fixed $\x < \x_T$. In the case $\x > \x_T$ both double critical points $z^t$ corresponding to $(\x, \y_+)$ and $z^b$ corresponding to $(\x, \y_-)$ are positive, with $0 < z^t < z^b$, so if one recenters the complex plane at any real number between $z^t$ and $z^b$ and considers the arguments of the two complex conjugate critical points corresponding to $\y_-<\y<\y_+$, we have a similar situation as above. See Figure~\ref{fig:cp_plot}.

\item when $D(\X, \Y) > 0$ (and thus for $(\x, \y)$ not in $\mathcal{L}$) there are two distinct real critical points. We study what happens for fixed $\x$. If $\y > \y_+$, $1 < z_- < z_+ < \frac{1}{\X}$ with $z_- = 1, z_+ = \frac{1}{\X}$ in the limit $\y \to \infty$ and $z_\pm$ converging to the unique double critical point in the limit $y \to \y_+$. If $\y < \y_-$, the situation is more complicated due to the singularity at $\X = \A \Y$ (equivalently on the line $\x = \a+\y$). We again distinguish two cases. First, $\x < \x_T$ fixed. For $\y \in (\x-\a, \y_-)$, if $\y \to y_-$, the two roots $z_\pm $ converge to the corresponding (negative) real double critical point. When $\y \to \x - \a$ from above, $z_- \to -\infty$ and $z_+ \to \frac{\A \Y - 1}{(1-\A+\A^2) \Y-1} < 0$ from below. When $\y$ passes below the line $\x-\a$ and goes to $-\infty$, $z_-$ goes from $\infty$ (for $\y$ just below the line $\y = \x-\a$) to $\frac{1}{\A}$ (at $\y \to -\infty$) while $z_+$ goes from $\frac{\A \Y - 1}{(1-\A+\A^2) \Y-1}$ at $\y = \x-\a$ to $\A$ at $\y \to -\infty.$  Second, the case $\x > \x_T$. $z_-$ goes from the double critical point $z_c > \frac{1}{\A}$ to $\frac{1}{\A}$ as $\y$ goes from $\y_-$ to $-\infty$, passing through $z_- = \frac{\A \Y - 1}{(1-\A+\A^2) \Y-1}$ at the line $\y = \x-\a$. $z_+$ goes from $z_c$ to $\infty$ as $\y$ moves from $\y_-$ and approaches the line $\y = \x-\a$ from above, and then jumps and goes from $\infty$ to $\A$ as $\y$ descends from just below the line $\y=\x-\a$ to $-\infty$.
 \end{itemize}

 \begin{figure}[!ht]
  \centering
  \includegraphics[scale=0.3]{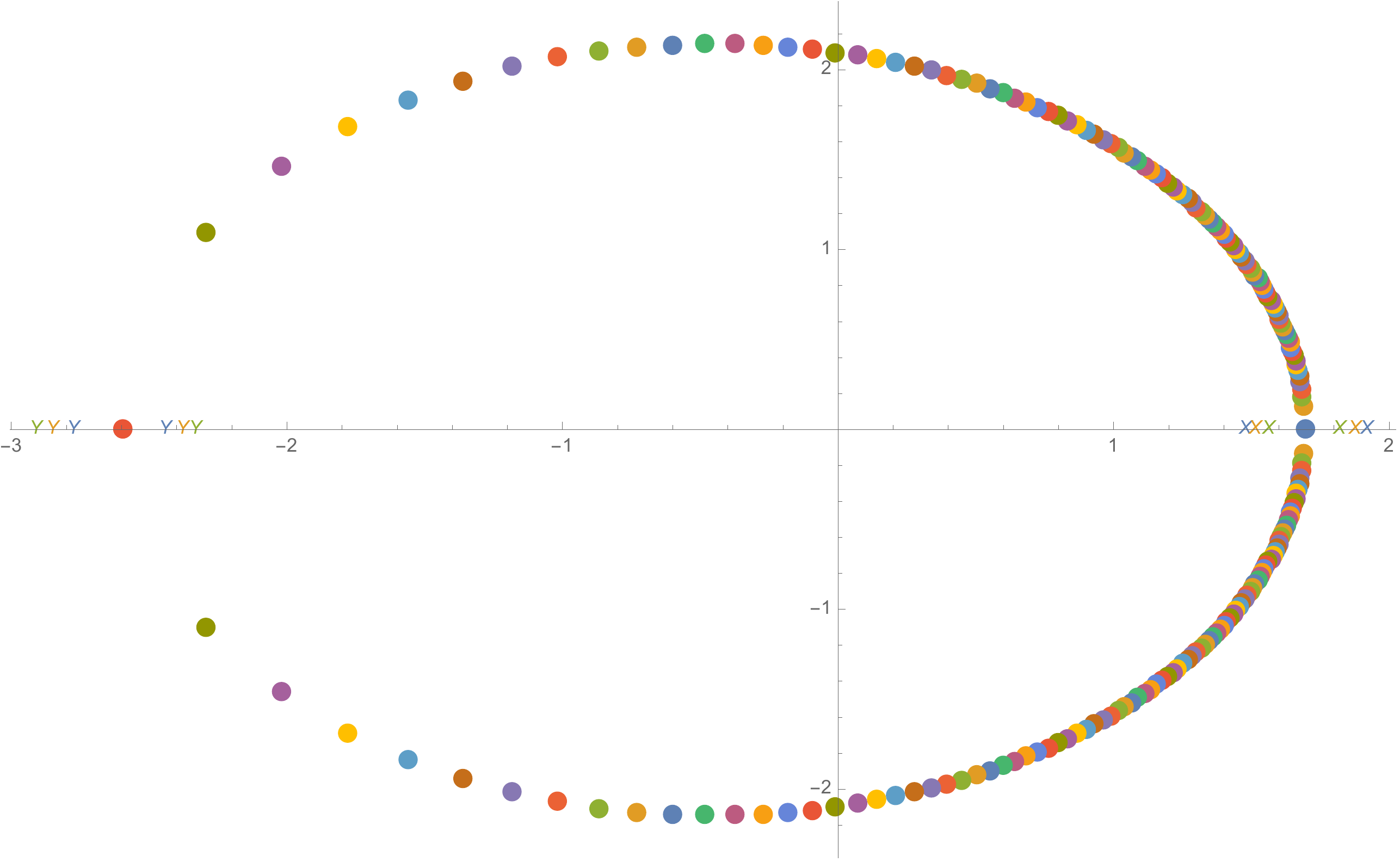}
  \includegraphics[scale=0.3]{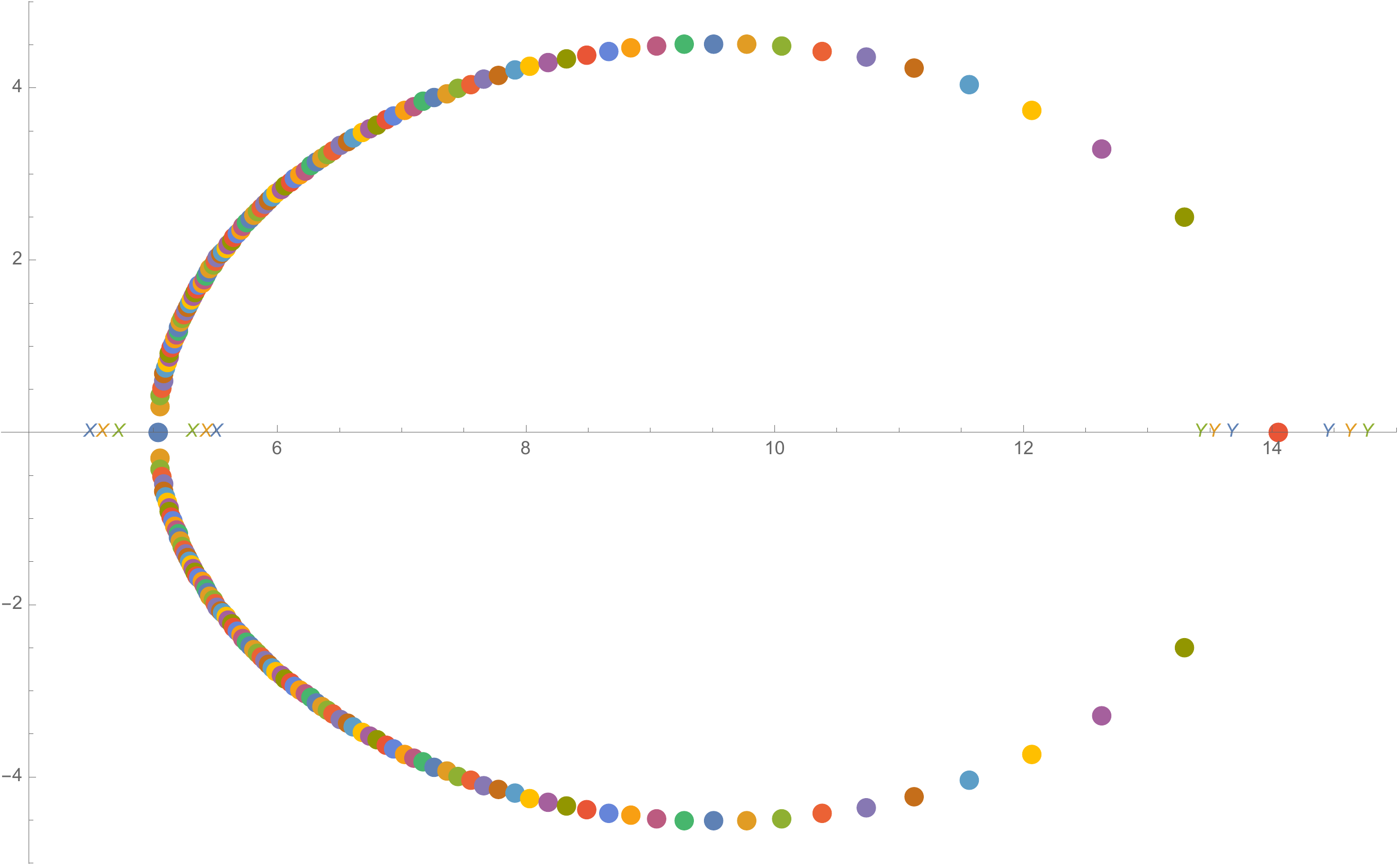}
  \caption{\fs{The critical points of $S$, for $\a=2$ and for $\x=1 < \x_T$ (left), $\x=1.97 > \x_T$ (right) fixed, as $\y$ varies from just above $\y_+$ to just below $\y_-$; they are distinct real for $\y>\y_+$ (denoted by an X) approaching the double critical point at $\y=\y_+$; they are complex conjugate for $\y_+ > \y > \y_-$ approaching the double critical point at $\y=\y_-$ and they become distinct real again (denoted by Y) diverging from said double critical point as $\y$ decreases below $\y_-$.}}
  \label{fig:cp_plot}
\end{figure}

 \begin{rem}
 In light of the above discussion, the curve $\mathcal{C}$ is the locus in the $(\x, \y)$ plane where $S$ has a double critical point. Given a point $(\x, \y) \in \mathcal{C}$, one computes the corresponding double critical point $z_c(\x, \y) = \frac{1+\X-(1+\A^2) \Y}{2(\X-\A \Y)}$. Importantly, one can go backwards as well and for each double critical point $z_c$, one can compute the corresponding $(\x, \y) = (\x(z_c), \y(z_c))$ by solving for $(\x, \y)$ in the following system of two equations
 \begin{align}
  \left( z \frac{\partial }{\partial z} \right) S (z_c; \x, \y) = 0, \qquad \left( z \frac{\partial }{\partial z} \right)^2 S (z_c; \x, \y) = 0
 \end{align}

 \noindent and in doing so one obtains a \emph{time/rational parametrization} of $\mathcal{C}$ as
 \begin{align}
  \mathcal{C} = \{(\x(z_c), \y(z_c)) : z_c \in \R \backslash (-1, 1] \}
 \end{align}
 \noindent with the time parameter played by the double critical point of $S$. In brief $z_c \mapsto (\x(z_c), \y(z_c))$ takes $\R \backslash (-1,1]$ to $\mathcal{C}$: as $z_c$ increases from $1+$ to $\infty$ we draw the upper part of $\mathcal{C}$ from $(0, \infty)$ to $(\x_T, \y_T)$, and then as $z_c$ increases from $-\infty$ to $-1$ we draw the lower part of $\mathcal{C}$ from $(\x_T, \y_T)$ to $(0, -\log \frac{4}{(1+\A)^2})$.
 \end{rem}

 Fix a point $(\x, \y)$ in the plane, with $0 \leq \x \leq \a$ and $\y \geq \x - \a$, and let $z_{\pm}$ be the corresponding critical points. In the asymptotic analysis that follows, the existence of a closed contour $C_0$ on which $\Re (S(z) - S(z_0)) = 0$ (where $z_0$ is one of the two critical points), and which passes to the right of $1$ (but the left of $\exp(\x) = \frac{1}{\X}$) and to the left of $0$ (thus encompassing the cut on the interval $[0, 1]$), will be the key ingredient in the proof. Such a contour, for $(\x, \y)$ in the critical region and thus for complex conjugate critical points ($z_0 = z_+$ in that case though the choice makes no difference as $S(\overline{z}) = \overline{S(z)}$), is depicted in Figure~\ref{fig:contours_pp} (top left). For its existence, we argue as follows. We have the following limits at 0 and $-\infty$ (along the real axis, say):
 \begin{align}
\lim_{z \to 0-} \Re S(z) = \infty, \qquad \lim_{z \to -\infty} \Re S(z) = -\infty
 \end{align}
 
 \noindent and so by the intermediate value theorem there will be a point $z$ on the negative real axis with $\Re (S(z) - S(z_0)) = 0$. Likewise on the interval $(1, \frac{1}{\X})$ $\Re (S(z) - S(z_0))$ changes sign and we thus have a point $z$ in this interval with $\Re (S(z) - S(z_0)) = 0$. Connecting the contours in the upper and lower half-planes will yield the desired $C_0$. For a precise technical description of this argument see Lemma 6.4 in \cite{B2007cyl} and note that the $f$ in that statement is exactly our $S$ without the log term. For $\x = 0$ the preceding discussion becomes simpler. In this case, for $z$ on the unit circle $\{|z|=1\}$, we have, by direct computation, $2 \Re(S(z)) = S(z) + S(1/z) = 0$. Moreover in the case of complex conjugate critical points $z_{\pm}$ we have $|z_{\pm}|=1$ and so the $C_0$ is just the unit circle. 
 
 We finally remark that in the case $a \to \infty$ the above argument simplifies considerably and $C_0$ is just the circle around the origin of radius $\exp(\x/2)$.
 
 Before stating our first asymptotic result, we fix a few useful notations. We concentrate on the case $0 < \x < \x_T$, but the case $\x_T < \x < \a$ can be treated similarly. We will study the case $\x = 0$ separately. 
 
 For $(\x, \y) \in \mathcal{L}$, we denote by $\gamma_+$ any simple counter-clockwise contour (path) joining the two corresponding critical points $z_+$ and $z_-$ just to the right of 1. It depends of course on $\theta(\x,\y)$ which is the argument of $z_+$. By $\gamma_-$ we denote a clockwise contour (path) joining the same two points but which passes to the left of 0. We state the result for $i' \leq i$. For $i > i'$ the only difference is one replaces $\gamma_+$ by $\gamma_-$.
 
 Finally, one only needs to look in the half-space $\y \geq \x - \a$. The reason for this is combinatorial: below the line $\y = \x - \a$ we will see only particles, as the partition $\mu^{(N-i)}$ has at most $i$ parts due to the interlacing constraints. Therefore the kernel will not be of interest around points $(\x, \y)$ in the half-space given by $y < \x - \a$.

 \begin{thm} \label{thm:k12_pp}  
  Let $(\x,\y) \in (0,\a) \times \R$ and  $\y > \x - \a$. Let 
  \begin{equation}
   (i, i') = \left( \left\lfloor \frac{\x}{r} \right\rfloor + \i, \left\lfloor \frac{\x}{r} \right\rfloor + \i' \right), \qquad
   (k, k') = \left( \left\lfloor \frac{\y}{r} \right\rfloor + \k, \left\lfloor \frac{\y}{r} \right\rfloor + \k' \right)
  \end{equation}
  where $\i', \i \in \N, \k, \k' \in \Z'$ are fixed. Then $ri, r i' \to \x$ and $r k, r k'\to \y$ as $r \to 0+$. When $rN \to \a$ we find that
\begin{equation}
  \lim_{r \to 0+} K_{1,2}(i, k; i', k') = 
  \begin{cases}  
  0, & \text{if } \y > \y_+, \\
  \frac{1}{2 \pi \im} \int_C (1-e^{-\x} z)^{\i-\i'} z^{\k'-\k-1} \dx z, & \text{otherwise } 
  \end{cases}
\end{equation}
where $C$ is 
\begin{itemize}
 \item $\gamma_{\pm}$ if $\y_- \leq y \leq \y_+$ ($\gamma_+$ if and only if $\i' \leq \i$);
 \item a positively oriented circle of radius $1 + \epsilon$ centered at the origin for some $0 < \epsilon \ll 1$ if $\y < \y_-$ and $\x < \x_T$;
 \item a positively oriented circle containing the cut $\mathcal{P}$ but passing to the right of 1 if $\y < \y_-$ and $\x > \x_T$.
\end{itemize}
 \end{thm}

 \begin{proof}
   We will give a similar argument to the one given in Section 3.1 of \cite{or}. The interested reader will note the same argument has been applied before (modulo notation, conventions, and some minor technical details), in various related models (for both normal and/or strict as opposed to symmetric/free boundary plane partitions --- see \cite{or, vul2, fs, bmrt}) and the analysis carries over almost mutatis mutandis.

   Throughout the proof we restrict to the case $\i' \leq \i$, in which case the $z$ contour is on the outside in the double contour integral formula for $K_{1,2}$. The other case follows similarly. We also write $S(z)$ in lieu of $S(z; \x, \y)$ whenever possible.
   
   The idea is that we deform the original integration contours for $K_{1,2}$ around the complex plane and have them pass through the critical point $z_+$ (and in the case of complex conjugate critical points, both $z_+$ and $z_-$). We want to make $\Re (S(z)-S(z_+))$ negative everywhere except at $z_+$ (or $z_{\pm}$ in the complex conjugate case), while making $\Re (S(w)-S(z_+))$ positive everywhere except $z_+$ (respectively $z_{\pm}$). We then observe $S(z)-S(w) = (S(z)-S(z_+)) - (S(w)-S(z_+))$ and employ multiple times the following simple limit:
   \begin{align} \label{basic_asymptotics}
    \int_\gamma \exp \left( \frac{1}{r} f(z) \right) \dx z \to 0, \ r \to 0+
   \end{align}

   \noindent if $f$ is smooth and $\Re f < 0$ for all but finitely many points along the simple closed contour $\gamma$.
   
   In the four panels of Figure~\ref{fig:contours_pp} we illustrate, for various situations that will arise, the level lines of $\Re(S(z)-S(z_+))$. The first three figures correspond to points $(\x, \y)$ with $\x < \x_T$, sitting above the arctic curve, inside the liquid region $\mathcal{L}$, and below the arctic curve but above the line $\x = \y + \a$ respectively. The last corresponds to $(\x, \y)$ below the arctic curve but with $\x > \x_T$, the only case that needs special treatment when $\x > \x_T$.

   $K_{1,2}$ is explicitly given in Theorem~\ref{thm:pp_corr}. In the limit, the integrand $\frac{F(N-i, z)}{F(N-i', w)} \frac{w^{k'+1}}{z^{k+1}}$ is approximated by
   \begin{align} \label{eq:app_integrand}
     \exp \frac{1}{r} \left( S(z; \x, \y) - S(w; \x, \y) \right)
   \end{align}
   
   \noindent and it is this function that will provide the dominant asymptotic contribution.

   Throughout the proof, contours of integration for $z$ and $w$ will move around. One has to take care that the $z$ contour never crosses the interval $[\frac{1}{\X}, \frac{1}{A}]$ where the poles of the function $F$ accumulate in the limit (but the $w$ contour can certainly cross this interval), and that the $w$ contour never crosses the interval $[0,1]$ where the zeros of the function $F$ (so poles in the $w$ variable) accumulate in the limit (and again, the $z$ contour is of course allowed to cross said interval). None of the operations described below move the $z$ or $w$ contours in a way that their respective forbidden intervals are crossed, as can be explicitly checked case by case.

   First, the case $\y > \y_+, \x < \x_T$. We deform the contours so that the $z$ contour, which is on the outside, passes through $z_+ \in (1, \exp(\x))$ at an angle orthogonal to the real axis --- which is locally the direction of steepest descent for $\Re S$, and otherwise contains the contour $C_0$ where $\Re(S(z)-S(z_+))=0$ (to ensure everywhere else $\Re(S(z)-S(z_+))<0$). Similarly we deform the $w$ contour so that it is contained in $C_0$ (so that $\Re(S(z)-S(z_+))>0$) and passes through $z_+$ parallel to the real axis. See Figure~\ref{fig:contours_pp} (top right). We observe that the factor $1/(z-w)$ does not cause problems as it is integrable:  we can bound $\int \frac{\dx z \dx w}{z-w}$ by the converging $\int_{-\delta}^{\delta} \int_{-\delta}^{\delta} \frac{\dx x \dx y}{\sqrt{x^2+y^2}} = 8 \delta \log(1+\sqrt{2})$ for some small positive real $\delta$. We conclude the integral decays exponentially fast to 0 in the limit $r \to 0+$.

   Second, the case $\y \in (\y_-, \y_+), \x < \x_T$, which is to say $(\x, \y) \in \mathcal{L}$. We proceed as before by passing both contours through the two critical points $z_+$ and $z_-$, so that they intersect orthogonally at both. Based on the gradient of $\Re (S(z)-S(z_+))$  the final contours look like in Figure~\ref{fig:contours_pp} (top left). We note that for this to be possible, we have to pass the $w$ contour to the outside of the $z$ contour on an arc passing through the left of 1 and connecting $z_+$ with $z_-$ (this arc is $\gamma_+$). Doing so we pick up the residue at $z = w$ for every $z$ between $z_+$ and $z_-$ along $\gamma_+$. The total contribution is finite and equal to
   \begin{equation}
    \begin{split}
     & \frac{1}{2 \pi \im} \int_{\gamma_+} Res_{z \to w} \left(\frac{F(N-i, z)}{F(N-i', w)} w^{k'-1/2} z^{-k-1/2} \frac{zw-1}{(z+1)(w-1)(z-w)}\right) \dx z = \\
     & \qquad = \frac{1}{2 \pi \im} \int_{\gamma_+} \frac{F(N-i, w)}{F(N-i', w)} w^{k'-k+1} \dx w \to \frac{1}{2 \pi \im}  \int_{\gamma_+} (1-e^{-\x} w)^{\i-\i'} w^{\k'-\k-1} \dx w, \ r \to 0+
   \end{split}
  \end{equation}
where to take the limit we have used the estimates from Appendix~\ref{sec:theta}. We now argue that the integral on the remaining contours vanishes in the limit. The remaining integrals will converge to 0 exponentially again using the basic fact from \eqref{basic_asymptotics}. Again the denominator factor $1/(z-w)$ poses no further issues. 

   The situation repeats for a third time, for $\y < \y_-, \y > \x - \a, \x < \x_T$. We pass both contours through $z_+ < 0$, but now as in the previous case the gradient of $\Re (S(z)-S(z_+))$ forces the $w$ contour on the outside and the $z$ on the inside, the $z$ contour passing orthogonally to the real axis while the $w$ contour passing parallel to it at $z_+$. See Figure~\ref{fig:contours_pp} (bottom left). As we have to pass the whole $w$ contour to the outside, we pick up a residue along the line $z=w$ over a whole closed contour passing through $z_+$ and to the left of 1 of contribution: 
   \begin{equation}
    \begin{split}
     & \frac{1}{2 \pi \im} \int_C Res_{z \to w} \left(\frac{F(N-i, z)}{F(N-i', w)} w^{k'-1/2} z^{-k-1/2} \frac{zw-1}{(z+1)(w-1)(z-w)}\right) \dx z = \\
     & \qquad = \frac{1}{2 \pi \im} \int_C \frac{F(N-i, w)}{F(N-i', w)} w^{k'-k+1} \dx w \to \frac{1}{2 \pi \im}  \int_C (1-e^{-\x} w)^{\i-\i'} w^{\k'-\k-1} \dx w, \ r \to 0+
    \end{split}
  \end{equation}
where again we have used Appendix~\ref{sec:theta}. The remaining integrals, as before, converge exponentially fast to 0 in the limit.

For $\x > \x_T$, the two cases $\y > \y_+$ (above the arctic curve) and $\y \in (\y_-, \y_+)$ (inside the liquid region) are handled using the same contours as if it were in the regime $\x < \x_T$. For the third case $\y < \y_-$, the situation is somewhat different as now the critical points $z_{\pm}$ are in the interval $(\exp(\a), \infty)$. We pass both contours through $z_+$ as before. Since the $z$ contour cannot cross $\mathcal{P}$, we enlarge it on the left side and have it pass through infinity until it comes back on the other side and encircles $\mathcal{P}$. We can do this as there is no residue at $\infty$ in the $z$ variable. Locally at $z_+$ the final contour $C_z$ is parallel to the axis. We now inflate the contour $C_w$, passing it over $C_z$ and picking up the stated residue over the desired contour $C$, so that it goes through $z_+$ locally perpendicular to the real axis. See Figure~\ref{fig:contours_pp} (bottom right). Further note that if $\i = \i', \k=\k'$ (the diagonal of the kernel), the residue integral over $C$ is 0, since the origin is no longer inside $C$.
\end{proof}

\begin{figure}[!ht]
  \begin{center}
   \includegraphics[scale=0.29]{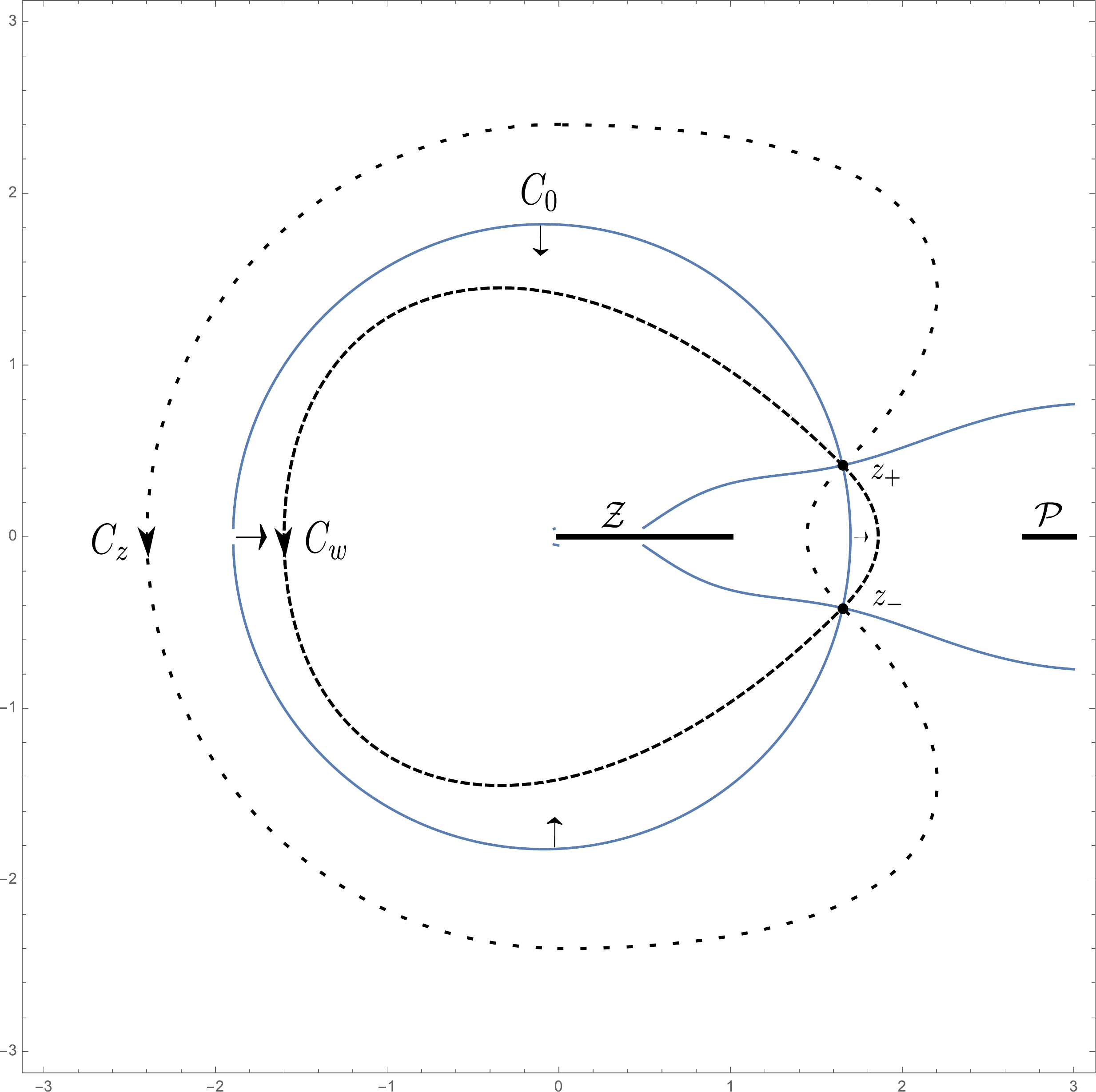} \quad \includegraphics[scale=0.29]{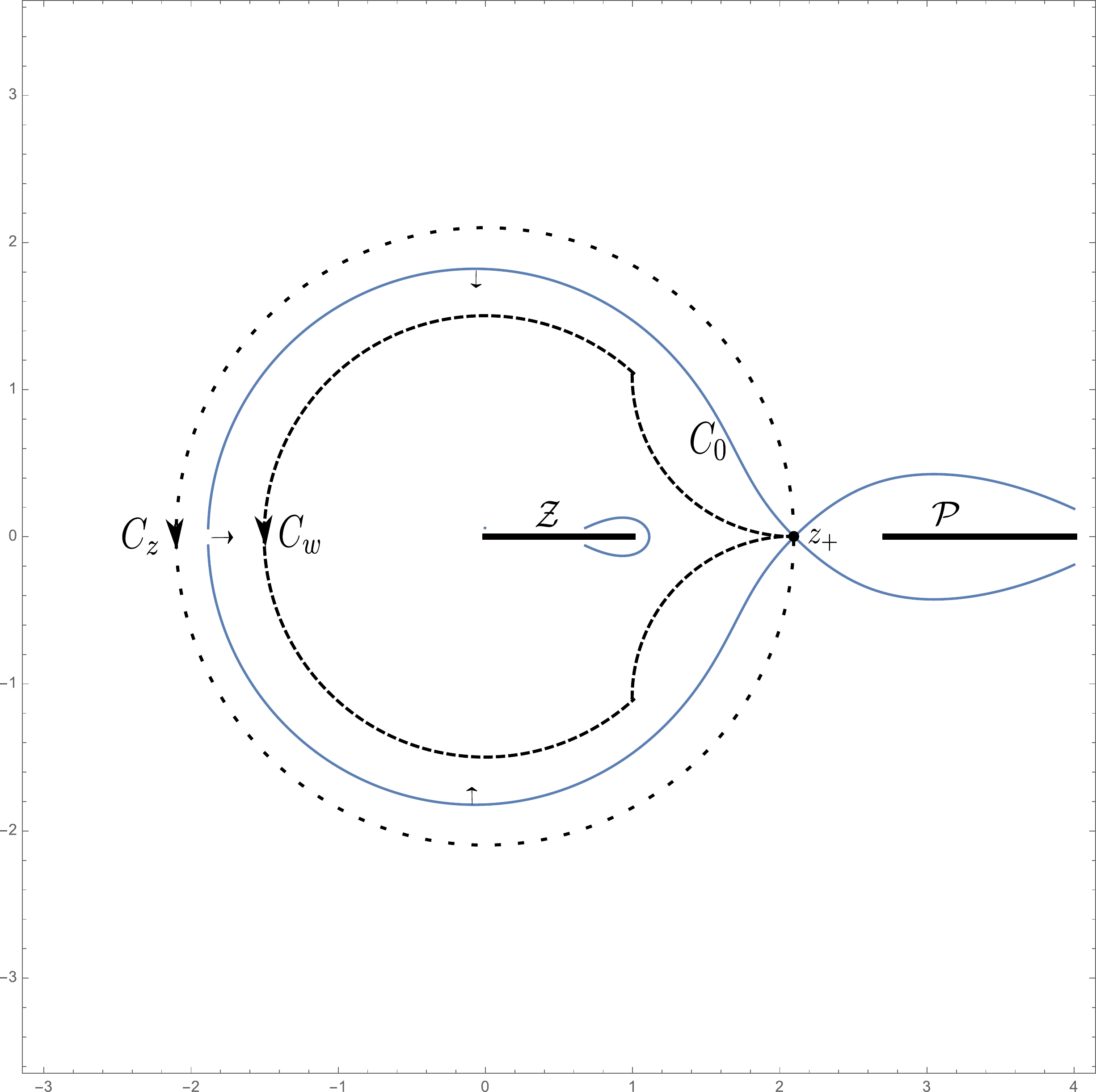}
  \end{center}

  \begin{center}
    \includegraphics[scale=0.29]{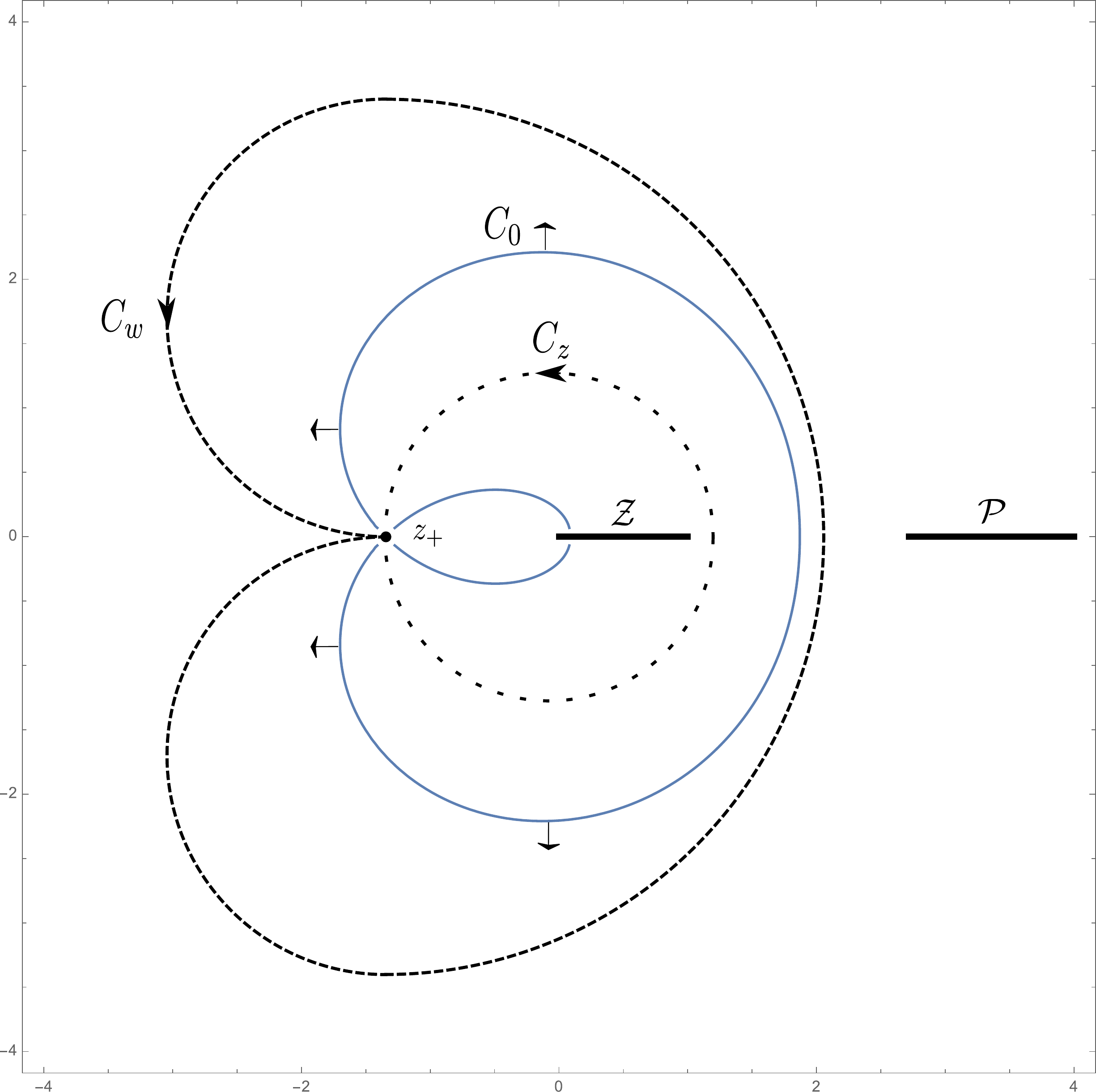} \quad \includegraphics[scale=0.29]{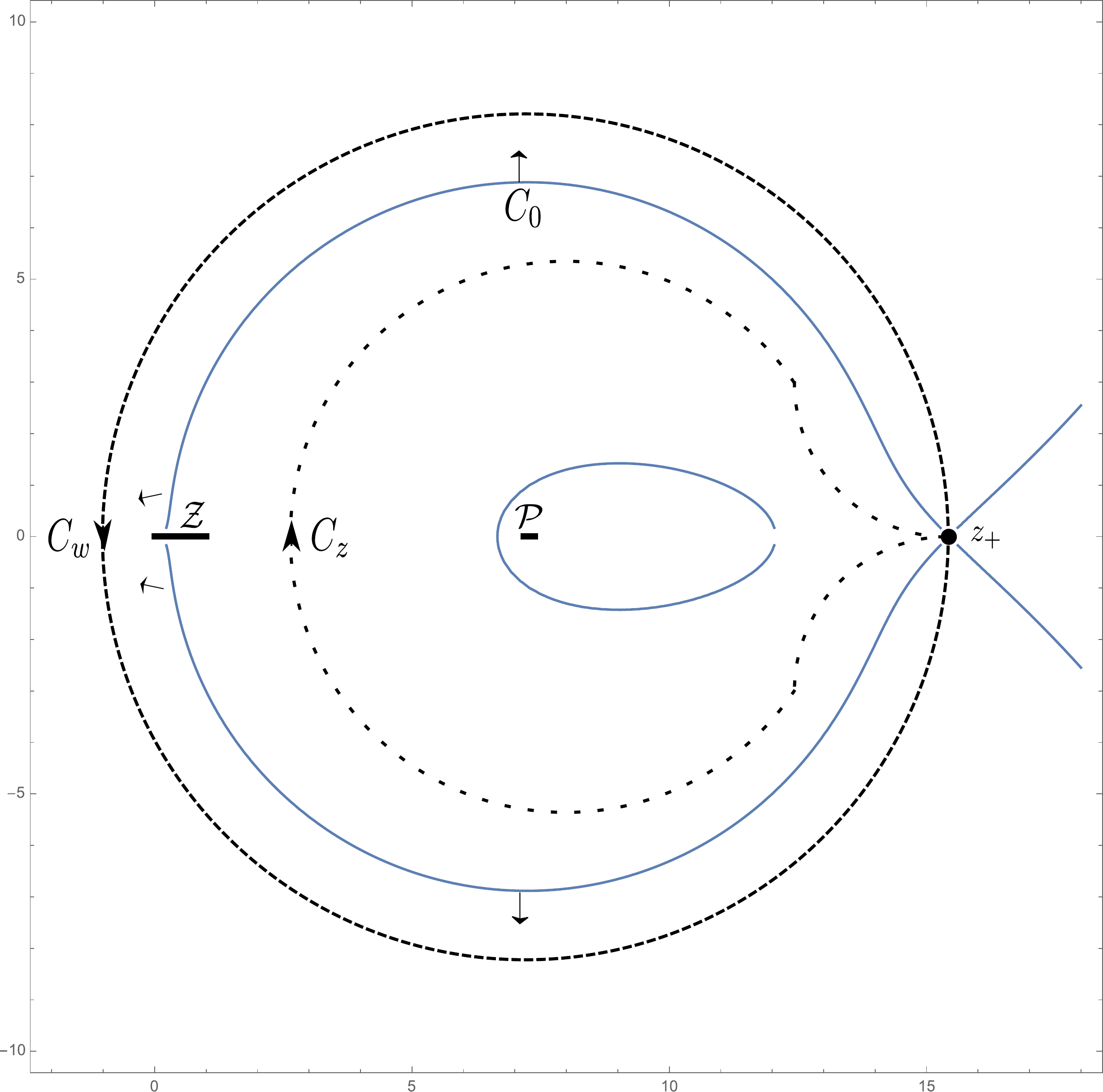}
   \end{center}
   \caption{\fs{The curves (solid) $\Re (S(z) - S(z_+)) = 0$ with arrows in the direction of ascent and the final $z$ and $w$ contours from the proof of Theorem~\ref{thm:k12_pp}. Top left: $z_+$ is one of the two complex conjugate critical points of $S$ corresponding to the liquid region; top right: $z_+$ is real corresponding to the region above the arctic curve; bottom left: $z_+$ is real and corresponds to the region below the arctic curve and $\x < \x_T$; bottom right: like bottom left but with $\x > \x_T$.}}
   \label{fig:contours_pp}
\end{figure}

\begin{rem}
 We mention one qualitative difference between the case $\x < \x_T$ handled in Figure~\ref{fig:contours_pp} (top left, top right, bottom left) and the case $\x > \x_T$. If $\x-\a < \y < \y_-$, asymptotically around $(\x, \y)$ (below the arctic curve), one sees only particles if $\x < \x_T$, while for $\x > \x_T$ one sees only holes. Both regions are frozen, but in different ways. See Figure~\ref{fig:arctic_curve_pp} for a numerical visualization of this.
\end{rem}

\begin{rem} \label{rem:incomplete_beta}
 Up to a change of variables $z \mapsto e^{-\x} z$ which introduces extra factors of the form $e^{\x(\k'-\k)}$ that nevertheless cancel in any pfaffian computation (as they are diagonal), the kernel
 \begin{equation}
  (\Delta \i, \Delta \k) \mapsto \frac{1}{2 \pi \im} \int_{\gamma_{\pm}} (1-z)^{\Delta \i} z^{-\Delta \k-1} \dx z
 \end{equation}
is called \emph{the incomplete beta kernel} \cite{or}; in our case $\Delta \i = \i - \i', \Delta \k = \k - \k'$. When we restrict ourselves to the same slice $\i= \i'$ a simple integration shows it becomes \emph{the discrete sine kernel}
 \begin{equation}
  (\k, \k') \mapsto \frac{\sin(\theta (\k - \k'))}{\pi (\k-\k')}
 \end{equation}
where $\theta = \theta(\x, \y)$ is the argument of $z_+ = z_+(\x, \y)$.
\end{rem}

Theorem~\ref{thm:k12_pp} justifies the word \emph{arctic curve} for $\mathcal{L}$ in the following sense: if one looks at a large system and scales appropriately, in the particle--hole description, above $\mathcal{C}$ we will see only holes in the sense that the probability of seeing a particle decays exponentially; below $\mathcal{C}$ we will see only particles --- the probability of seeing particles is exponentially close to 1 that is, as can be verified from the explicit limit of the diagonal elements of the kernel $K(i,k; i,k)$. Inside $\mathcal{C}$ we will see a mixture, and this region is called the \emph{liquid region} --- denoted above by $\mathcal{L}$. We illustrate this in Figure~\ref{fig:arctic_curve_pp}. 

Theorem~\ref{thm:k12_pp} also gives us the limiting density of particles, given below. Similar results for various related models of plane partitions can be found in e.g.\ \cite[Section 2]{ck}, \cite[Corollary 2]{or}, \cite[Sections 4 and 8]{ken2} and~\cite[Section 3]{bf}.
\begin{cor}
 In the limit \eqref{eq:var_scaling_pp} and around macroscopic point $(\x, \y)$ the density of particles is
\begin{equation}
    \rho(\x, \y) = \frac{\theta(\x, \y)}{\pi}, \qquad \theta(\x, \y) = \arg(z_+(\x,\y)).
\end{equation}
\end{cor}

Note that, for $\x < \x_T$, the density is 0 above $\mathcal{C}$ (as the two critical points are real so $\theta = 0$), strictly between 0 and 1 inside $\mathcal{L}$ and 1 below (as the two roots are negative below $\mathcal{C}$). For $\x > \x_T$, the density is 0 above $\mathcal{C}$, strictly between 0 and 1 when inside $\mathcal{L}$ where the two roots are complex conjugate, 0 again (!) for $\y$ below $\mathcal{C}$ but above the line $\x = \a + \y$, and 1 below $\mathcal{C}$ and below the aforementioned line. This can immediately be translated into the density of horizontal tiles in the plane partition picture via the change of variables $\y = \h + \frac{\x}{2}$ --- here $\h = \lim_{r \to 0+} hr$ is the macroscopic ordinate/height for lozenges.

In the case $\a \to \infty$ ($\A \to 0$) the arctic curve $\mathcal{C}$ becomes particularly simple. It can be written as half the zero locus (the part $\x > 0$) of
\begin{equation}
  \begin{split}
 (1+U+V)(1+U-V)(1-U+V)(1-U-V) = 0, \qquad (U, V) = (e^{-\frac{\x}{2}},e^{-\frac{\y}{2}}).
  \end{split}
\end{equation}
It is thus half the boundary of the amoeba of the polynomial $p(U,V) = 1 + U + V$ \cite{ko, ko2, kos} and can be recovered using the techniques in \cite{ko}, independent of the description as a Schur process. We recall that the amoeba of a polynomial $p(U,V) \in \C[U,V]$ is the set 
\begin{align}
 \{ (\log |U|, \log|V|): (U,V) \in (\C \setminus 0)^2, p(U,V)=0 \}.
\end{align}

We make a small digression and address the three-dimensional nature of the picture. It is possible to recover the three-dimensional limit surface from the above formulas. This was first obtained by Cerf and Kenyon \cite{ck} in the mathematical literature and by Bl\"ote, Hilhorst and Nienhuis in the physics literature \cite{bnh}, and while we can recover their formulas, as already noted in \cite{or} (whose exposition we follow for this purpose), further analysis --- in particular a concentration inequality type of result --- is needed to make the computations rigorous (but see \cite{ck} where this is done using an alternative method). We also note that all the cited references deal with regular (non-symmetric) plane partitions, but the limit surface is the same after a mild reparametrization (so that indeed the symmetric plane partitions, as opposed to just the free boundary halves, are distributed according to the same $q^{\text{Volume}}$ measure as the ordinary plane partitions). Here and below we consider $\rho$ as a function of $(\x, \h)$, and we extend to the symmetric case by $\rho(\x, \h) = \rho(-\x, \h)$. We denote $X, Y, Z$ the three-dimensional coordinates of space, with $Z$ the vertical, $X$ the south-west and $Y$ the south-east coordinates respectively in the $(1,1,1)$ projection. Then a parametrization of the limit surface, depicted in Figure~\ref{fig:pp_large} for the case of unbounded bottom, is:
\begin{equation}
 Z(\x, \h) = \int_{-\infty}^{\h} (1-\rho(\x,s))ds, \qquad X(\x, \h) = Z(\x, \h) - \h - \frac{\x}{2}, \qquad Y(\x, \h) = Z(\x, \h) - \h + \frac{\x}{2}.
\end{equation}

\begin{SCfigure}[][ht]
  \includegraphics[scale=0.5]{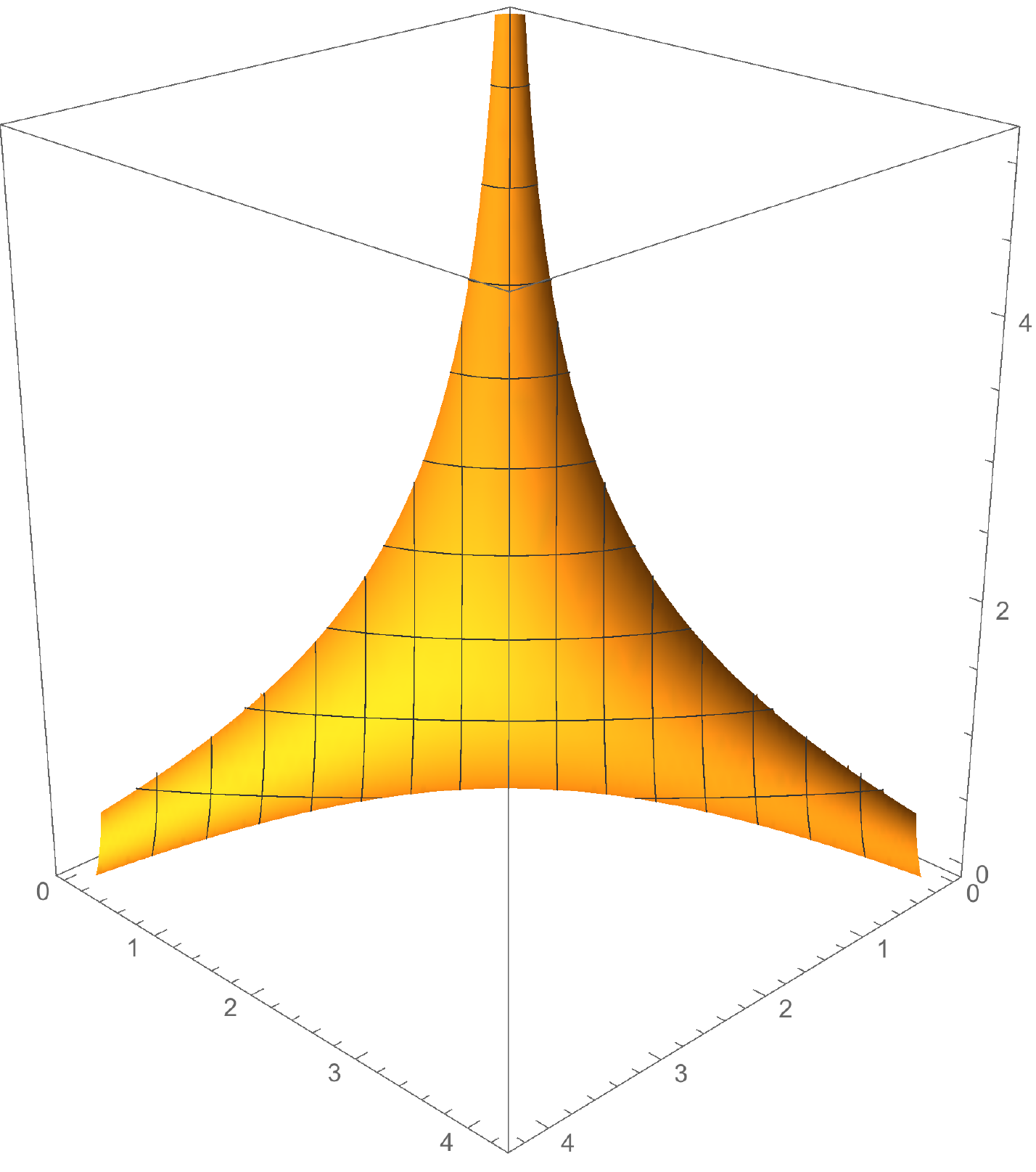}
  \caption{\fs{The limit surface for both symmetric and non-symmetric random plane partitions.}} 
  \label{fig:pp_large}
  \end{SCfigure}

We now address, again in the case $\x > 0$, the limits of $K_{1,1}$ and $K_{2,2}$. It will be easier to work with the conjugated kernel $\tilde{K} (i, k; i', k')$ defined thusly. We fix $(\x, \y)$ and consider the corresponding $z_+$. Then (with $S(z_+) := S(z_+; \x, \y)$):
\begin{align}
 \left( \begin{array}{rr}
  \tilde{K}_{1,1} & \tilde{K}_{1,2} \\
  \tilde{K}_{2,1} & \tilde{K}_{2,2} 
 \end{array}
 \right) 
 = 
   \left( \begin{array}{rr}
  e^{- \Re S(z_+)/r} & 0 \\
  0 & e^{\Re S(z_+)/r} 
 \end{array}
 \right)
 \left( \begin{array}{rr}
  K_{1,1} & K_{1,2} \\
  K_{2,1} & K_{2,2} 
 \end{array}
 \right) 
    \left( \begin{array}{rr}
  e^{-\Re S(z_+)/r} & 0 \\
  0 & e^{ \Re S(z_+)/r} 
 \end{array}
 \right).
 \end{align}
 
It follows that diagonal elements of $\tilde{K}$ converge to zero.

 \begin{thm} \label{thm:diagonal_pp}
With the same assumptions as in Theorem~\ref{thm:k12_pp}, in the asymptotic regime given by \eqref{eq:var_scaling_pp} we find that 
$\lim_{r \to 0+} \tilde{K}_{1,1}(i, k; i', k') = \lim_{r \to 0+} \tilde{K}_{2,2}(i, k; i', k') = 0$.
 \end{thm}

\begin{proof}
 It turns out the analysis was already carried out in the proof of Theorem~\ref{thm:k12_pp}, with one particular difference: now we can move both the $z$ and the $w$ contours around and deform them into one and the same contour, called $\gamma_{1,1}$ for $\tilde{K}_{1,1}$ and $\gamma_{2,2}$ for $\tilde{K}_{2,2}$ --- we can do this as there is no more residue/singularity to take care of at $z=w$ as the diagonal elements of the kernel do not contain a $(z-w)$ term in the denominator of the integrand. In the frozen regions ($\y > \y_+$ or $\y < \y_-$) the contours both pass through $z_+$, while in the liquid region ($\y \in (\y_-, \y_+)$) both pass through $z_+$ and $z_- = \overline{z_+}$ respectively. In all three cases they intersect in the critical point(s) at right angles, and away from them $\gamma_{1,1}$ follows a direction of descent, while $\gamma_{2,2}$ one of ascent. In fact in all three cases, $\gamma_{1,1}$ can be taken to be the final $z$ contour from the proof of Theorem~\ref{thm:k12_pp}, while $\gamma_{2,2}$ the final $w$ contour. See Figure~\ref{fig:contours_pp} (top left, top right, bottom left): $\gamma_{1,1} = C_z$ and $\gamma_{2,2} = C_w$.
 
 More precisely, the integrands in $\tilde{K}_{1,1}$ and $\tilde{K}_{2,2}$ can be approximated by 
\begin{equation}
    \exp \frac{1}{r} \left( \tilde{S}(z; \x, \y) + \tilde{S}(z; \x, \y) \right) \text{ and\ } \exp -\frac{1}{r} \left( \tilde{S}(z; \x, \y) + \tilde{S}(z; \x, \y) \right)
\end{equation}
respectively as $r \to 0+$ where $ \tilde{S}(z; \x, \y) = S(z; \x, \y) - \Re S(z_+; \x, \y)$ with the observation that throughout the proof the dependence on $\x, \y$ will be omitted.

 The important remark is that $\tilde{S}$ has real part zero at $z_+$ (and in the liquid region, at $z_-$ as well), and so we deform the two contours into one and the same such that for $\tilde{K}_{1,1}$ the real part of $S(z)-S(z_+)$ decreases away from $z_+$ (and so will become negative), while for $\tilde{K}_{2,2}$ it increases away from $z_+$ thus becoming positive. Then in both cases the integrals will converge to 0 as $r$ goes to 0 using the basic asymptotic fact \eqref{basic_asymptotics}. We finally remark that for $\tilde{K}_{1,1}$ deformation of the contours has to avoid the interval $(\frac{1}{\X}, \frac{1}{\A})$, while for $\tilde{K}_{2,2}$ we need to avoid the interval $(0,1)$ --- both of which can be achieved.
 \end{proof}

\begin{rem}
$\Re S(z_+)$ can be made more explicit as both $S$ and $z_+$ have explicit formulas, but we decided to avoid this to streamline the proof. In the case $\a \to \infty$ the argument (originally the one from \cite{or}) simplifies considerably, and it will be presented, in a slightly modified form (i.e., for a different model) in Section~\ref{sec:overpartitions}.
\end{rem}

The discussion above, especially Theorems~\ref{thm:k12_pp} and \ref{thm:diagonal_pp} and Remark~\ref{rem:incomplete_beta}, allows us to formulate the main result as: ``away from the arctic curve, for $\x > 0$, the local pfaffian correlations, in the limit, become determinantal with kernel given by the incomplete beta kernel.''

\begin{thm} \label{thm:pp_bulk_overall}
    Let $n > 0$ be a natural number, $\a \in (0, \infty)$, $\x, \y$ reals with $\x \in (0, \a), \y > \x-\a$,  and for $1 \leq s \leq n$ fix $n$ pairs $\i_s \in \N, \k_s \in \Z'$. As $r \to 0+$ assume we have $i_s \in \N, k_s \in \Z'$ depending on $r$ and converging $r i_s \to \x$, $r k_s \to \y$ thusly:
   \begin{align}
   i_s = \left\lfloor \frac{\x}{r} \right\rfloor + \i_s, \qquad k_s = \left\lfloor \frac{\y}{r} \right\rfloor + \k_s, 
   \end{align}

   \noindent and denote $U = \{ (i_1, j_1), \dots, (i_n, j_n)\}.$ Then in the asymptotic regime given by \eqref{eq:var_scaling_pp} we have
   \begin{align}
    {\varrho}(U) \to
    \begin{cases} 
    0, & \text{ if } \y > \y_+ \text{ or } (\x > \x_T \text{ and } \y < \y_-), \\
    \det_{1 \leq u, v \leq n} B(\Delta \i_{uv}, \Delta \k_{uv}),  & \text{ if } \y \in (\y_-, \y_+), \\
    1, & \text{ if } \x < \x_T \text{ and } \y < \y_-
       \end{cases}
   \end{align} 
   
   \noindent where $\Delta \i_{uv} = \i_u - \i_v, \Delta \k_{uv} = \k_u - \k_v$, $B(\Delta \i, \Delta \k)$ is the incomplete beta kernel 
   \begin{align}
    (\Delta \i, \Delta \k) \mapsto \frac{1}{2 \pi \im} \int_{\gamma_{\pm}} (1-z)^{\Delta \i} z^{-\Delta \k-1} \dx z
   \end{align}
   
   \noindent with $\gamma_{\pm}$ connecting the two critical points $z_{\pm}$, $\gamma_+$ passing to the right of 0 (and $\gamma_-$ passing to the left), with $\gamma_+$ being chosen in the case $\i_v \leq \i_u$ (and $\gamma_-$ otherwise).
\end{thm}

In the case $\x = 0$ ($\X = 1$), the limiting regime $r \to 0+$ for the local microscopic coordinates becomes $(i, k)$ 
\begin{equation}
  i = \i, \qquad k = \left\lfloor \frac{\y}{r} \right\rfloor + \k
\end{equation}
which for $i$ is of course a relabeling we nevertheless use to be consistent with our notation so far. Because $\y_+ = \infty$, the coordinate $\y$ is either in the liquid region ($\y > \y_- = - \log \frac{4}{(1+ \A)^2}$) in which case the double critical points are complex conjugate of modulus 1, or in the frozen region ($-\a < \y < \y_-$) in which case the double critical points are real negative. Moreover, on the circle $\{ |z| = 1 \}$ (of interest in the liquid region) the real part of $S$ is constant and equal to 0:
\begin{align}
2 \Re S(z) = S(z) + S(1/z) = 0.
\end{align}

\noindent with values increasing (positive) and decreasing (negative) outside (remark that in the notation from above $C_0 = \{ |z| = 1\}$).

When one passes the contours through the critical points, the preceding arguments work almost unchanged. The few observations we make are as follows. 

First, the contours for $z$ and $w$ exist even when $\x = 0$ since for any finite $q = \exp(-r)$ there is enough space to the right of 1 between the largest zero of $F$ ($ = q < 1$) and its smallest pole ($ = q^{-\i-1} > 1$). Thus, even though we have a pole at $1$ in the integrand for $K_{1,2}$ and $K_{2,2}$, this will never give a residual contribution as the contours can avoid passing through 1. 

Second, in both cases --- $\y$ inside/outside the liquid region, we have to exchange the two contours (and if $\y$ is inside the liquid region, both have to pass through $z_+ = \overline{z_-}$), in which case we will pick up the contributions from the residues along the hyperbola $z w = 1$ in the analysis of $K_{1,1}$ and $K_{2,2}$. Their contributions are, in the limit $r \to 0+$:
\begin{equation}
  \begin{split}
Res_{z \to \frac{1}{w}} K_{1,1} &\to  \int_{\gamma_+} (1-w)^{\i} \left(1-\frac{1}{w}\right)^{\i'} w^{\k - \k' - 1} \frac{1-w}{1+w} \frac{\dx w}{2 \pi \im},\\
Res_{z \to \frac{1}{w}} K_{2,2} &\to  \int_{-\gamma_-} (1-w)^{-\i} \left(1-\frac{1}{w}\right)^{-\i'} w^{\k' - \k - 1} \frac{w+1}{w-1} \frac{\dx w}{2 \pi \im}
  \end{split}
\end{equation}
where $\gamma_+$ is a closed contour passing through $z_+ < 0$ and to the right of 1 if $\y < \y_-$ or otherwise joins the two complex conjugate critical points (again to the right of 1) if $\y > \y_-$. The minus sign in front of $\gamma_-$ in the second integral can be brought inside the integral and appears due to the fact $\gamma_-$ is oriented bottom-to-top, while our $z$ and $w$ contours are counterclockwise.

Third, $\tilde{K} = K$ if $\y > \y_-$ as the real part of $S$ at the critical points is 0.

In view of the above, we have shown the correlations remain pfaffian when $\x=0$ with a kernel we have explicitly computed. We formally state the result.

\begin{thm} \label{thm:boundary_corr_pp}
    Let $n > 0$ be a natural number, $\a \in (0, \infty)$, $\y > -\a$ real,  and for $1 \leq s \leq n$ fix $n$ pairs $\i_s \in \N, \k_s \in \Z'$. As $r \to 0+$ assume we have $i_s \in \N$, as well as $k_s \in \Z'$ depending on $r$ converging thusly:
    \begin{equation}
   i_s = \i_s, \qquad k_s = \left\lfloor \frac{\y}{r} \right\rfloor + \k_s.
    \end{equation}
   
  \noindent Denote $U = \{ (i_1, j_1), \dots, (i_n, j_n)\}.$ Then in the asymptotic regime given by \eqref{eq:var_scaling_pp} the correlations ${\varrho}(U)$ converge to $1$ if $\y < \y_-$ and are otherwise pfaffian with $2 \times 2$ matrix kernel given by:

  \begin{equation}
    \begin{split}
    \mathsf{K}_{1,1}(\i, \k; \i', \k') &= \int_{\gamma_{+}} (1-z)^{\i} \left(1-\frac{1}{z}\right)^{\i'} z^{\k - \k' - 1} \frac{1-z}{1+z} \frac{\dx z}{2 \pi \im}, \\
    \mathsf{K}_{1,2}(\i, \k; \i', \k') &= \int_{\gamma_{\pm}} (1-z)^{\i - \i'} z^{\k' - \k - 1} \frac{\dx z}{2 \pi \im},\\
    \mathsf{K}_{2,2}(\i, \k; \i', \k') &= \int_{\gamma_{-}} (1-z)^{-\i} \left(1-\frac{1}{z}\right)^{-\i'} z^{\k' - \k - 1} \frac{1+z}{1-z} \frac{\dx z}{2 \pi \im}
    \end{split}
  \end{equation}
   
   \noindent where $\gamma_{\pm}$ are as in Theorem~\ref{thm:pp_bulk_overall} and $\gamma_+$ is taken if and only if $\i' \leq \i$.
   \end{thm}

   \begin{rem}
     \label{rem:fincrucial}
It is crucial that we keep $i_s=\i_s$ finite as $r \to 0+$ to obtain a pfaffian process. If instead we rescale
$i_s=f(r)+\i_s$ with $1 \ll f(r) \ll r^{-1}$, then we obtain a determinantal process with the kernel $\mathsf{K}_{1,2}$. To see this, first observe that under this rescaling, by deforming the contours as before, we may still reduce the double contour integral representation for
$\mathsf{K}_{1,2}$ (resp.\ $\mathsf{K}_{1,1}$ and $\mathsf{K}_{2,2}$) to a single integral of the residue of the integrand at $z=w$ (resp.\ at $zw=1$). The expression for the residue involves some $q$-Pochhammer symbols which may be estimated using \eqref{eq:zqnas} --- namely $(z;q)_i = (1-z)^i e^{o(i)}$ for $i \ll r^{-1}$.
For $\mathsf{K}_{1,2}$, the residue has the same finite limit as before, since the dependency on $f(r)$ disappears in the ratio $\frac{F(N-i,w)}{F(N-i',w)}$. But this is not the case for $\mathsf{K}_{1,1}$ and $\mathsf{K}_{2,2}$ which involve respectively the product $F(N-i,w)F(N-i',w)$ and its inverse. Instead, we find that the dominant contribution to the integral comes from the endpoints,
since $|1-z|$ is maximal on $\gamma_+$ and minimal on $\gamma_-$ at the endpoints. Setting $R(i,z)=(z;q)_{i}$, we conclude that $R^{-1}(i,z_+)R^{-1}(j,z_+){K}_{1,1}$ and $R(i,z_+)R(j,z_+){K}_{2,2}$ both converge to zero. Using the conjugation trick of Proposition~\ref{ConjPfaff}, the correlation functions are found to be determinantal in the limit. 
\end{rem}

\section{Plane overpartitions}
 \label{sec:overpartitions}

A \emph{plane overpartition} is a plane partition where in each row the last occurrence of an integer can be overlined or not while all the other occurrences of this integer are not overlined, and in each column the first occurrence of an integer can be overlined or not while all the other occurrences of this integer are overlined. A plane overpartition with the largest entry at most $N$ and shape $\lambda$ can be recorded as a sequence of partitions
$
\emptyset \prec \lambda^{(1)} \prec'\lambda^{(2)}\prec\cdots \prec \lambda^{(2n-1)} \prec' \lambda^{(2N)}=\lambda
$
where $\lambda^{(i)}$ is the partition whose shape is formed by all fillings greater than ${N-i/2}$, where the convention is that $\overline{k}=k-1/2$. An example of a plane overpartition is given in Figure~\ref{fig:plane-overpartitions}.
\begin{SCfigure}[][h]
  { \begin{tikzpicture}[scale=0.6]
\draw[thick, smooth] (0,0)--(0,4)--(5,4);
\draw[thick, smooth] (0,0)--(1,0);
\draw[thick, smooth] (0,1)--(2,1);
\draw[thick, smooth] (0,2)--(4,2);
\draw[thick, smooth] (0,3)--(5,3);
\draw[thick, smooth] (1,0)--(1,4);
\draw[thick, smooth] (2,1)--(2,4);
\draw[thick, smooth] (3,2)--(3,4);
\draw[thick, smooth] (4,2)--(4,4);
\draw[thick, smooth] (5,3)--(5,4);
\node at (0.5,0.5) { 1};
\node at (0.5,1.5) {$\overline{3}$};
\node at (0.5,2.5) { 3};
\node at (0.5,3.5) { 4};
\node at (1.5,1.5) { $\overline{1}$};
\node at (1.5,2.5) { 3};
\node at (1.5,3.5) { $\overline{4}$};
\node at (2.5,2.5) { $\overline{3}$};
\node at (2.5,3.5) { $\overline{3}$};
\node at (3.5,2.5) {$\overline{2}$};
\node at (3.5,3.5) { 2};
\node at (4.5,3.5) {2};
\end{tikzpicture}
}
\caption{\fs{A plane overpartition $\emptyset \prec (1) \prec' (2) \prec (2,2) \prec' (3,3,1) \prec (5,3,1) \prec' (5,4,1) \prec (5,4,1,1) \prec' (5,4,2,1)$.}}  
\label{fig:plane-overpartitions}
\end{SCfigure}

A plane partition $\pi$ is called a (diagonally) strict plane partition if its diagonals $\pi^{(t)}=(\pi_{i,i+t})_{i\geq1}$ are strict partitions, i.e. strictly decreasing sequences of integers.  By deleting the overlines in a plane overpartition one obtains a strict plane partition. Conversely, a strict plane partition can be overlined to obtain a plane overpartition and there are $2^{\# \textrm{border components}}$ different ways to do it. A border component of a strict plane partitions is a set of rookwise-connected boxes (i.e., a connected ribbon/border strip) filled with the same number. The strict plane partition obtained by deleting the overlines in Figure~\ref{fig:plane-overpartitions} has five border components.

A measure that to a plane overpartition with the largest entry at most $N$ assigns a weight $q^{|\textrm{sum of all entries}|}$ is an HV-ascending Schur process with $t=1$ and $x_1=x_2=q^N, \dots   ,x_{2N-1}=x_{2N}=q^1$.

The asymptotics of strict plane partitions was studied in \cite{vul2} using a variant of the original  Schur process, where in the definition of the process the skew Schur $P$ and $Q$ functions were used instead of the standard Schur symmetric functions. There a strict plane partition  is represented by a finite point configuration where each point in the configuration corresponds  to one entry in the strict plane partition. Precisely, a strict partition can be represented by the set of its parts, since all parts are different, and hence a strict plane partition $\pi$ can be represented by a finite subset of  $\Z \times \Z_{>0}$, where $(t,x)$ belongs to it if $x$ is a part of $\pi^{(t)}$. See Figure~\ref{fig:overpartitions-pointconfig} (left). The set of blue points is the point configuration corresponding to the strict plane partition obtained by deleting the overlines in Figure~\ref{fig:plane-overpartitions}.

The above point configuration is not a suitable representation of a plane overpartition, since there are $2^{\# \textrm{border components}}$ different ways to overline a strict plane partition. We need a new set of point configurations to represent these different overlinings. We explain briefly how these point configurations are obtained and refer the reader to \cite{bcc} and \cite{ccv} for details and proofs.

Starting from the point configuration of the corresponding strict plane partition we first construct a set of nonintersecting paths, where each path corresponds to a row of the plane partition. Going from $\pi_{i,j}$ to $\pi_{i,j+1}$ we take an eastbound edge followed by $\pi_{i,j}-\pi_{i,j+1}$ southbound edges if $\pi_{i,j}$ is not overlined and a southeastbound edge followed by $\pi_{i,j}-\pi_{i,j+1}-1$ southbound edges if $\pi_{i,j}$ is overlined. See Figure~\ref{fig:overpartitions-pointconfig} (left). Once we have the nonintersecting paths we can produce a domino tiling.  Dominos are placed diagonally and cover $\Z \times \Z_{\geq0}$. It is possible to place them in such a way so that different types of dominos (divided by the direction of the position and the color of the top corner in a chessboard fashion coloring) correspond to different edges. See Figure~\ref{fig:overpartitions-pointconfig} (right). Finally,  place two particles (holes) at the center of each of the two domino squares if the top corner of that domino is black (white). The point configuration is the set of all particles. See Figure~\ref{fig:overpartitions-pointconfig} (right), where particles are shown in solid colors. It corresponds to the example from Figure~\ref{fig:plane-overpartitions}.

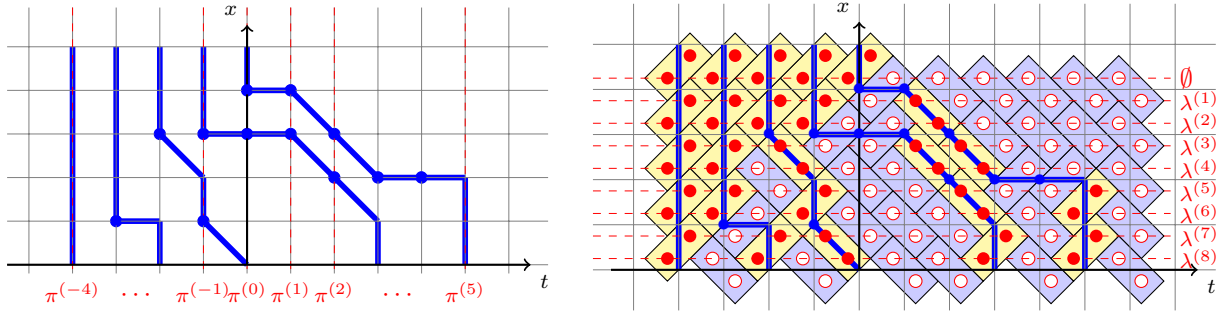
\begin{figure}[ht]
\begin{tikzpicture}[scale=0.58]
\foreach \x in {0}{ \foreach \y in {3,4}
 {
 \draw[line width=0.75mm,blue] (\x,\y) -- (\x+1,\y);
 \fill[blue] (\x,\y) circle[radius=4pt] ;
 }};

 \foreach \x in {-1}{ \foreach \y in {3}
 {
 \draw[line width=0.75mm,blue] (\x,\y) -- (\x+1,\y);
  \fill[blue] (\x,\y) circle[radius=4pt] ;
 }};

 \foreach \x in {-3}{ \foreach \y in {1}
 {
 \draw[line width=0.75mm,blue] (\x,\y) -- (\x+1,\y);
  \fill[blue] (\x,\y) circle[radius=4pt] ;
 }};

 \foreach \x in {3}{ \foreach \y in {2}
 {
 \draw[line width=0.75mm,blue] (\x,\y) -- (\x+1,\y);
  \fill[blue] (\x,\y) circle[radius=4pt] ;
 }};

 \foreach \x in {4}{ \foreach \y in {2}
 {
 \draw[line width=0.75mm,blue] (\x,\y) -- (\x+1,\y);
  \fill[blue] (\x,\y) circle[radius=4pt] ;
 }};

 \foreach \x in {1}{ \foreach \y in {3,4}
 {
 \draw[line width=0.75mm,blue] (\x,\y) -- (\x+1,\y-1);
  \fill[blue] (\x,\y) circle[radius=4pt] ;
}};

  \foreach \x in {2}{ \foreach \y in {2,3}
 {
 \draw[line width=0.75mm,blue] (\x,\y) -- (\x+1,\y-1);
  \fill[blue] (\x,\y) circle[radius=4pt] ;
}};

 \foreach \x in {-1}{ \foreach \y in {1}
 {
 \draw[line width=0.75mm,blue] (\x,\y) -- (\x+1,\y-1);
  \fill[blue] (\x,\y) circle[radius=4pt] ;
}};

 \foreach \x in {-2}{ \foreach \y in {3}
 {
 \draw[line width=0.75mm,blue] (\x,\y) -- (\x+1,\y-1);
  \fill[blue] (\x,\y) circle[radius=4pt] ;
}};

\foreach \x in {5}{ \foreach \y in {2,1}
 {
 \draw[line width=0.75mm,blue] (\x,\y) -- (\x,\y-1);
}};

\foreach \x in {3}{ \foreach \y in {1}
 {
 \draw[line width=0.75mm,blue] (\x,\y) -- (\x,\y-1);
}};

\foreach \x in {-1,-2,-3,-4}{ \foreach \y in {4,5}
 {
 \draw[line width=0.75mm,blue] (\x,\y) -- (\x,\y-1);
}};

\foreach \x in {-3,-4}{ \foreach \y in {2,3}
 {
 \draw[line width=0.75mm,blue] (\x,\y) -- (\x,\y-1);
}};

\foreach \x in {-1}{ \foreach \y in {2}
 {
 \draw[line width=0.75mm,blue] (\x,\y) -- (\x,\y-1);
}};

\foreach \x in {-2,-4}{ \foreach \y in {1}
 {
 \draw[line width=0.75mm,blue] (\x,\y) -- (\x,\y-1);
}};

\foreach \x in {0}{ \foreach \y in {5}
 {
 \draw[line width=0.75mm,blue] (\x,\y) -- (\x,\y-1);
}};

\draw[step=1cm,gray,very thin] (-5.5,-0.2) grid (6.9,5.9);
\foreach \x in {-4,-1,0,1,2,5} \draw[dashed,red,very thin] (\x,5.9)--(\x,-0.2)node[anchor=north] {\fs{$\pi^{(\x)}$}};
\node[red, very thin] at (-2.5,-0.7) {$\cdots$};
\node[red, very thin] at (3.5,-0.7) {$\cdots$};
 \draw[thick,->] (-5.5,0) -- (6.5,0) node[anchor=north west] {\fs{$t$}};
\draw[thick,->] (0,0) -- (0,5.5) node[anchor=south east] {\fs{$x$}};
\foreach \xy in {(-3,1),(-2,3),(-1,3),(-1,1),(0,4), (0,3),(1,4),(1,3),(2,3),(2,2),(3,2),(4,2)} {\node at \xy {$\color{blue} \bullet$};};
\end{tikzpicture}
\quad
\begin{tikzpicture}[scale=0.6]
\foreach \x in {0}{ \foreach \y in {3,4}
 {\filldraw [fill=blue!20!white, draw=black]  (\x-0.25,\y-0.25) -- (\x+0.25,\y-0.75) --  (\x+1.25,\y+0.25) -- (\x+0.75,\y+0.75)-- cycle;
 \draw[line width=0.75mm,blue] (\x,\y) -- (\x+1,\y);
 \filldraw[fill=white, draw=red] (\x+0.25,\y-0.25) circle[radius=4pt] ;
  \filldraw[fill=white, draw=red] (\x+0.75,\y+0.25) circle[radius=4pt] ;}};

 \foreach \x in {-1}{ \foreach \y in {3}
 {\filldraw [fill=blue!20!white, draw=black]  (\x-0.25,\y-0.25) -- (\x+0.25,\y-0.75) --  (\x+1.25,\y+0.25) -- (\x+0.75,\y+0.75)-- cycle;
 \draw[line width=0.75mm,blue] (\x,\y) -- (\x+1,\y);
 \filldraw[fill=white, draw=red] (\x+0.25,\y-0.25) circle[radius=4pt] ;
 \filldraw[fill=white, draw=red] (\x+0.75,\y+0.25) circle[radius=4pt] ;}};

 \foreach \x in {-3}{ \foreach \y in {1}
 {\filldraw [fill=blue!20!white, draw=black]  (\x-0.25,\y-0.25) -- (\x+0.25,\y-0.75) --  (\x+1.25,\y+0.25) -- (\x+0.75,\y+0.75)-- cycle;
 \draw[line width=0.75mm,blue] (\x,\y) -- (\x+1,\y);
 \filldraw[fill=white, draw=red] (\x+0.25,\y-0.25) circle[radius=4pt] ;
 \filldraw[fill=white, draw=red] (\x+0.75,\y+0.25) circle[radius=4pt] ;}};

 \foreach \x in {3}{ \foreach \y in {2}
 {\filldraw [fill=blue!20!white, draw=black]  (\x-0.25,\y-0.25) -- (\x+0.25,\y-0.75) --  (\x+1.25,\y+0.25) -- (\x+0.75,\y+0.75)-- cycle;
 \draw[line width=0.75mm,blue] (\x,\y) -- (\x+1,\y);
 \filldraw[fill=white, draw=red] (\x+0.25,\y-0.25) circle[radius=4pt] ;
 \filldraw[fill=white, draw=red] (\x+0.75,\y+0.25) circle[radius=4pt] ;}};

 \foreach \x in {4}{ \foreach \y in {2}
 {\filldraw [fill=blue!20!white, draw=black]  (\x-0.25,\y-0.25) -- (\x+0.25,\y-0.75) --  (\x+1.25,\y+0.25) -- (\x+0.75,\y+0.75)-- cycle;
 \draw[line width=0.75mm,blue] (\x,\y) -- (\x+1,\y);
 \filldraw[fill=white, draw=red] (\x+0.25,\y-0.25) circle[radius=4pt] ;
 \filldraw[fill=white, draw=red] (\x+0.75,\y+0.25) circle[radius=4pt] ;}};

  \foreach \x in {2,3,4,5,6}{ \foreach \y in {4}
 {\filldraw [fill=blue!20!white, draw=black]  (\x-0.75,\y+0.25) -- (\x+0.25,\y-0.75) --  (\x+0.75,\y-0.25) -- (\x-0.25,\y+0.75)-- cycle;
 \filldraw[fill=white, draw=red] (\x+0.25,\y-0.25) circle[radius=4pt] ;
 \filldraw[fill=white, draw=red] (\x-0.25,\y+0.25) circle[radius=4pt] ;}};

  \foreach \x in {3,4,5,6}{ \foreach \y in {3}
 {\filldraw [fill=blue!20!white, draw=black]  (\x-0.75,\y+0.25) -- (\x+0.25,\y-0.75) --  (\x+0.75,\y-0.25) -- (\x-0.25,\y+0.75)-- cycle;
 \filldraw[fill=white, draw=red] (\x+0.25,\y-0.25) circle[radius=4pt] ;
 \filldraw[fill=white, draw=red] (\x-0.25,\y+0.25) circle[radius=4pt] ;}};

   \foreach \x in {-2,0,1,6}{ \foreach \y in {2}
 {\filldraw [fill=blue!20!white, draw=black]  (\x-0.75,\y+0.25) -- (\x+0.25,\y-0.75) --  (\x+0.75,\y-0.25) -- (\x-0.25,\y+0.75)-- cycle;
 \filldraw[fill=white, draw=red] (\x+0.25,\y-0.25) circle[radius=4pt] ;
 \filldraw[fill=white, draw=red] (\x-0.25,\y+0.25) circle[radius=4pt] ;}};

  \foreach \x in {0,1,2,4,6}{ \foreach \y in {1}
 {\filldraw [fill=blue!20!white, draw=black]  (\x-0.75,\y+0.25) -- (\x+0.25,\y-0.75) --  (\x+0.75,\y-0.25) -- (\x-0.25,\y+0.75)-- cycle;
 \filldraw[fill=white, draw=red] (\x+0.25,\y-0.25) circle[radius=4pt] ;
 \filldraw[fill=white, draw=red] (\x-0.25,\y+0.25) circle[radius=4pt] ;}};

  \foreach \x in {-3,-1,1,2,4,6}{ \foreach \y in {0}
 {\filldraw [fill=blue!20!white, draw=black]  (\x-0.75,\y+0.25) -- (\x+0.25,\y-0.75) --  (\x+0.75,\y-0.25) -- (\x-0.25,\y+0.75)-- cycle;
 \filldraw[fill=white, draw=red] (\x+0.25,\y-0.25) circle[radius=4pt] ;
 \filldraw[fill=white, draw=red] (\x-0.25,\y+0.25) circle[radius=4pt] ;}};

 \foreach \x in {1}{ \foreach \y in {3,4}
 {\filldraw [fill=yellow!40!white, draw=black]  (\x-0.25,\y-0.25) -- (\x+0.75,\y-1.25) --  (\x+1.25,\y-0.75) -- (\x+0.25,\y+0.25)-- cycle;
 \draw[line width=0.75mm,blue] (\x,\y) -- (\x+1,\y-1);
 \fill[red] (\x+0.25,\y-0.25) circle[radius=4pt];
\fill[red] (\x+0.75,\y-0.75) circle[radius=4pt];}};

  \foreach \x in {2}{ \foreach \y in {2,3}
 {\filldraw [fill=yellow!40!white, draw=black]  (\x-0.25,\y-0.25) -- (\x+0.75,\y-1.25) --  (\x+1.25,\y-0.75) -- (\x+0.25,\y+0.25)-- cycle;
 \draw[line width=0.75mm,blue] (\x,\y) -- (\x+1,\y-1);
 \fill[red] (\x+0.25,\y-0.25) circle[radius=4pt];
\fill[red] (\x+0.75,\y-0.75) circle[radius=4pt];}};

 \foreach \x in {-1}{ \foreach \y in {1}
 {\filldraw [fill=yellow!40!white, draw=black]  (\x-0.25,\y-0.25) -- (\x+0.75,\y-1.25) --  (\x+1.25,\y-0.75) -- (\x+0.25,\y+0.25)-- cycle;
 \draw[line width=0.75mm,blue] (\x,\y) -- (\x+1,\y-1);
 \fill[red] (\x+0.25,\y-0.25) circle[radius=4pt];
\fill[red] (\x+0.75,\y-0.75) circle[radius=4pt];}};

 \foreach \x in {-2}{ \foreach \y in {3}
 {\filldraw [fill=yellow!40!white, draw=black]  (\x-0.25,\y-0.25) -- (\x+0.75,\y-1.25) --  (\x+1.25,\y-0.75) -- (\x+0.25,\y+0.25)-- cycle;
 \draw[line width=0.75mm,blue] (\x,\y) -- (\x+1,\y-1);
\fill[red] (\x+0.25,\y-0.25) circle[radius=4pt];
\fill[red] (\x+0.75,\y-0.75) circle[radius=4pt];}};

\foreach \x in {5}{ \foreach \y in {2,1}
 {\filldraw [fill=yellow!40!white, draw=black]  (\x-0.75,\y-0.75) -- (\x-0.25,\y-1.25) --  (\x+0.75,\y-0.25) -- (\x+0.25,\y+0.25)-- cycle;
 \draw[line width=0.75mm,blue] (\x,\y) -- (\x,\y-1);
\fill[red] (\x+0.25,\y-0.25) circle[radius=4pt];
\fill[red] (\x-0.25,\y-0.75) circle[radius=4pt];}};

\foreach \x in {3}{ \foreach \y in {1}
 {\filldraw [fill=yellow!40!white, draw=black]  (\x-0.75,\y-0.75) -- (\x-0.25,\y-1.25) --  (\x+0.75,\y-0.25) -- (\x+0.25,\y+0.25)-- cycle;
 \draw[line width=0.75mm,blue] (\x,\y) -- (\x,\y-1);
\fill[red] (\x+0.25,\y-0.25) circle[radius=4pt];
\fill[red] (\x-0.25,\y-0.75) circle[radius=4pt];}};

\foreach \x in {-1,-2,-3,-4}{ \foreach \y in {4,5}
 {\filldraw [fill=yellow!40!white, draw=black]  (\x-0.75,\y-0.75) -- (\x-0.25,\y-1.25) --  (\x+0.75,\y-0.25) -- (\x+0.25,\y+0.25)-- cycle;
 \draw[line width=0.75mm,blue] (\x,\y) -- (\x,\y-1);
\fill[red] (\x+0.25,\y-0.25) circle[radius=4pt];
\fill[red] (\x-0.25,\y-0.75) circle[radius=4pt];}};

\foreach \x in {-3,-4}{ \foreach \y in {2,3}
 {\filldraw [fill=yellow!40!white, draw=black]  (\x-0.75,\y-0.75) -- (\x-0.25,\y-1.25) --  (\x+0.75,\y-0.25) -- (\x+0.25,\y+0.25)-- cycle;
 \draw[line width=0.75mm,blue] (\x,\y) -- (\x,\y-1);
\fill[red] (\x+0.25,\y-0.25) circle[radius=4pt];
\fill[red] (\x-0.25,\y-0.75) circle[radius=4pt];}};

\foreach \x in {-1}{ \foreach \y in {2}
 {\filldraw [fill=yellow!40!white, draw=black]  (\x-0.75,\y-0.75) -- (\x-0.25,\y-1.25) --  (\x+0.75,\y-0.25) -- (\x+0.25,\y+0.25)-- cycle;
 \draw[line width=0.75mm,blue] (\x,\y) -- (\x,\y-1);
\fill[red] (\x+0.25,\y-0.25) circle[radius=4pt];
\fill[red] (\x-0.25,\y-0.75) circle[radius=4pt];}};

\foreach \x in {-2,-4}{ \foreach \y in {1}
 {\filldraw [fill=yellow!40!white, draw=black]  (\x-0.75,\y-0.75) -- (\x-0.25,\y-1.25) --  (\x+0.75,\y-0.25) -- (\x+0.25,\y+0.25)-- cycle;
 \draw[line width=0.75mm,blue] (\x,\y) -- (\x,\y-1);
\fill[red] (\x+0.25,\y-0.25) circle[radius=4pt];
\fill[red] (\x-0.25,\y-0.75) circle[radius=4pt];}};

\foreach \x in {0}{ \foreach \y in {5}
 {\filldraw [fill=yellow!40!white, draw=black]  (\x-0.75,\y-0.75) -- (\x-0.25,\y-1.25) --  (\x+0.75,\y-0.25) -- (\x+0.25,\y+0.25)-- cycle;
 \draw[line width=0.75mm,blue] (\x,\y) -- (\x,\y-1);
\fill[red] (\x+0.25,\y-0.25) circle[radius=4pt];
\fill[red] (\x-0.25,\y-0.75) circle[radius=4pt];}};

\foreach \x in {0} \draw[dashed,red,very thin] (-5.9,4.25-\x/2)--(6.9,4.25-\x/2)node[anchor=west] {$\emptyset$};
\foreach \x in {2,4,6,8} \draw[dashed,red,very thin] (-5.9,4.25-\x/2)--(6.9,4.25-\x/2)node[anchor=west] {\fs{$\lambda^{(\x)}$}};
\foreach \x in {1,3,5,7} \draw[dashed,red,very thin] (-5.9,4.25-\x/2)--(6.9,4.25-\x/2)node[anchor=west] {\fs{$\lambda^{(\x)}$}};
\draw[step=1cm,gray,very thin] (-5.9,-0.9) grid (7.9,5.9);
 \draw[thick,->] (-5.5,0) -- (7.5,0) node[anchor=north west] {\fs{$t$}};
\draw[thick,->] (0,0) -- (0,5.5) node[anchor=south east] {\fs{$x$}};
\foreach \xy in {(-3,1),(-2,3),(-1,3),(-1,1),(0,4), (0,3),(1,4),(1,3),(2,3),(2,2),(3,2),(4,2)} {\node at \xy {$\color{blue} \bullet$};};
\end{tikzpicture}
\caption{\fs{Strict plane partition point configuration (left) and plane overpartition point configuration (right).}}
\label{fig:overpartitions-pointconfig}
\end{figure}

One can show that a point configuration representing a plane overpartition
$
\emptyset \prec \lambda^{(1)} \prec'\lambda^{(2)}\prec\cdots \prec \lambda^{(2n-1)} \prec' \lambda^{(2N)}=\lambda
$
is a subset of  $\left(\Z+\frac{1}{4}\right)\times \left(\Z +\frac{3}{4}\right) \cup \left(\Z+\frac{3}{4}\right)\times \left(\Z +\frac{1}{4}\right)$ consisting of points $(t,x)$ such that the Maya diagram of  $\lambda^{i(x)}$ contains a particle at the position $\overline{t}$: $\overline{t}=\lambda^{i(x)}_j-j+1/2$, for some $j\geq 1$ with $i(x)=2N-2(x-1/4)$ and
\begin{equation}
\label{toverline}
\overline{t}=\lceil{t-1/2}\rceil-1/2.
\end{equation}

For the HV-ascending Schur process with $x_1=x_2=q^N, \dots , x_{2N-1}=x_{2N}=q^1$:
\begin{equation}
F_{HV}(i,z)=\frac{(-q^{N+1-\lfloor{i/2}\rfloor}z;q)_{\lfloor{i/2}\rfloor}}{(q^{N+1-\lceil{i/2}\rceil}z;q)_{\lceil{i/2}\rceil}}\frac{(q/z;q)_{N}}{(-q/z;q)_N}, \quad i=1,\dots,2N.
\end{equation}
If we set $F(x,z)=F_{HV}(i(x),z) $ then
$
    F(x,z)=\frac{(-q^{\lceil x+1/2\rceil}z;q)_{N+1-{\lceil x+1/2\rceil}}}{(q^{\lceil x\rceil}z;q)_{N+1-{\lceil x\rceil}}}\frac{(q/z;q)_{N}}{(-q/z;q)_N}
$ and when $N \to \infty$
\begin{equation}
 F(x,z)=\frac{(-q^{\lceil x+1/2\rceil}z;q)_{\infty}}{(q^{\lceil x\rceil}z;q)_{\infty}}\frac{(q/z;q)_{\infty}}{(-q/z;q)_\infty}.
\end{equation}

\begin{rem}
\label{zerosandpoles}
For $q\in (0,1)$ and fixed $x$, zeros of $F$ as a function of $z$ are $-q^{-\lceil x+1/2\rceil},-q^{-\lceil x+1/2\rceil-1},\dots$, which belong to $(-\infty,-q^{-\lceil{x+1/2}\rceil}]$, and $q,q^2,\dots$, which belong to $(0,q]$. Poles are $q^{-\lceil x\rceil}, q^{-\lceil x\rceil-1},\dots$, which belong to $[q^{-\lceil{x}\rceil},\infty)$,  and $-q,-q^2,\dots$, which belong to $[-q,0)$.

In our asymptotic analysis we will consider $q\in (0,1)$ and $x_i$ such that $rx_i\to \chi$ when $r \to 0+$ where $q=e^{-{r}}$. We will need to avoid zeros or poles of $F(x_i,z)$ along certain contours in $z$ for $r$ close to $0+$. Zeros, respectively poles,  could be avoided if we make sure to choose contours that do not cut $\mathcal{Z}(\chi)=(-\infty,-e^{\chi}] \cup (0,1]$, respectively $\mathcal{P}(\chi)=[-1,0)\cup [e^{\chi},\infty).$
\end{rem}

\begin{thm}
\label{KHVkernel}
  Let $(t_1,x_1), \dots, (t_n,x_n) \in \left(\Z+\frac{1}{4}\right)\times \left(\Z +\frac{3}{4}\right) \cup \left(\Z+\frac{3}{4}\right)\times \left(\Z +\frac{1}{4}\right)$. Pfaffian correlations are given by the following matrix kernel:
  \begin{equation}
    \begin{split}
   K(t_i, x_i; t_j, x_j) =\left[\begin{array}{cc}K_{1,1}(t_i,x_i;t_j,x_j)&K_{1,2}(t_i,x_i;t_j,x_j)\\
-K_{1,2}(t_j,x_j;t_i,x_i)&K_{2,2}(t_i,x_i;t_j,x_j)\end{array}\right]
    \end{split}
  \end{equation}
  \noindent where
  \begin{equation}
    \begin{split}
  K_{1,1}(t_i,x_i;t_j,x_j)&=[ z^{\overline{t_i}} w^{\overline{t_j}} ] F(x_i, z) F(x_j, w) \frac{\sqrt{zw}(z-w)}{(z+1)(w+1)(zw-1)}, \quad 1<|w|, 1<|z|, \\
  K_{1,2}(t_i,x_i;t_j,x_j)&= \left[ \frac{z^{\overline{t_i}}}{w^{\overline{t_j}}} \right] \frac{F(x_i, z)}{F(x_j, w)} \frac{\sqrt{zw}(zw-1)}{(z-w)(z+1)(w-1)}, \quad 1<|z|, 1<|w|, |w|{\substack{< \\ >}} |z|, \textrm{ for } x_i{\substack{\ge \\ <}}x_j,\\
  K_{2,2}(t_i,x_i;t_j,x_j)&= \left[ \frac{1}{z^{\overline{t_i}} w^{\overline{t_j}}} \right] \frac{1}{F(x_i, z) F(x_j, w)} \frac{\sqrt{zw}(z-w)}{(z-1)(w-1)(zw-1)}, \quad 1<|z|, 1<|w|,
    \end{split}
  \end{equation}

\noindent and $\overline{t}$ is as in \eqref{toverline}.
\end{thm}

\begin{rem} \label{remarkpfaffianmultiplication} Let $\chi \in \R$. Define
\begin{equation}
\label{Ktilde}
\tilde{K}(t_i,x_i;t_j,x_j)=\left(\begin{matrix}e^{\chi \overline{t_i}/2} & 0\\0&e^{-\chi \overline{t_i}/2}\end{matrix}\right){K}(t_i,x_i;t_j,x_j)\left(\begin{matrix}e^{\chi \overline{t_j}/2} & 0\\0&e^{-\chi \overline{t_j}/2}\end{matrix}\right).
\end{equation}
Then
\begin{equation}
\pf [\tilde{K}(t_i,x_i;t_j,x_j)]_{1 \leq i,j \leq n} = \pf [K(t_i,x_i;t_j,x_j)]_{1 \leq i,j \leq n}.
\end{equation}
Let
\begin{equation}
\label{Gfunction}
G(x,t,z)=r\left(\log F(x,z)-\overline{t}(\log z -\chi/2)\right).
\end{equation}
Using Theorem~\ref{KHVkernel}
\begin{equation}
\label{K11tilde}
\tilde{K}_{1,1}(t_i,x_i;t_j,x_j)=\frac{1}{(2\pi \im)^2}\int_{C_z}\int_{C_w} \frac{z-w}{\sqrt{zw}(z+1)(w+1)(zw-1)}\exp\left[\frac{1}{r}\left(G(\overline{t_i},x_i,z)+G(\overline{t_j},x_j,w)\right)\right]\dx z\dx w
\end{equation}
where $C_z$ and $C_w$ are simple closed counterclockwise oriented contours such that $|z|>1$, $z\notin \mathcal{P}(\chi)$, $|w|>1$, and $w\notin \mathcal{P}(\chi)$ for all $z\in C_z$ and $w\in C_w$;
\begin{equation}
\label{K12tilde}
\tilde{K}_{1,2}(t_i,x_i;t_j,x_j)=\frac{1}{(2\pi \im)^2}\int_{C_z}\int_{C_w} \frac{zw-1}{\sqrt{zw}(z-w)(z+1)(w-1)}\exp\left[\frac{1}{r}\left(G(\overline{t_i},x_i,z)-G(\overline{t_j},x_j,w)\right)\right]\dx z\dx w
\end{equation}
where $C_z$ and $C_w$ are simple closed counterclockwise oriented contours such that  $|z|>1$, $z\notin \mathcal{P}(\chi)$, $|w|> 1$, $w\notin \mathcal{Z}(\chi)$, $|z|>|w|$ ($|z|<|w|$), for all  $z\in C_z$ and $w\in C_w$ when  $x_i \geq x_j$ ($x_i < x_j$);
\begin{equation}
\label{K22tilde}
\tilde{K}_{2,2}(t_i,x_i;t_j,x_j)=\frac{1}{(2\pi \im)^2}\int_{C_z}\int_{C_w} \frac{z-w}{\sqrt{zw}(z-1)(w-1)(zw-1)}\exp\left[\frac{1}{r}\left(-G(\overline{t_i},x_i,z)-G(\overline{t_j},x_j,w)\right)\right]\dx z\dx w
\end{equation}
where $C_z$ and $C_w$ are two simple closed counterclockwise oriented contours such that $|z|>1$, $z\notin \mathcal{Z}(\chi)$, $|w|>1$, $w\notin \mathcal{Z}(\chi)$ for all $z\in C_z$ and $w\in C_w$.
\end{rem}

Before we proceed with the asymptotics, we set up some notation. 

Define $C^{+}(R,\theta)$ ($C^-(R,\theta)$) to be the counterclockwise (clockwise) oriented arc on $|z|=R$ from $Re^{-i\theta}$ to $Re^{i\theta}$ for $R>0$ and $0\leq\theta\leq \pi$. The counterclockwise oriented circle $|z|=R$ is then $C^{+}(R,\theta) \cup -C^-(R,\theta)$. 

For a sequence of real numbers $a_i$ we define an (infinite) integer matrix
\begin{equation}
{\Delta}(a)_{ij}=\lceil a_i \rceil-\lceil a_j \rceil.
\end{equation}

Let $
\mathcal{D}=\left\{(\tau,\chi)|\tau \in \R, \chi \in \R_{\geq 0}, -1\leq f(\tau,\chi)\leq1\right\}=\left\{(\tau,\chi)|\chi \in \R_{\geq 0}, -\tau_c(\chi) \leq \tau \leq \tau_c(\chi)\right\}$
where
\begin{equation}\label{taucritical}
f(\tau,\chi)=\frac{(e^{\chi}+1)(e^\tau-1)}{2e^{\chi/2}(e^{\tau}+1)},  \quad
\tau_c(\chi)=2\log\frac{e^{\chi/2}+1}{e^{\chi/2}-1}.
\end{equation}
Note that $\mathcal{D}$ is a domain bounded by three curves: $\chi=0$ and $-1 \pm e^{\chi/2}-e^{\tau/2}\mp e^{\tau/2}e^{\chi/2}=0$ for $\chi>0.$ Amoeba of $P(z,w)=-1+z+w+zw$ is a domain bounded by the following four curves:
\begin{equation}
-1 \pm e^{\omega}-e^{\xi}\mp e^{\xi}e^{\omega}=0, \quad \omega>0, \text{\ and\ } -1 \pm e^{\omega}+e^{\xi}\pm e^{\xi}e^{\omega}=0, \quad \omega<0.
\end{equation}
If we set $(\xi,\omega)=(\tau/2,\chi/2)$ then $\mathcal{D}$ is the half of the amoeba of $-1+z+w+zw$ for $\omega \geq 0$. Let
\begin{equation}
\label{thetacritical}
\theta_c(\tau,\chi)=\begin{cases}
\arccos(f(\tau,\chi)),&-\tau_c(\chi) \leq \tau \leq \tau_c(\chi),\\
0&\tau > \tau_c(\chi),\\
\pi&\tau<-\tau_c(\chi).
\end{cases}
\end{equation}

\begin{thm}
\label{limitofK12}
 Let $rt_i, \, rt_j \to \tau$, $r x_i, \, rx_j \to \chi$ when $r \to 0+$ where  $q=e^{-r}$. Assume $\Delta(t-1/2)_{ij}$, $\Delta(x)_{ij}$, $\Delta(x+{1/2})_{ij}$ do not change with $r$. Then 
\begin{equation}
\lim_{r \to 0+} \tilde{K}_{1,2}(t_i,x_i;t_j,x_j)=\frac{e^{-\frac{\chi}{2} \Delta (t-1/2)_{ij}}}{2\pi \im}  \int_{C^{\pm}(e^{-\chi/2},\theta_c(\tau,\chi))}\frac{1}{w^{\Delta(t-1/2)_{ij}+1 }}\frac{(1-w)^{\Delta(x)_{ij}}}{(1+w)^{\Delta(x+1/2)_{ij}}}\dx w
\end{equation}
where we choose $C^{+}(e^{-\chi/2},\theta_c(\tau,\chi))$ if $x_i\geq x_j$ and $C^{-}(e^{-\chi/2},\theta_c(\tau,\chi))$ otherwise and where $\theta_c(\tau,\chi)$ is as in \eqref{thetacritical}.
\end{thm}

\begin{proof}
We start with \eqref{K12tilde}. The asymptotics is determined by the limit of $G$ when $r\to 0+$:
\begin{equation}
S(z;\tau,\chi)= -\dilog(-e^{-\chi}z)-\dilog(\frac{1}{z})+\dilog(e^{-\chi}z)+\dilog(-\frac{1}{z})-\tau(\log z-\chi/2)
\end{equation}
where $\dilog(z)$ is defined in Appendix~\ref{sec:theta}. Ignoring the $\log$ part of $S$, the function is analytic on $\C \backslash{\mathcal{P}(\chi)}$.
We can easily compute
\begin{equation}
\label{SDer}
z\frac{\dx}{\dx z}S(z;\tau,\chi)=-\tau+\log \frac{(1+e^{-\chi}z)(z+1)}{(1-e^{-\chi}z)(z-1)}
\end{equation}
where by the Cauchy--Riemann equations
\begin{equation}
z\frac{\dx}{\dx z}S(z;\tau,\chi)=x\frac{\dx}{\dx x}\Re S +y\frac{\dx}{\dx y}\Re S+\im \left(y\frac{\dx}{\dx x}\Re S -x\frac{\dx}{\dx y}\Re S\right).
\end{equation}

The real part of $S(z;\tau,\chi)$ vanishes on the circle $z=e^{\chi/2}e^{i\theta}$. This implies that on this circle the imaginary part of $z\frac{\dx}{\dx z}S(z;\tau,\chi)$ vanishes too, which is the derivative of $\Re S$ in the direction of the tangent, while the real part is equal to  $R\frac{\dx}{\dx R}\Re S(z;\tau,\chi)$. Then for $z=e^{\chi/2}e^{i\theta}$
\begin{equation}
R\frac{\dx}{\dx R}\Re S(z;\tau,\chi)=-\tau+\log\left| \frac{z+1}{z-1}\right|^2,
\end{equation}
which is negative if and only if $\cos \theta<f(\tau,\chi)$.

If $(\tau,\chi) \in \mathcal{D}$ and $z=e^{\chi/2}e^{i\theta}$ then $R\frac{\dx}{\dx R}\Re S(z;\tau,\chi)$ changes the sign along $|z|=e^{\chi/2}$ at  $\theta=\pm \theta_c(\tau,\chi)$ being positive for $|\theta|<|\theta_c(\tau,\chi)|$.  We then deform contours so that the real parts of $S(z,\tau,\chi)$ and $S(w,\tau,\chi)$ are negative everywhere on the new contours except at the critical points. The new contours look like $\gamma_z$ and $\gamma_w$ in Figure~\ref{newcontours}. The integrals over these new contours vanish as $r\to 0+$, but we pick up the residue at $z=w$ on $C^{+}(e^{\chi/2},\theta_c(\tau,\chi))$ ($C^{-}(e^{\chi/2},\theta_c(\tau,\chi)$) when $x_i\geq x_j$ ($x_i<x_j$).  The residue is equal to
\begin{equation}
\frac{1}{2\pi \im} e^{\frac{\chi}{2} \Delta (t-1/2)_{ij}} \int_{C^{\pm}(e^{\chi/2},\theta_c(\tau,\chi))}\frac{1}{w^{\Delta(t-1/2)_{ij}+1 }}\frac{(1-e^{-\chi}w)^{\Delta(x)_{ij}}}{(1+e^{-\chi}w)^{\Delta (x+1/2)_{ij}}}\dx w.
\end{equation}
The change of variables $w \mapsto e^{-\chi}w$ brings the expression from the statement of the theorem.
\begin{SCfigure}[][ht]
 \begin{tikzpicture}[scale=1.5]
  \draw[smooth] (-3,0)--(3,0);
  \draw[smooth] (0,-1.5)--(0,1.5);

\draw (0,0) circle (1.3);

\draw [thick, dotted, smooth] (1.1,0) .. controls ({1.3*cos(30)},{1.3*sin(30)}) and ({1.3*cos(50)},{1.3*sin(50)}).. ({1.3*cos(45)},{1.3*sin(45)});
\draw[thick, dotted, smooth] ({1.3*cos(45)},{1.3*sin(45)}) .. controls  ({1.9*cos(90)},{1.9*sin(90)}) and ({2*cos(140)},{2*sin(140)}) .. (-1.5,0);
\draw [thick, dotted, smooth] (1.1,0) .. controls ({1.3*cos(-30)},{1.3*sin(-30)}) and ({1.3*cos(-50)},{1.3*sin(-50)}).. ({1.3*cos(-45)},{1.3*sin(-45)});
\draw[thick, dotted, smooth] ({1.3*cos(-45)},{1.3*sin(-45)}) .. controls  ({1.9*cos(-90)},{1.9*sin(-90)}) and ({2*cos(-140)},{2*sin(-140)}) .. (-1.5,0);
  \node[label = above left: $\gamma_{z}$] at (-1,.9) {};

\draw [thick, dashed, smooth] (1.4,0) .. controls ({1.5*cos(20)},{1.5*sin(20)}) and ({1.4*cos(30)},{1.4*sin(30)}).. ({1.3*cos(45)},{1.3*sin(45)});
\draw[thick, dashed, smooth] ({1.3*cos(45)},{1.3*sin(45)}) .. controls  ({1.4*cos(90)},{1.4*sin(90)}) and ({1.6*cos(130)},{1.6*sin(130)}) .. (-1.1,0);
\draw [thick, dashed, smooth] (1.4,0) .. controls ({1.5*cos(-20)},{1.5*sin(-20)}) and ({1.4*cos(-30)},{1.4*sin(-30)}).. ({1.3*cos(-45)},{1.3*sin(-45)});
\draw[thick, dashed, smooth] ({1.3*cos(-45)},{1.3*sin(-45)}) .. controls  ({1.4*cos(-90)},{1.4*sin(-90)}) and ({1.6*cos(-130)},{1.6*sin(-130)}) .. (-1.1,0);
 \node[label = below right: $\gamma_{w}$] at (-.9,.9) {};

\draw[line width=1.25mm, smooth] (1.8,0)--(3,0);
  \draw[line width=1.25mm, smooth] (-1.8,0)--(-3,0);
  \draw[line width=1.25mm, smooth] (0,0)--(0.9,0);
  \draw[line width=1.25mm, smooth] (-0,0)--(-0.9,0);

  \node[label = below: \fs{$-1$}] at (-0.9,0) {};
  \node[label = below: \fs{$1$}] at (0.9,0) {};
  \node[label = below: \fs{$e^{\chi}$}] at (1.8,0) {};
    \node[label = below: \fs{$-e^{\chi}$}] at (-1.8,0) {};
  \node[label = above: \fs{$\mathcal{Z}(\chi)$}] at (-2.5,0) {};
  \node[label = above: \fs{$\mathcal{P}(\chi)$}] at (2.5,0) {};
  \node[label = above: \fs{$\mathcal{Z}(\chi)$}] at (.5,0) {};
  \node[label = above: \fs{$\mathcal{P}(\chi)$}] at (-.5,0) {};

  \draw[smooth] (2,1)--(2.5,1);
 \node[label=right:\fs{$|z|=e^{\chi/2}$}] at (2.5,1){};
 \end{tikzpicture}
 \vspace{-20pt}
 \caption{\fs{Deformed contours in $z$ and $w$.}} 
 \label{newcontours}
\end{SCfigure}
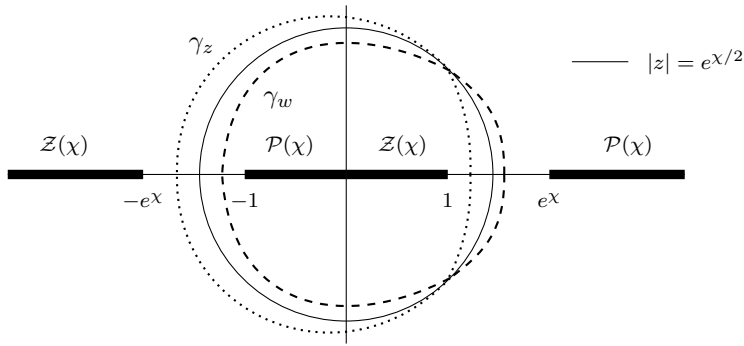

Some special care is needed when $\chi=0$ because we are deforming the contours in a neighborhood of  $|z|=1$ and $|w|=1$. In this case $(\tau,\chi)\in \mathcal{D}$ for all $\tau$. Observe that if we deform the contours as described above, the new contour in $z$, respectively $w$,  will still encompass $z=-1$, respectively $w=1$, and we will not be picking up any residues coming from  $z=-1$ or $w=1$. Also note that while deforming the contours we remain away from $\mathcal{P}(\chi)$ for the $z$-contour and $\mathcal{Z}(\chi)$ for the $w$-contour.

Lastly, the integral $1/(z-w)$ converges at the points of intersection of the contours. We can for simplicity make them cross orthogonally; then the integral of $1/(z-w)$ along these contours is absolutely bounded by the convergent integral $\iint 1/\sqrt{x^2+y^2}dxdy$.

If $(\tau,\chi)$ does not belong to $\mathcal{D}$ then $R\frac{\dx}{\dx R}\Re S(z;\tau,\chi)$ does not change sign. It is strictly positive for $\tau<0$ and strictly negative for $\tau>0$. We just need to push the contour $C_z$ inside (for $\tau<0$) or outside (for $\tau>0$) of $|z|=e^{\chi/2}$ and push $C_w$ outside (for $\tau<0$) or inside (for $\tau>0$) of $|w|=e^{\chi/2}$ to obtain that both  $\Re S(z;\tau,\chi)$ and $\Re S(w;\tau,\chi)$ are negative along the new contours. Then when $\tau>0$ we get that $K_{1,2}$ vanishes in the limit and when $\tau<0$ we pick up the residue at $z=w$ along  the whole circle $|z|=e^{\chi/2}$.

\end{proof}

\begin{rem}
The density function of a particle present at $(\tau,\chi)$ is
\begin{equation}
\rho(\tau,\chi)=\lim_{r \to 0+} K_{1,2}(t,x;t,x), \quad \textrm{for } rt \to \tau \textrm{ and } rx \to \chi.
\end{equation}
From the theorem above
\begin{equation}
\rho(\tau,\chi)=\frac{1}{2\pi \im} \int_{C^+(e^{-\chi/2},\theta_c(\tau,\chi))}\frac{1}{z}\dx z=\frac{\theta_c(\tau,\chi)}{\pi}.
\end{equation}
\end{rem}
A three-dimensional depiction of a plane overpartition $\pi$, obtained by stacking $\pi_{i,j}$ unit cubes above the $(i,j)$-th position, is deciphered from the corresponding point configuration in the following way. Each particle $(t,x)$ in the configuration is mapped to $(X,Y,Z) \in \Z_{\geq 0}^3$ where $X$ is the number of particles to the right of $(t,x)$ counting the particle itself, $Y$ is the number of holes to the left of $(t,x)$, or equivalently $X+\lfloor{t-1/2}\rfloor$, and  $Z$ is equal to $\lceil x+1/2 \rceil$.
The limit shape of strict plane partitions is:
\begin{equation}
X(\tau,\chi)=\int_{\tau}^{\tau_c(\chi)}\rho(t,\chi)\dx t, \quad
Y(\tau,\chi)=\int_{-\tau_c(\chi)}^\tau(1-\rho(t,\chi))\dx t=\int_{\tau}^{\tau_{c}(\chi)} \rho(t,\chi)\dx t+\tau, \quad
Z(\tau,\chi)=\chi
\end{equation}
where $\tau_c(\chi)$ is given by \eqref{taucritical}.
The two formulas for $Y$ are indeed equal from the properties of the model, but this can be confirmed directly since
\begin{equation}
\int_{-\tau_c(\chi)}^0(1-\rho(t,\chi))\dx t=\int_{-\tau_c(\chi)}^0\frac{\pi-\arccos{f(t,\chi)}}{\pi}\dx t=\int_{-\tau_{c}(\chi)}^0\frac{\arccos{f(-t,\chi)}}{\pi}\dx t=\int_{0}^{\tau_{c}(\chi)} \rho(t,\chi)\dx t.
\end{equation}

\begin{figure}[t]
  \begin{center}
  \includegraphics[scale=0.135]{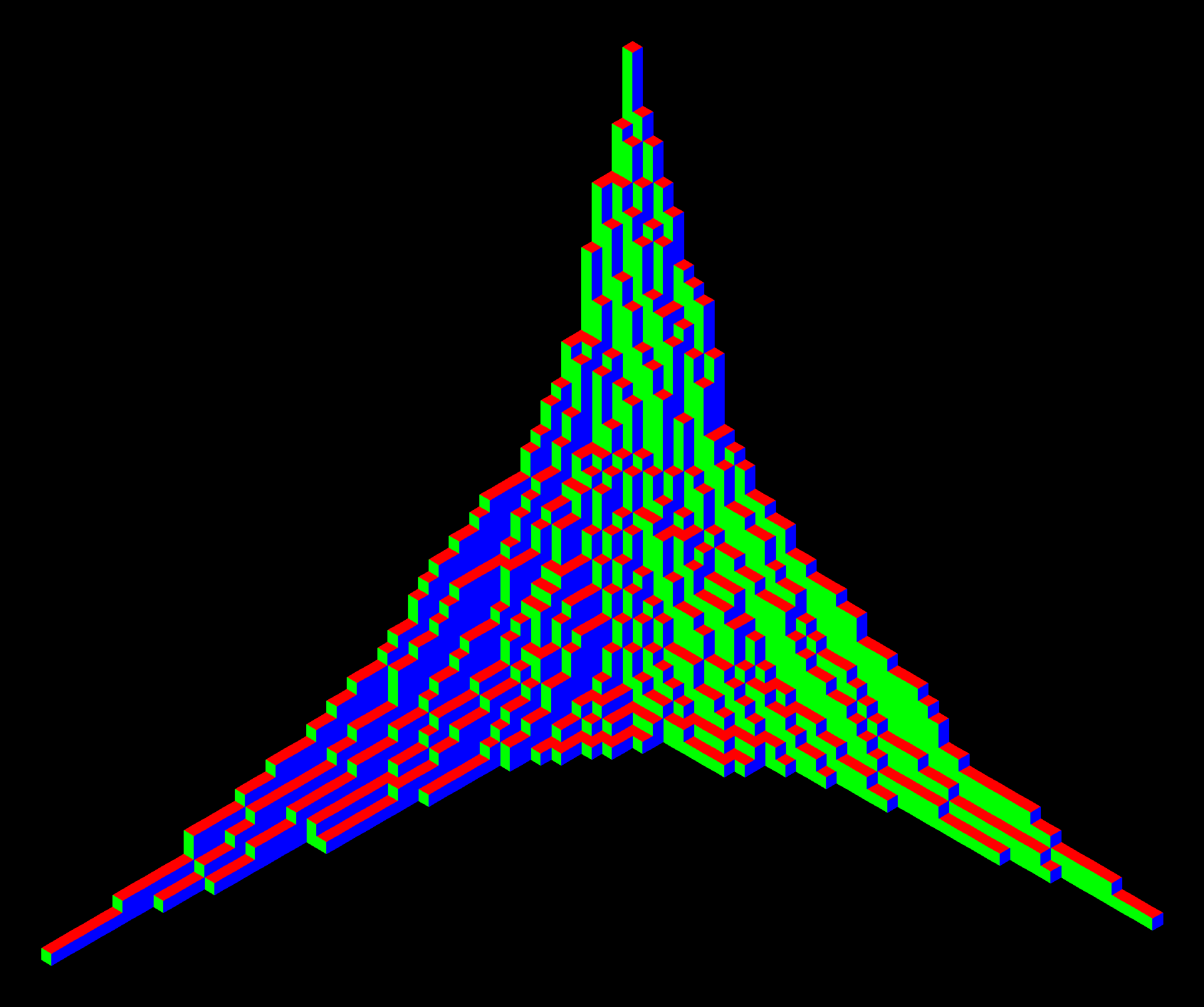}
  \end{center}
  \caption{\fs{A large random plane overpartition shown as a strict plane partition.}}
  \label{fig:large_pop}
\end{figure}

An exactly sampled large plane overpartition for $q=0.9$ is shown in Figure~\ref{fig:large_pop} as a strict plane partition and as a domino tiling in Figure~\ref{fig:po_large}.  The simulations were carried out using the algorithms of~\cite{bbbccv}.

Under the same conditions as in Theorem~\ref{limitofK12}, diagonal elements of $\tilde{K}$ vanish for $\chi>0$.
\begin{thm}
\label{limitofK11andK22}
 Let $rt_i, \, rt_j \to \tau$, $r x_i,\,  rx_j \to \chi$ when $r \to 0+$ where  $q=e^{-r}$. Assume $\Delta(t-1/2)_{ij}$, $\Delta(x)_{ij}$, $\Delta(x+{1/2})_{ij}$ do not change with $r$. Then for $\chi>0$
\begin{equation}
\lim_{r \to 0+} \tilde{K}_{1,1}(t_i,x_i;t_j,x_j)=\lim_{r \to 0+} \tilde{K}_{2,2}(t_i,x_i;t_j,x_j)=0.
\end{equation}
When $\chi=0$ assume in addition that $x_i$ does not change with $r$. Set $e_i=\lceil x_i \rceil -\lceil x_i+1/2 \rceil$. Then
\begin{align}
  \label{residueK11}
  \lim_{r\to0+}\tilde{K}_{1,1}(t_i,x_i;t_j,x_j)=\frac{1}{2\pi \im} \int_{C^+(1,\theta_c(\tau,0))}\frac{(-1)^{\lceil x_i \rceil-1}}{w^{-\Delta(t-1/2)_{ij}+e_i+1}}\frac{(1-w)^{\lceil x_i\rceil +\lceil x_j\rceil-1}}{(1+w)^{\lceil x_i+1/2\rceil +\lceil x_j+1/2\rceil-1}}\dx w, \\
  \label{residueK22}
  \lim_{r\to0+}\tilde{K}_{2,2}(t_i,x_i;t_j,x_j)=\frac{1}{2\pi \im} \int_{C^-(1,\theta_c(\tau,0))}\frac{(-1)^{\lceil x_i\rceil-1}}{w^{\Delta(t-1/2)_{ij}-e_i+1}}\frac{(1-w)^{1-\lceil x_i\rceil -\lceil x_j\rceil}}{(1+w)^{1-\lceil x_i+1/2\rceil -\lceil x_j+1/2\rceil}}\dx w.
  \end{align}
\end{thm}

\begin{proof}
We start with \eqref{K11tilde}. The analysis was done in the proof of Theorem~\ref{limitofK12}. We can conclude that we can deform $C_z$ and $C_w$ into new contours such that  the real parts of $S(z,\tau,\chi)$ and $S(w,\tau,\chi)$ are negative everywhere on the new contours except at the critical points. We can choose to deform them to the same contour $\gamma_z$ in Figure~\ref{newcontours}. When $\chi>0$ there is no residue to pick up since while deforming the contours we keep $|z|>1$, $z\notin \mathcal{P}(\chi)$, $|w|>1$ and $z\notin \mathcal{P}(\chi)$. The proof is analogous for $\tilde{K}_{2,2}$.

In the case where $\chi=0$ we can proceed as above, but now for $\tilde{K}_{1,1}$ we will be picking up the residue at $z=1/w$ along $C^+(1,\theta_c(\tau,0))$ for $\tilde{K}_{1,1}$ and along  $-C^-(1,\theta_c(\tau,0))$ for $\tilde{K}_{2,2}$.  Computing the residues we get  \eqref{residueK11} and \eqref{residueK22}.
\end{proof}

Let $(\tau,\chi)\in \R \times \R_{\geq 0}, t, x, y \in \Z$. Define
\begin{equation}
  \begin{split}
B_{\tau,\chi}^\pm(t,x,y)=
\displaystyle \frac{1}{2\pi \im}\int_{C^{\pm}(e^{-\chi/2},\theta_c(\tau,\chi))}\frac{1}{z^{t+1}}\frac{(1-z)^{x}}{(1+z)^y}\dx z
  \end{split}
\end{equation}
where $\theta_c(\tau,\chi)$ is as in \eqref{thetacritical}.

From Theorems~\ref{limitofK12} and \ref{limitofK11andK22} we obtain that in the limit the bulk has determinantal correlations. Precisely, after factoring out constants that cancel out in the pfaffian as in Remark~\ref{remarkpfaffianmultiplication}, we get the following.

\begin{thm}\label{PODetPP}
For $i=1,\dots,n$ let $rt_i \to \tau$, $r x_i \to \chi>0$ when $r \to 0+$ where  $q=e^{-r}$. Assume $\Delta(t-1/2)$, $\Delta(x)$, $\Delta(x+{1/2})$ do not change with $r$. Then
\begin{equation}
 \lim_{r\to0+}\varrho((t_1,x_1),\dots,(t_n,x_n))=\det \left[B^{\tau,\chi}_\pm\left(\Delta(t-1/2)_{ij},\Delta (x)_{ij},\Delta( x+1/2)_{ij}\right)\right]_{i,j=1,\dots,n}
\end{equation}
where we choose $+$ if $x_i\geq x_j$ and $-$ otherwise.
\end{thm}

At the boundary $\chi=0$ limit correlations remain pfaffian.
\begin{thm} \label{POPfaffPP}
For $i=1,\dots,n$ and $q=e^{-r}$ let $t_i$ and $x_i$ be such that $rt_i \to \tau$ when $r \to 0+$, and $\Delta(t-1/2)$, and $x_i$  do not change with $r$. Then
 \begin{equation}
 \lim_{r\to0+}\varrho((t_1,x_1),\dots,(t_n,x_n))=\pf \left[
 \begin{matrix}B_{1,1}(i,j)&B_{1,2}(i,j)\\
 -B_{1,2}(j,i)&B_{2,2}(i,j)
 \end{matrix}\right]_{1 \leq i,j \leq n}
\end{equation}
 where
 \begin{equation}
  \begin{split}
 B_{1,1}(i,j)&=(-1)^{\lceil x_i \rceil-1}B_{\tau,0}^+\left(-\Delta(t-1/2)_{ij}+e_i,\lceil x_i \rceil+\lceil{x_j}\rceil -1,\lceil x_i+1/2 \rceil+\lceil{x_j+1/2}\rceil -1\right), \\
B_{1,2}(i,j)&=B_{\tau,0}^\pm\left(\Delta(t-1/2)_{ij},\Delta(x)_{ij},\Delta(x+1/2)_{ij}\right), \\
B_{2,2}(i,j)&=(-1)^{\lceil x_i \rceil-1}B_{\tau,0}^-\left(\Delta(t-1/2)_{ij}-e_i,1-\lceil x_i \rceil-\lceil{x_j}\rceil ,1-\lceil x_i+1/2 \rceil-\lceil{x_j+1/2}\rceil \right)
  \end{split}
  \end{equation}
 where $e_i=\lceil x_i \rceil -\lceil x_i+1/2 \rceil$ and we choose $+$ if $x_i\geq x_j$ and $-$ otherwise.
\end{thm}

\begin{rem}
If in Theorem~\ref{POPfaffPP} instead of keeping $x_i$ constant we let $x_i \to \infty$, while still keeping $rx_i \to 0$, we get that the correlations in the limit become determinantal with the kernel $B_{1,2}(i,j)$. To show this one can adapt the proof of Theorem \ref{limitofK11andK22} where, as before, we deform the contours and end up with the residue at $z=1/w$ along $C^\pm(1,\theta_c(\tau,0))$. We can then use the fact that  $(z;q)_x=(1-z)^xe^{o(x)}$ when $x \to \infty$ and $rx\to 0$ and that $|(1-z)/(1+z)|$ achieves its maximum (minimum) on $C^+$ ($C^-$) at the endpoints to show that $R(i,z_c)R(j,z_c)\tilde{K}_{1,1}$ and $R^{-1}(i,z_c)R^{-1}(j,z_c)\tilde{K}_{2,2}$ both vanish, where the conjugation factor $R(i,z)$ is $(z;q)_{x_i}/(-z;q)_{x_i}$ and $z_c$ is one of the endpoints $z_c=e^{i\theta_c(\tau,0)}$. See also Remark~\ref{rem:fincrucial}.
\end{rem}

\begin{rem} Theorem~\ref{MainTheoremPO} follows from Theorems~\ref{PODetPP} and~\ref{POPfaffPP} if we only consider the point process on points $(t_1,x_1), \dots, (t_n,x_n) \in  \left(\Z+\frac{3}{4}\right)\times \left(\Z +\frac{1}{4}\right)$, in which case $\lceil x_i \rceil=\lceil x_i+1/2\rceil$. To obtain Theorem~\ref{MainTheoremPO} we need to rescale our microscopic coordinates $(t,x)$ as
\begin{equation}
 t = \left\lfloor \frac{\tau}{r} \right\rfloor + \k, \qquad x = \left\lfloor \frac{\chi}{r}\right\rfloor + \i+\frac{1}{4}
\end{equation}
and note that there $(\x,\y)$ and $\mathcal{L}$ were used instead of $(\tau, \chi)$ and $\mathcal{D}$.
\end{rem}

\section{Conclusion}
\label{sec:conc}

In this paper we have introduced the Schur process with two free boundary partitions; computed its pfaffian correlations upon charge mixing; and provided a uniform asymptotic treatment, upon killing one of the free boundaries, of various combinatorial and statistical mechanical models: symmetric last passage percolation, symmetric plane partitions and plane overpartitions.

We mention a few directions that merit further investigation: asymptotic analysis of the full two free boundary correlations, where we expect new kernels and ``universality'' behavior to appear --- we plan to address this in a future work; possible generalizations to Hall--Littlewood, $q$-Whittaker, or even Macdonald processes --- Fock space does not quite exist here so new ideas in the spirit of~\cite{bbcw17} are needed; analysis of other combinatorial models that \emph{can} be treated using one free boundary --- see~\cite[Section 7]{rai} for a flavor; and possible connections to matrix models, $\tau$-functions, enumerative algebraic geometry and map counting.

\paragraph{Acknowledgments.}

We thank Jinho Baik, Guillaume Barraquand, Philippe Biane, Elia Bisi,
Alexei Borodin, C\'edric Boutillier, Ivan Corwin, Philippe Di Francesco,
Patrik Ferrari, Mikael de la Salle, Ole Warnaar,
Michael Wheeler, and Nikos Zygouras for useful conversations.
We also thank the anonymous referee for suggesting valuable improvements
to the paper.

Most of this work was done while the first three authors were
at the D\'epartement de math\'ematiques et applications, \'Ecole
normale sup\'erieure, Paris. We also acknowledge hospitality and
support from the Galileo Galilei Institute during the 2015 program on
``Statistical Mechanics, Integrability and Combinatorics''. Part of
this work was done while D.B.\ was visiting the ENS de Lyon and
L\'aszl\'o Erd\H{o}s at IST Austria, and remerciements are due to both
institutions and to L\'aszl\'o for their hosting.
J.B.\ recently enjoyed the hospitality from the PCMI 2017 Research
Program and from the ESI programme on ``Algorithmic and Enumerative
Combinatorics''.

We acknowledge financial support from the ``Combinatoire \`a Paris''
project funded by the City of Paris (D.B.\ and J.B.), from the Agence
Nationale de la Recherche via the grants ANR 12-JS02-001-01
``Cartaplus'' and ANR-14-CE25-0014 ``GRAAL'' (J.B.), from Paris
Sciences \& Lettres and from the ERC Advanced Grant No.\ 338804 (P.N.).

\appendix

\section{Fredholm pfaffians and point processes} \label{sec:pfaffian}

For an antisymmetric $2n \times 2n$ matrix $A,$ the \emph{pfaffian} of
$A$ is given by
\begin{equation}\label{defpfaff}
  \pf A =\frac{1}{2^{n}n!}\sum_{\sigma \in S_{2n}}\sgn(\sigma)A_{\sigma(1),\sigma(2)}A_{\sigma(3),\sigma(4)}\cdots A_{\sigma(2n-1),\sigma(2n)}
\end{equation}
and one can show that $(\pf A)^{2}=\det A$.

\paragraph{Matrix kernels and Fredholm pfaffians.}

Here we briefly introduce the notions needed for the text. For more
information, see \cite{rai} or \cite[Appendix B]{OQR16}.
A \emph{matrix kernel} on a space $X$  is a matrix-valued function on $X \times X$. It is said \emph{antisymmetric} if $K(x,y)=-K(y,x)^T$. We use $[K(x_i, x_j)]_{1 \leq i,j \leq n}$ to denote the block matrix whose block $(i,j)$ is $K(x_i,x_j)$. If the kernel $K$ is antisymmetric and of even size, then so is the matrix $[K(x_i, x_j)]_{1 \leq i,j \leq n}$ and we can compute its pfaffian. Note that this pfaffian is invariant under a permutation of $S=\{x_1,x_2,\dots,x_n\}$ and we can use $\pf K(S)$ to denote it.

\begin{prop}\label{ConjPfaff}
Let $K$, $\tilde{K}$ be $2 \times 2$ antisymmetric kernels on $X$ such that there exists a complex-valued function $f$ on $X$ such that
\begin{equation}
\tilde{K}(x,y):=
 \begin{pmatrix*}[l]
e^{f(x)}        &  0   \\
  0     &   e^{-f(x)} 
\end{pmatrix*} K(x,y) \begin{pmatrix*}[l]
e^{f(y)}        &  0   \\
  0      &   e^{-f(y)} 
\end{pmatrix*}.
\end{equation}
Then
\begin{equation}
    \pf [K(x_i, x_j)]_{1 \leq i,j \leq n} = \pf [\tilde{K}(x_i, x_j)]_{1 \leq i,j \leq n}.
\end{equation}
\end{prop}

 Let $(X,{m})$ be a measure space and $K$ a scalar kernel. The \emph{Fredholm determinant} is defined as
\begin{equation}\label{FredholmDeterminant}
\det(I+ K)_{X}:=\sum_{n=0}^{\infty}\frac{1}{n!}\int_{X^{n}}\dx m^{n}(x)\det [K(x_i, x_j)]_{1 \leq i,j \leq n}
\end{equation}  
where $I(x,y)=\delta_{x,y}$. For a $2 \times 2$  antisymmetric  matrix kernel $K$ the \emph{Fredholm pfaffian} is defined in analogous way by 
 \begin{equation}\label{FPF}
\pf(J+ K)_X=\sum_{n=0}^{\infty}\frac{1}{n!}\int_{X^{n}}\dx {m}^{n}(x) \pf [K(x_i, x_j)]_{1 \leq i,j \leq n}
\end{equation}    
where $J$ is the antisymmetric matrix kernel $
J(x,y)=\delta_{x,y} \begin{pmatrix*}[c]
0        &  1   \\
 -1      &   0 
\end{pmatrix*}.$

When $X$ is finite and $m$ is a counting measure then $\pf(J+K)_X$ coincides with the ordinary pfaffian of $J+K$. Identities involving ordinary pfaffians and determinants can usually be generalized to the Fredholm setting, e.g. $\pf \left( J+K  \right)_{X}^2=\det(I-JK)_X$ where $I=\delta_{x,y}\begin{pmatrix*}[c]
1        &  0   \\
 0      &   1 
\end{pmatrix*}.$

The Fredholm determinant \eqref{FredholmDeterminant} is absolutely convergent if we assume that  a scalar kernel $K$ satisfies $|K(x,y)|\leq D$ for all $(x,y) \in X \times X$ and $\int_X \dx m(x) =M < \infty$ since by Hadamard's bound   
$\lvert \det [K(x_i,x_j)]_{1 \leq i,j \leq n}\rvert \leq n^{n/2} D^n,
$ which gives us that 
\eqref{FredholmDeterminant} is dominated by the absolutely convergent series
\begin{equation} \label{absconser}
\sum _{n=0}^\infty\frac{n^{n/2}(DM)^n}{n!}.
\end{equation}

A little bit more generally, to check that the Fredholm pfaffian is dominated by the series \eqref{absconser} it is enough to have that the antisymmetric kernel $K$ satisfies  
\begin{equation}\label{AbsConvFP}
\lvert\pf [K(x_i, x_j)]_{1 \leq i,j \leq n}\rvert \leq n^{n/2} {D}^n  \prod_{i=1}^n f(x_i),
\end{equation}
for some measurable function $f$ such that $\int_X  \dx m(x) f(x)=M < \infty$.

\paragraph{Pfaffian point processes.}

We now briefly introduce point processes. For more detail see for example \cite{Joh05} or  \cite{bo}.  Let ${X}$ be a locally compact Polish space. A configuration is any subset of ${X}$ with no accumulation points. Let $
\mathit{Conf}({X})$ be the set of all configurations and let $\mathbb{P}$ be a probability measure on $\mathit{Conf}({X})$. $\mathbb{P}$ induces a Radon measure on $X^n$, assigning to a bounded Borel set $S$ the expected number of $n$-tuples of distinct points that fall in $S$. If this measure is absolutely continuous, then it has the Radon--Nikodym derivative $\rho_n$, known as the $n$-point correlation function, with respect to some reference measure $m$. For a bounded Borel set $B \subseteq X$, the gap probability $p_B$, which is the probability that a random configuration has no intersection with $B$, is given by
\begin{equation}
p_B =\mathbb{E} \left[\prod_i(1-\chi_B(x_i))\right]=\sum_{n=0}^{\infty}\frac{(-1)^{n}}{n!}\int_{B^{n}}\dx m^n(x)\rho_n(x_1,\dots,x_n).
\end{equation} 
 In the case of a discrete set ${X}$ endowed with the counting measure, the correlation function has a simple interpretation. For an $n$-tuple of distinct points $S=\{x_1,\dots,x_n\}$, the $n$-point correlation function $\rho_n(x_1,x_2,\dots,x_n)$
 is equal to $
 \mathbb{P}(C_S)$ where $C_S$  is the set of all configurations that contain $S$. For $B$ finite, the gap probability formula says
 \begin{equation}
 p_B= \sum_{S \subseteq B}(-1)^{\#S} \mathbb{P}(C_S),
 \end{equation}
 which is simply the inclusion--exclusion formula.
 
A point process is called \emph{pfaffian} if its $n$-point correlation function can be written in terms of a $2 \times 2$ antisymmetric matrix kernel $K$, called the correlation kernel, as
\begin{equation}
\rho_n(x_1,x_2,\dots,x_n)=\pf [K(x_i, x_j)]_{1 \leq i,j \leq n}.
\end{equation}
For a pfaffian point process, the gap probability for a bounded Borel set $B \subseteq X$ is given by the Fredholm pfaffian
\begin{equation}\label{gap}
p_B=\sum_{n=0}^{\infty}\frac{(-1)^{n}}{n!}\int_{B^{n}}\dx {m}^{n}(x) \pf [K(x_i, x_j)]_{1 \leq i,j \leq n}=\pf(J-K)_B.
\end{equation}

\begin{rem}
In Section~\ref{sec:slpp} we investigate the gap probabilities of $B=(s_1,\infty) \times \cdots \times (s_k,\infty) \subset (\Z')^k$ where $m$ is the counting measure. $B$ is not bounded in this case, but since it can be written as the union of bounded sets, \eqref{gap} holds by the dominated convergence theorem as long as we can claim the Fredholm pfaffian is absolutely convergent. To show it is absolutely convergent we use the lemma below which gives a sufficient condition for \eqref{AbsConvFP} to be satisfied. In this case we will use $f(x)=e^{-bx}$, $b>0$ for which $\int_B \dx m(x) e^{-bx}$ is indeed finite.  
\end{rem}

\begin{lem}\label{lemexp}
Let $K=\begin{pmatrix*}[l]
K_{1,1}       &  K_{1,2}   \\
 K_{2,1}    &  K_{2,2} 
\end{pmatrix*}$ be a $2\times 2$ antisymmetric kernel and $ c>d\geq 0$. Let $x, y \in \R$ and suppose there is a $C>0$ such that
\begin{equation}
|K_{1,1}(x, y)|\leq C e^{-c x-c y}, \quad |K_{1,2}(x, y)|\leq C e^{-c x+d y}, \quad |K_{2,2}(x,y)|\leq Ce^{dx+dy}.
\end{equation}
Then there exist $b>0$ and $D>0$ such that 
\begin{equation}
\lvert \pf [K(x_i, x_j)]_{1 \leq i,j \leq n}\rvert \leq n^{n/2}D^{n}e^{-b\sum_{i=1}^{n}x_i}.
\end{equation}
\end{lem}

 \begin{proof}  
Define    
\begin{equation}
\tilde{K}(x,x^{\prime})=
 \begin{pmatrix*}[l]
e^{cx}        &  0   \\
  0     &   e^{-dx} 
\end{pmatrix*} K(x,y) \begin{pmatrix*}[l]
e^{cy}        &  0   \\
  0      &   e^{-dy} 
\end{pmatrix*}.
\end{equation}
Let $S=\{x_1,\dots,x_n\}.$ Then $\lvert \pf K(S)\rvert=e^{(d-c)\sum_{i=1}^n{x_i}} \lvert \pf \tilde{K}(S)\vert \leq e^{(d-c)\sum_{i=1}^n{x_i}} \sqrt{(2n)^nC^{2n}}$ where the inequality comes from Hadamard's bound (since $\tilde{K}(x_i,x_j)<C$). We choose $D=C\sqrt{2}$ and $b=c-d$.
\end{proof}

\section{Pochhammer, theta and some elementary asymptotics} \label{sec:theta}

For $q$ a parameter, the $q$-Pochhammer symbol of length $n \in \N \cup \{ \infty \}$, the multiplicative, and the additive, theta functions are defined by:
\begin{equation}
 (x; q)_n := \prod_{i=0}^{n-1} (1-x q^i), \qquad  \theta_q (x) := (x; q)_{\infty} (q/x; q)_{\infty}, \qquad \theta_3 (z; q) := \sum_{n \in \Z} q^{\frac{n^2}{2}} z^n
\end{equation}
where $|q|<1$. For finite $n$ we have $(x; q)_n = (x; q)_{\infty}/(q^n x ; q)_{\infty}$\footnote{This should be taken as the definition for negative integer values of $n$.}. They are related by the Jacobi triple product identity~\cite{dol}:
\begin{equation} \label{eq:jtp}
 \theta_3(z; q) = (q;q)_{\infty} \prod_{i = \frac{1}{2}, \frac{3}{2}, \frac{5}{2},\dots} \left( 1 + q^i z \right) \left( 1+q^i/z \right) =  (q;q)_{\infty} \theta_q(-z \sqrt{q}).
\end{equation}
Finally note the quasi-periodicity relations:
\begin{equation}
    \theta_q(qx) = -\frac{1}{x} \theta_q(x), \qquad
    \theta_q\left(\frac{1}{x}\right) = -\frac{1}{x} \theta_q(x), \qquad
    \theta_q\left(\frac{x}{q}\right) = -\frac{x}{q} \theta_q(x).
\end{equation}
We use the convention that concatenation means product: $(a_1, a_2, \dots; q)_n := \prod_i (a_i; q)_n$ and similarly $\theta_q (a_1, a_2, \dots) := \prod_i \theta_q(a_i).$

We now turn to certain limits of Pochhammer symbols (and hence of theta functions). First we need the dilogarithm function $\di(z)$. It is defined by the power series representation
\begin{equation}
\di(z) = \sum_{n \geq 1} \frac{z^2}{n^2}, \qquad |z|<1
\end{equation}
with analytic continuation given by
\begin{equation}
\di (z)=-\int_{0}^{z}\frac{\log(1-u)}{u} \dx u,\quad z\in \mathbb{C} \backslash [1,\infty).
\end{equation}
Differentiating the dilogarithm, we obtain the usual logarithm $\left( z \frac{\dx }{\dx z} \right) \di(z) = -\log(1-z).$ If $q = \exp(-r)$ and $r \to 0+$, we have
\begin{equation}
  \label{eq:zqinfas}
  \log (z;q)_{\infty} \sim -\frac{\di(z)}{r}
\end{equation}
while
\begin{equation}
  \label{eq:zqnas}
  \log (z;q)_n \sim
  \begin{cases}
    \frac{1}{r}(\di(e^{-A} z)-\di(z)) & \text{if $n r \to A > 0$,} \\
    n \log (1-z) & \text{if $n r \to 0$.}
  \end{cases}
\end{equation}
Using this, one can prove that if $u(q)$ is a function that tends to $u \in (0,1)$ as $r \to 0+$, then for fixed $a$ and $b$
\begin{equation}
 \lim_{\substack{q = e^{-r},\\ r \to 0+}} \frac{(q^a u(q);q)_{\infty}}{(q^b u(q);q)_{\infty}} = (1-u)^{b-a}.
\end{equation}

\section{Fermionic propagators}
\label{sec:fermionic_expectations}

\begin{prop} \label{prop_inner_prod_right}
Using the notation of \eqref{def:free_boundary} and \eqref{def:variation_free_boundary}, we have:
\begin{align}
      \vcv \psi(z) \psi(w) \ket{\underline{v}} &= \frac{v^2 \sqrt{zw} (z-w)}{(z+v)(w+v)(zw-v^2)},\ \text{for} \  \left| \frac{v}{w} \right| < 1, \left| \frac{v}{z} \right| < 1,\notag \\ 
      \vcv \psi(z) \psi^*(w) \ket{\underline{v}} &= \frac{\sqrt{zw} (zw - v^2)}{(z-w)(w-v)(z+v)},\  \text{for} \  \left| \frac{v}{z} \right| < 1, \left| \frac{v}{w} \right| < 1, \left| \frac{w}{z} \right| < 1,\notag \\
      \vcv \psi^*(z) \psi(w) \ket{\underline{v}} &= \frac{\sqrt{zw} (zw - v^2)}{(z-w)(w+v)(z-v)},\  \text{for} \  \left| \frac{v}{z} \right| < 1, \left| \frac{v}{w} \right| < 1, \left| \frac{w}{z} \right| < 1,\notag \\
      \vcv \psi^*(z) \psi^*(w) \ket{\underline{v}} &= \frac{v^2 \sqrt{zw} (z-w)}{(z-v)(w-v)(zw-v^2)},\ \text{for} \  \left| \frac{v}{z} \right| < 1, \left| \frac{v}{w} \right| < 1,\notag \\
      \vcv \psi(z) \psi(w) \ket{\underline{v^{er}}} &= \frac{v^2 (z-w)}{\sqrt{zw} (zw-v^2)},\ \text{for} \  \left| \frac{v^2}{zw} \right|<1, \notag \\
      \vcv \psi(z) \psi^*(w) \ket{\underline{v^{er}}} &= \frac{w \sqrt{w} (zw-v^2)}{\sqrt{z} (z-w) (w^2-v^2)},\ \text{for} \  \left| \frac{w}{z} \right| < 1, \left| \frac{v}{w} \right| < 1,\notag  \\
      \vcv \psi^*(z) \psi(w) \ket{\underline{v^{er}}} &= \frac{z \sqrt{z} (zw-v^2)}{\sqrt{w} (z-w) (z^2-v^2)},\ \text{for} \  \left| \frac{w}{z} \right| < 1, \left| \frac{v}{z} \right| < 1,\notag  \\
      \vcv \psi^*(z) \psi^*(w) \ket{\underline{v^{er}}} &= \frac{v^2 zw \sqrt{zw} (z-w)}{(z^2-v^2) (w^2-v^2) (zw-v^2)},\ \text{for} \  \left| \frac{v}{z} \right| < 1, \left| \frac{v}{w} \right| < 1,\notag  \\
      \vcv \psi(z) \psi(w) \ket{\underline{v^{ec}}} &= \frac{v^2 zw \sqrt{zw} (z-w)}{(z^2-v^2) (w^2-v^2) (zw-v^2)},\ \text{for} \  \left| \frac{v}{z} \right| < 1, \left| \frac{v}{w} \right| < 1,\notag \\
      \vcv \psi(z) \psi^*(w) \ket{\underline{v^{ec}}} &= \frac{z \sqrt{z} (zw-v^2)}{\sqrt{w} (z-w) (z^2-v^2)},\ \text{for} \  \left| \frac{w}{z} \right| < 1, \left| \frac{v}{z} \right| < 1, \\
      \vcv \psi^*(z) \psi(w) \ket{\underline{v^{ec}}} &= \frac{w \sqrt{w} (zw-v^2)}{\sqrt{z} (z-w) (w^2-v^2)},\ \text{for} \  \left| \frac{w}{z} \right| < 1, \left| \frac{v}{w} \right| < 1,\notag  \\
      \vcv \psi^*(z) \psi^*(w) \ket{\underline{v^{ec}}} &= \frac{v^2 (z-w)}{\sqrt{zw} (zw-v^2)},\ \text{for} \  \left| \frac{v^2}{zw}\right| < 1,\notag  \\
      \vcv \psi(z) \psi(w) \ket{\underline{v^{\beta, \ or}}} &= \frac{v^2 \sqrt{zw} (z-w) }{ (zw-v^2) (z+\beta) (w+\beta)},\ \text{for} \ \left| \frac{v^2}{zw} \right|<1, \left| \frac{\beta}{z} \right|<1, \left| \frac{\beta}{w} \right|<1,\notag \\
      \vcv \psi(z) \psi^*(w) \ket{\underline{v^{\beta, \ or}}} &= \frac{\sqrt{zw} (zw-v^2) (w + \beta)}{(z-w) (w^2-v^2) (z+\beta)},\ \text{for} \  \left| \frac{w}{z} \right| < 1, \left| \frac{v}{w} \right| < 1, \left| \frac{\beta}{z} \right| < 1,\notag  \\
      \vcv \psi^*(z) \psi(w) \ket{\underline{v^{\beta, \ or}}} &= \frac{\sqrt{zw} (zw-v^2) (z+\beta)}{(z-w) (z^2-v^2) (w+\beta)},\ \text{for} \  \left| \frac{w}{z} \right| < 1, \left| \frac{v}{z} \right| < 1, \left| \frac{\beta}{w} \right| < 1,\notag  \\
      \vcv \psi^*(z) \psi^*(w) \ket{\underline{v^{\beta, \ or}}} &= \frac{v^2 \sqrt{zw} (z-w) (z+\beta) (w+\beta)}{(z^2-v^2) (w^2-v^2) (zw-v^2)},\ \text{for} \  \left| \frac{v}{z} \right| < 1, \left| \frac{v}{w} \right| < 1,\notag  \\
      \vcv \psi(z) \psi(w) \ket{\underline{v^{\alpha, \ oc}}} &= \frac{v^2 \sqrt{zw} (z-w) (z-\alpha) (w-\alpha)}{(z^2-v^2) (w^2-v^2) (zw-v^2)},\ \text{for} \  \left| \frac{v}{z} \right| < 1, \left| \frac{v}{w} \right| < 1,\notag \\
      \vcv \psi(z) \psi^*(w) \ket{\underline{v^{\alpha, \ oc}}} &= \frac{\sqrt{zw} (zw-v^2) (z-\alpha)}{(z-w) (z^2-v^2) (w-\alpha)},\ \text{for} \  \left| \frac{w}{z} \right| < 1, \left| \frac{v}{z} \right| < 1, \left| \frac{\alpha}{w} \right| < 1,\notag  \\
      \vcv \psi^*(z) \psi(w) \ket{\underline{v^{\alpha, \ oc}}} &= \frac{\sqrt{zw} (zw-v^2) (w-\alpha)}{(z-w) (w^2-v^2) (z-\alpha)},\ \text{for} \  \left| \frac{w}{z} \right| < 1, \left| \frac{v}{w} \right| < 1, \left| \frac{\alpha}{z} \right| < 1,\notag  \\
      \vcv \psi^*(z) \psi^*(w) \ket{\underline{v^{\alpha, \ oc}}} &= \frac{v^2 \sqrt{zw} (z-w)}{ (zw-v^2) (z-\alpha) (w-\alpha)},\ \text{for} \  \left| \frac{v^2}{zw}\right| < 1, \left| \frac{\alpha}{z} \right| < 1, \left| \frac{\alpha}{w} \right| < 1.\notag 
\end{align}
\end{prop}

\begin{proof}
  Ping-pong (one-sided) and the boson--fermion correspondence~\eqref{eq:boson_fermion}. Equivalently, the argument proving~\eqref{eq:eval_prop_1fb} applies throughout.
\end{proof}

We now turn to the (unmodified) two free boundary propagators. Recall definition~\eqref{eq:utfcvdef}. 
\begin{prop} \label{prop_inner_prod_left_right} We have:
\begin{align}
      \utfcv \psi(z) \psi(w) \vtfv &= \frac{((uv)^2; (uv)^2)^2_{\infty} \theta_{(uv)^2}(\frac{w}{z})}  {(- \frac{v}{z}, - \frac{v}{w}, uz, uw; uv)_{\infty} \theta_{(uv)^2} (u^2 zw) } \cdot   \frac{\theta_3 \left( ( \frac{t zw}{v^2})^2 ; (uv)^4 \right)}{(uv; uv)_{\infty}} \cdot \frac{v^2 }{t w \sqrt{zw}}, \notag \\
      & for \left| \frac{v}{z} \right| < 1, \left| \frac{v}{w} \right| < 1, |uz| < 1,|uw| < 1, |uv|<1, \notag \\
      \utfcv \psi(z) \psi^*(w) \vtfv &= \frac{((uv)^2; (uv)^2)^2_{\infty} \theta_{(uv)^2}(u^2 zw)}  {(- \frac{v}{z}, \frac{v}{w}, uz, -uw; uv)_{\infty} \theta_{(uv)^2} (\frac{w}{z}) } \cdot \frac{\theta_3 \left( ( \frac{t z}{w})^2 ; (uv)^4 \right)}{(uv; uv)_{\infty}} \cdot \sqrt{\frac{w}{z}}, \notag \\
      & for \left| \frac{v}{z} \right| < 1, \left| \frac{v}{w} \right| < 1, |uz| < 1,|uw| < 1, \left| \frac{w}{z} \right| < 1, |uv|<1, \\
      \utfcv \psi^*(z) \psi(w) \vtfv &= \frac{((uv)^2; (uv)^2)^2_{\infty} \theta_{(uv)^2}(u^2 zw)}  {( \frac{v}{z}, - \frac{v}{w}, - uz, uw; uv)_{\infty} \theta_{(uv)^2} (\frac{w}{z}) } \cdot \frac{\theta_3 \left( ( \frac{t w}{z})^2 ; (uv)^4 \right)}{(uv; uv)_{\infty}} \cdot \sqrt{\frac{z}{w}}, \notag \\
      & for \left| \frac{v}{z} \right| < 1, \left| \frac{v}{w} \right| < 1, |uz| < 1,|uw| < 1, \left| \frac{w}{z} \right| < 1, |uv|<1, \notag \\
      \utfcv \psi^*(z) \psi^*(w) \vtfv &= \frac{((uv)^2; (uv)^2)^2_{\infty} \theta_{(uv)^2}(\frac{w}{z})}  {(\frac{v}{z}, \frac{v}{w}, -uz, -uw; uv)_{\infty} \theta_{(uv)^2} (u^2 zw) } \cdot \frac{\theta_3 \left( ( \frac{t v^2}{zw})^2 ; (uv)^4 \right)}{(uv; uv)_{\infty}} \cdot \frac{v^2 t}{w \sqrt{zw}}, \notag \\
      & for \left| \frac{v}{z} \right| < 1, \left| \frac{v}{w} \right| < 1, |uz| < 1,|uw| < 1, |uv|<1.\notag 
\end{align}
\end{prop}

\begin{proof}
  Ping-pong and the boson--fermion correspondence~\eqref{eq:boson_fermion}. All are similar to the proof of~\eqref{eq:eval_prop_2fb}.
\end{proof}

\bibliographystyle{myhalpha}
\bibliography{freeboundaries}

\end{document}